\date{} 
\title{The Fourier Spectrum of Critical Percolation}
\author{Christophe Garban \and G\'abor Pete \and Oded Schramm}
\documentclass[12pt]{article}
\usepackage{amsmath}
\usepackage{amssymb}
\usepackage{amsthm}
\usepackage{amsfonts}
\usepackage{graphicx}
\newif\ifhyper\IfFileExists{hyperref.sty}{\hypertrue}{\hyperfalse}
\ifhyper\usepackage{hyperref}  
\input epsf.sty
\input labelfig.tex

\newif\iffigures\figurestrue
\figuresfalse

\def\hitem#1#2{\item[\hypertarget{#1}{#2.}]\expandafter\gdef\csname LBL#1ITM\endcsname{#2}}
\def\iref#1{\hyperlink{#1}{\csname LBL#1ITM\endcsname}}
\else
\def\hitem#1#2{\item[{#2.}]\expandafter\gdef\csname LBL#1ITM\endcsname{#2}}
\def\iref#1{{\csname LBL#1ITM\endcsname}}
\fi

\newif\ifdraft
\drafttrue \long\def\note#1/{\ifdraft{\marginpar{{$\Longleftarrow$}} \bf [#1] }\fi}
\long\def\comment#1{}
\long\def\old#1{}

\numberwithin{equation}{section}
\numberwithin{figure}{section}

\newtheorem{theorem}{Theorem}
\numberwithin{theorem}{section}
\newtheorem{corollary}[theorem]{Corollary}
\newtheorem{lemma}[theorem]{Lemma}
\newtheorem{proposition}[theorem]{Proposition}
\newtheorem{conjecture}[theorem]{Conjecture}

\theoremstyle{remark}\newtheorem{definition}[theorem]{Definition}
\theoremstyle{remark}\newtheorem{remark}[theorem]{Remark}
\def\eref#1{(\ref{#1})}
\let\qqed=\qed
\def\QED{\qqed\bigskip}
\let\qed=\QED
\newcommand{\Prob} {{\mathbb P}}
\newcommand{\R}{\mathbb{R}}
\newcommand{\Q}{\mathbb{Q}}
\newcommand{\C}{\mathbb{C}}
\newcommand{\Z}{\mathbb{Z}}
\newcommand{\N}{\mathbb{N}}
\def\Quad{\mathcal Q}
\def\Piv{\mathscr{P}}

\def\diam{\mathop{\mathrm{diam}}}

\def\SLEkk#1/{$\mathrm{SLE}(#1)$}
\def\SLEr#1/{$\mathrm{SLE(\kappa;#1)}$}
\def\SLEkr#1;#2/{$\mathrm{SLE(#1;#2)}$}
\def\SLEk/{\SLEkk{\kappa}/}
\def\SLEtwo/{\SLEkk2/}
\def\SLE/{$\mathrm{SLE}$}
\def\SLEab/{\SLEkr 4; {a/\hco-1}, {b/\hco-1}/}

\def\Ito/{It\^o}
\def \eps {\epsilon}
\def \P {\Prob}
\def\md{\mid}
\def\Bb#1#2{{\def\md{\bigm| }#1\bigl[#2\bigr]}}
\def\BB#1#2{{\def\md{\Bigm| }#1\Bigl[#2\Bigr]}}
\def\Bs#1#2{{\def\md{\mid}#1[#2]}}
\def\Pb{\Bb\P}
\def\Eb{\Bb\E}
\def\PB{\BB\P}
\def\EB{\BB\E}
\def\Ps{\Bs\P}
\def\Es{\Bs\E}

\def \p {{\partial}}
\def \E {{\mathbb E}}

\def\closure{\overline}
\def\ev#1{{\mathcal{#1}}}
\def \proof {{ \medbreak \noindent {\bf Proof.} }}
\def\proofof#1{{ \medbreak \noindent {\bf Proof of #1.} }}

\def\bl{\bigl}

\def\nn{R}
\def\Ent{\mathrm{Ent}}

\def\SmirnovICM{MR2275653} 
\def\SmirnovIsing1{arXiv:0708.0039}
\def\DPSL{DPSL}
\def\PivotalMeasure{arXiv:1008.1378}
\def\LMN{MR1370363} 
\def\KKL{21923}
\def\SchrammSteif{math.PR/0504586}
\def\SchrammSmirnovNoise{SSblacknoise}
\def\BKSnoise{MR2001m:60016}
\def\KestenScaling{MR88k:60174}

\def\SmirnovPerc{MR1851632}

\def\AizenmanDuplantierAharony{ADA} 

\def\LSWii{MR2002m:60159b}

\def\SmirnovWerner{MR1879816}

\def\LSWoneArm{MR2002k:60204}

\def\KestenSZh{MR1637089}
\def\Reimer{MR2001g:60017}

\def\BMbook{BMbook}
\def\WWperc{arXiv:0710.0856}

\def\HaggstromPeresSteif{MR1465800}
\def\BHPS{MR1959784}
\def\SchrammICM{MR2334202}
\def\TsirelsonStFlour{MR2079671}
\def\FriedgutKalai{MR1371123}
\def\BromanSteif{MR2223951}
\def\JonaSteifCircle{MR2393996}
\def\PeresSchSteif{MR2521411}
\def\DynBoolean{MR1472953}
\def\BSexcPlanes{MR1641830}
\def\PeresSteif{MR1626782}
\def\DavarTree{MR2357674}
\def\DaverEtAl{MR2226886}
\def\HaggstromPemantle{MR1716767}
\def\HoffmanRecu{MR2235173}
\def\NolinKesten{MR2438816}
\def\Mattila{MR1333890}
\def\Grimm{Grimmett:newbook}
\def\HPSIIC{HPS:IIC}
\def\KestenIIC{MR88c:60196}
\def\BernsteinVazirani{MR1471988}
\def\LSSdomination{MR1428500}

\def\noopsort#1{}

\def\bl{\begin{lemma}}
\def\el{\end{lemma}}
\def\bth{\begin{theorem}}
\def\eth{\end{theorem}}
\def\bc{\begin{corollary}}
\def\ec{\end{corollary}}
\def\bcj{\begin{conjecture}}
\def\ecj{\end{conjecture}}
\def\bpr{\begin{proposition}}
\def\epr{\end{proposition}}
\def\bde{\begin{definition}}
\def\ede{\end{definition}}
\newcommand{\be}{\begin{eqnarray}}
\newcommand{\ee}{\end{eqnarray}}
\newcommand{\bes}{\begin{eqnarray*}}
\newcommand{\ees}{\end{eqnarray*}}

\usepackage{mathrsfs}

\def\Spec{\mathscr{S}}
\def\Ann{\mathcal{A}}
\def\ANN{\mathfrak{A}}
\def\ANNp{\ANN'}
\def\ANNpp{\ANN''}

\def\bgamma{\bar\gamma}
\def\gammas{\gamma^*}
\def\bgammas{\bgamma^*}
\def\F{\mathcal F}
\def\Qb{\Bb{\Q}}
\def\1{1}
\def\bits{\I}
\def\I{\mathcal{I}}
\def\b{\beta}
\def\llwb{\lambda_{B,W}}
\def\lala(#1,#2){\lambda_{#1,#2}}

\def\NN{\mathcal{N}}

\def\coa{\vartheta}
\def\cod{\vartheta^*}
\def\Rs{\mathcal{Z}}
\def\EX{\mathscr{E}}
\def\Ess_#1{\E |\Spec_{f_{#1}}|}
\def\determ{\mathcal{U}}
\def\determc{\determ^c}

\def\AA{\mathscr A}

\def\window{\theta}
\def\concentration{tightness }

\def\concentrated{tight}

\begin{document}
\maketitle

\renewcommand{\thefootnote}{}
\footnotetext{\hspace{-6mm}Microsoft Research (CG, GP and OS), Univ.\ Paris-Sud and ENS (CG), and Univ.\ of Toronto (GP).
CG was partially supported by the ANR under the grant ANR-06-BLAN-0058.
GP was partially supported by the Hungarian National Foundation for Scientific Research, grant T049398, and by an NSERC Discovery Grant.} 
\renewcommand{\thefootnote}{\arabic{footnote}}

\begin{abstract}
Consider the indicator function $f$ of a two-dimensional percolation crossing event.
In this paper, the Fourier transform of $f$ is studied and sharp bounds are
obtained for its lower tail in several situations.
Various applications of these bounds are derived. In particular, we show
that the set of exceptional times
of dynamical critical site percolation
on the triangular grid in which the origin percolates has dimension $31/36$ a.s.,
and the corresponding dimension in the half-plane is $5/9$.
It is also proved that critical bond percolation on the square grid has
exceptional times a.s.
Also, the asymptotics of the number of sites that need to be resampled in order
to significantly perturb the global percolation configuration in a large square
is determined. 
\end{abstract}
\newpage
\tableofcontents
\newpage

\section{Introduction} \label{s.intro}

\subsection{Some general background}\label{ss.genback}
The Fourier expansion of functions on $\R^d$ is an indispensable tool with
numerous applications.
Likewise, the harmonic analysis of functions
defined on the discrete cube $\{-1,1\}^d$ has found a host
of applications; see the survey~\cite{MR2208732}.
Yet the Fourier expansion of some functions of interest
is rather poorly understood.
Here, we study the harmonic analysis of functions arising from 
planar percolation and answer most if not all of the previously
posed problems regarding their Fourier expansion.
We also derive some applications to the behavior of percolation
under noise and in the study of dynamical percolation.
It is hoped that some of the techniques introduced here will be helpful
in the harmonic analysis of other functions.

Let $\I$ be some finite set, and let $\Omega=\{-1,1\}^\I$ be endowed with the uniform
measure.
The Fourier basis on $\Omega$ consists of all the functions of the form
$\chi_S(\omega):= \prod_{i\in S}\omega_i$, where $S\subseteq\I$.
(These functions are also sometimes called the Walsh functions.)
It is easily seen to be an orthonormal basis with respect to the inner product $\Eb{f\,g}$.
Therefore, for every $f:\Omega\to \R$, we have
\begin{equation}\label{e.fourier}
f=\sum_{S\subseteq\I}\widehat f(S)\,\chi_S\,,
\end{equation}
where $\widehat f(S):=\Es{f\,\chi_S}$.
If $\Es{f^2}=1$, then the random variable $\Spec=\Spec_f\subseteq\I$ with distribution given by
$$
\Pb{\Spec =S}=\widehat f(S)^2 
$$
will be called the {\bf Fourier spectral sample} of $f$.
Due to Parseval's formula, this is indeed a probability distribution.
The idea to look at this as a probability distribution was proposed in~\cite{\BKSnoise},
though the study of the weights $\widehat f(S)^2$ is \lq\lq ancient\rq\rq, and boils down to the
same questions in a different lingo.
As noted there, important properties of the function $f$ are encoded in the law of the
spectral sample. For example, suppose that
$f:\{-1,1\}^\I\to\{-1,1\}$. Let $x\in\{-1,1\}^\I$ be random and uniform,
and let $y$ be obtained from
$x$ by resampling\footnote{In the definition of~\cite{\BKSnoise},
each bit is flipped with probability $\eps$, rather than resampled.
This accounts for some discrepancies involving factors of $2$.
The present formulation generally produces simpler formulas.}
 each coordinate with probability
$\epsilon$ independently, where $\eps\in(0,1)$.
Then $y$ is referred to as an $\eps$-noise of $x$.
Since for $i\in\I$ we have $\Eb{x_i\,y_i}=1-\eps$ it follows that
$\Eb{\chi_S(x)\,\chi_S(y)}=(1-\eps)^{|S|}$ for $S\subseteq\I$
and hence it easily follows by using the Fourier expansion~\eqref{e.fourier} that
\begin{equation}
\label{e.fcor}
\Eb{f(x)\,f(y)}= \Eb{(1-\eps)^{|\Spec_f|}}.
\end{equation}
Thus, the stability or sensitivity of $f$ to noise is encoded in the law of $|\Spec|$.

One mathematical model in which noise comes up is that of the Poisson dynamics on $\Omega$,
in which each coordinate is resampled according to a Poisson process of rate $1$,
independently. 
This is, of course, just the continuous time random walk on $\Omega$.
If $x_t$ denotes this continuous time Markov process started at the stationary (uniform) measure
on $\Omega$, then $x_t$ is just $\eps$-noise of $x_0$, where $\eps=1-e^{-t}$.
Indeed, the Markov operator defined by
$$
T_tf(x)=\Eb{f(x_t)\md x_0=x}
$$
is diagonalized by the Fourier basis:
$$
T_t \chi_S= e^{-t|S|}\chi_S\,.
$$
It is therefore hardly surprising that the behavior of $|\Spec_f|$ will play an important
role in the study of the generally non-Markov process $f(x_t)$.
These types of questions have been under investigation in the context of Bernoulli percolation~\cite{\HaggstromPeresSteif}, \cite{\PeresSteif}, \cite{\HaggstromPemantle},  \cite{\PeresSchSteif}, \cite{\DavarTree}, other percolation type processes \cite{\DynBoolean}, \cite{\BSexcPlanes}, \cite{\BromanSteif},
and also more generally~\cite{\BHPS}, \cite{\JonaSteifCircle}, \cite{\HoffmanRecu}, \cite{\DaverEtAl}.
Estimates of the Fourier coefficients played an important role in the proof that the
dynamical version of critical site percolation on the triangular grid a.s.\ has percolation times~\cite{\SchrammSteif}.
These estimates can naturally be phrased in terms of properties of the random variable $|\Spec|$.

Recall that the (random) set of {\bf pivotals} of $f:\{-1,1\}^\I\to\{-1,1\}$ is the set of $i\in\I$
such that flipping the value of $\omega_i$ also changes the value of $f(\omega)$.
It is easy to see~\cite{\KKL} that the first moment of the number of pivotals of $f$
is the same as the first moment of $|\Spec_f|$.
Gil Kalai (personal communication) observed that the same is true for the second moment,
but not for the higher moments.
(We will recall the easy proof of this fact in Section~\ref{ss.ssg}.)
This often facilitates an easy estimation of $\Eb{|\Spec_f|}$ and $\Eb{|\Spec_f|^2}$.

It is often the case that $\Eb{|\Spec|^2}$ is of the same order of magnitude as $\Eb{|\Spec|}^2$,
and this implies that with probability bounded away from zero, the random variable $|\Spec|$
is of the same order of magnitude as its mean.
However, what turns out to be much harder to estimate is the probability that
$|\Spec|$ is positive and much smaller than its mean. (In particular, this is 
much harder than the analogous \concentration result for pivotals, see Remark~\ref{r.pivotals} below.)
This is very relevant to applications; as can be seen from~\eqref{e.fcor}, the
probability that $|\Spec|$ is small is what matters most in understanding the correlation
of $f(x)$ with $f$ evaluated on a noisy version of $x$. Likewise, in the dynamical
setting, the lower tail of $|\Spec|$ controls the switching rate of $f$.
Indeed, the primary purpose of this paper is getting good estimates on
$\Pb{0<|\Spec_f|<s}$ for indicators of crossing events in percolation
and deriving the consequences of such bounds.

\subsection{The main result}\label{ss.main}

We consider two percolation models: critical bond percolation on the square grid $\Z^2$ and
critical site percolation on the triangular grid. See \cite{\Grimm,\WWperc} for background. 
These two models are believed to behave
essentially the same, but the mathematical understanding of the latter is significantly superior to
the former due to Smirnov's theorem~\cite{\SmirnovPerc} and its consequences.
Fix some large $R>0$, and consider the event that (in either of these percolation models)
there is an open (i.e., occupied) left-right crossing of the square $[0,R]^2$.
Let $f_R$ denote the $\pm1$ indicator function of this event; that is, $f_R=1$ if there is
a crossing and $f_R=-1$ when there is no crossing.
In this case, $\I$ is the relevant set of bonds or sites, depending on whether we are in the
bond or site model. The probability space is $\Omega=\{-1,1\}^\I$ with the uniform measure.
Here, it is convenient that $p_c=1/2$ for both models, and so the relevant measure
on $\Omega$ is uniform.

The paper~\cite{\BKSnoise}
posed the problem of studying the law of $\Spec_{f_R}$.
There, it was proved that $\Pb{0<|\Spec_{f_R}|<c\,\log R}\to 0$ as $R\to\infty$
for some $c>0$. This had the implication that
$f_R$ is asymptotically noise sensitive;
that is, $\lim_{R\to\infty} \Eb{f_R(x)\,f_R(y)}-\Eb{f_R(x)}\Eb{f_R(y)}=0$,
when $y$ is an $\eps$-noisy version of $x$ and $\eps>0$ is fixed.
It was also asked in~\cite{\BKSnoise} whether $\Pb{0<|\Spec|<R^\delta}\to 0$ for some $\delta>0$.
This was later proved in~\cite{\SchrammSteif} with $\delta=1/8+o(1)$ in the setting
of the triangular lattice and with an unspecified $\delta>0$ for the square lattice.
While certainly a useful step forward, these results were far from being sharp.
But the issue is more than just a quantitative question. The most natural
hypothesis is that $|\Spec|$ is always proportional to its mean when it is nonzero, or more
precisely that
\begin{equation}
\label{e.tight}
\sup_{R>1} \PB{0<|\Spec_{f_R}|<t\,\E|\Spec_{f_R}|} \to 0\qquad\text{as }t\searrow 0\,.
\end{equation}
To illustrate the fact that~\eqref{e.tight} is not a universal principle,
we note that it does not hold, e.g., for the 
$\pm 1$-indicator function of the event that
there is a left-right percolation in the square $[0,R]^2$ and this square has more
open sites [or edges] than closed sites [or edges].
We prove~\eqref{e.tight} in the present paper, and give useful bounds that are sharp up to constants
on the left hand side of~\eqref{e.tight}.
This is the content of our first theorem.

\begin{theorem}\label{t.1}
As above, let $R>1$ and let $f_R$ be the $\pm1$ indicator function of the left-right
crossing event of the square $[0,R]^2$ in critical bond percolation
on $\Z^2$ or site percolation on the triangular grid.
The spectral sample of $f_R$ satisfies
\begin{equation}
\label{e.ssbounds}
\PB{0<|\Spec_{f_R}|<\Ess_r} \asymp \Bigl(\frac{\Ess_R/R}{\Ess_r/r}\Bigr)^2
\end{equation}
holds for every $r\in[1,R]$,
and $\asymp$ denotes equivalence up to positive multiplicative constants.
\end{theorem}

To make the sharp bound~\eqref{e.ssbounds} more explicit, one needs to discuss
estimates for $\Ess_r$. 
The value of $\Ess_r$ is estimated by the probability of the so called \lq\lq alternating
$4$-arm event\rq\rq, which we will treat in detail in Subsection~\ref{ss.multi}.
For now, let us just mention that it is known that 
\begin{equation}
\label{e.oneminusdelta}
\Ess_R/\Ess_r \le C\,(R/r)^{1-\delta}
\end{equation}
for any $R\geq r \geq 1$ with some $\delta>0$ and $C<\infty$, and that, for the triangular lattice,
\begin{equation}
\label{e.3d4}
\Ess_R \,/\, \Ess_r \asymp (R/r)^{3/4+o(1)}
\end{equation}
as $R/r\to\infty$ while $r\ge 1$, which
follows from Smirnov's theorem~\cite{\SmirnovPerc} and the SLE-based
analysis of the percolation exponents in~\cite{\SmirnovWerner}.
(This will be proved in Section~\ref{ss.lsq}.)
From~\eqref{e.3d4} and~\eqref{e.ssbounds},
we get for the triangular grid 
\begin{equation}
\label{e.lam} 
\Pb{0<|\Spec_{f_R}|\le\lambda\,\Ess_R} \asymp \lambda^{2/3+o(1)},
\end{equation}
 where $\lambda$ may depend on $R$, but is restricted
to the range $\bigl[(\Ess_R)^{-1},1\bigr]$.
Here, the $o(1)$ represents a function of $\lambda$ and $R$ that
tends to $0$ as $\lambda\to 0$, uniformly in $R$.
This answers Problem 5.1 from~\cite{\SchrammICM}.
Even for the square lattice,~\eqref{e.tight} follows from Theorem~\ref{t.1} combined with (\ref{e.oneminusdelta}).
Below, we prove~\eqref{e.lam2}, which is a variant of~\eqref{e.lam} with slightly different
asymptotics.

There is nothing particularly special about the square with regard to Theorem~\ref{t.1}.
The proof applies to every rectangle of a fixed shape (with the implied constants
depending on the shape).
For percolation crossings in more general shapes, Theorem~\ref{t.loc} gives bounds
on the behavior of $\Spec$ away from the boundary, and we also prove
Theorem~\ref{t.quadtight}, which
is some analog of~\eqref{e.tight}.
However, we chose not to go into the complications that would arise when trying to
prove~\eqref{e.ssbounds} in this general context. The issue is that $\Spec$ could possibly be of larger size near ``peaks'' of the boundary of $\Quad$ (especially if it is fractal-like), so the boundary contributions might dominate $\E|{\Spec}|$ and might also have a complicated influence on the tail behavior. Still, as it turns out, $\Spec$ is unlikely to be close to the boundary, so we can ignore these effects to a certain extent and prove Theorem~\ref{t.quadtight}.

We also remark that $\Pb{\Spec_f=\emptyset}=\Es{f}^2=\widehat f(\emptyset)^2$
is generally easy to compute.
The level $0$ Fourier coefficient $\widehat f(\emptyset)$ has more to do with the
way the function is normalized than with its fundamental properties.
For this reason, $|\Spec|=0$ is separated out in bounds such as~\eqref{e.ssbounds}.

At this point we mention that Theorem~\ref{t.radial} gives the bound
analogous to Theorem~\ref{t.1}, but dealing with
the spectrum of the indicator function for a percolation crossing
from the origin to distance $R$ away.

\subsection{Applications to noise sensitivity}\label{ss.intronoise}

\begin{figure}[htbp]
\SetLabels
(.495*.83)$\rightarrow$\\
(.76*.665)$\downarrow$\\
(.26*.665)$\uparrow$\\
(.76*.32)$\downarrow$\\
(.26*.32)$\uparrow$\\
(.495*.15)$\leftarrow$\\
\endSetLabels
\centerline{
\AffixLabels{%
\vbox{
\hbox{
\epsfysize=2in \epsffile{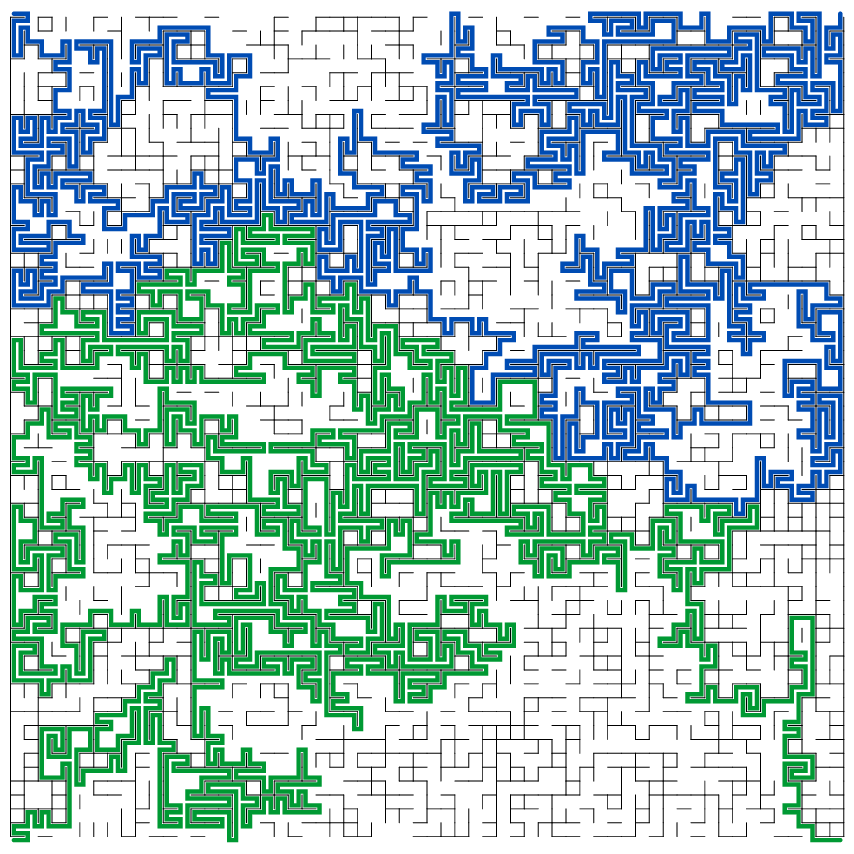}\qquad
\epsfysize=2in \epsffile{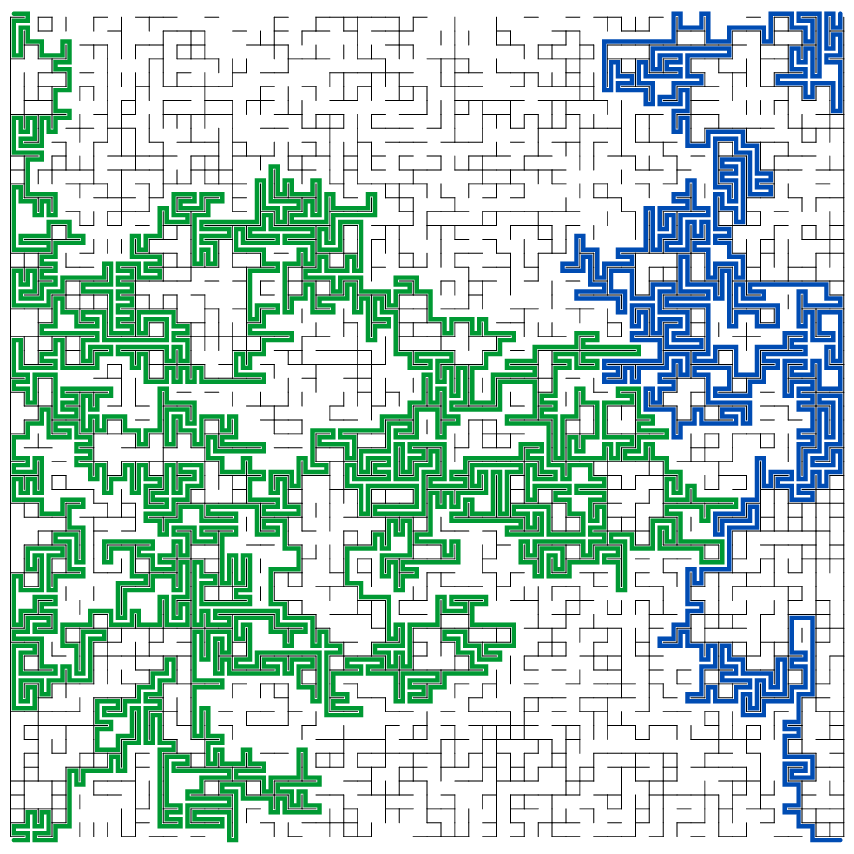}
}
\vskip0.2in
\hbox{
\epsfysize=2in \epsffile{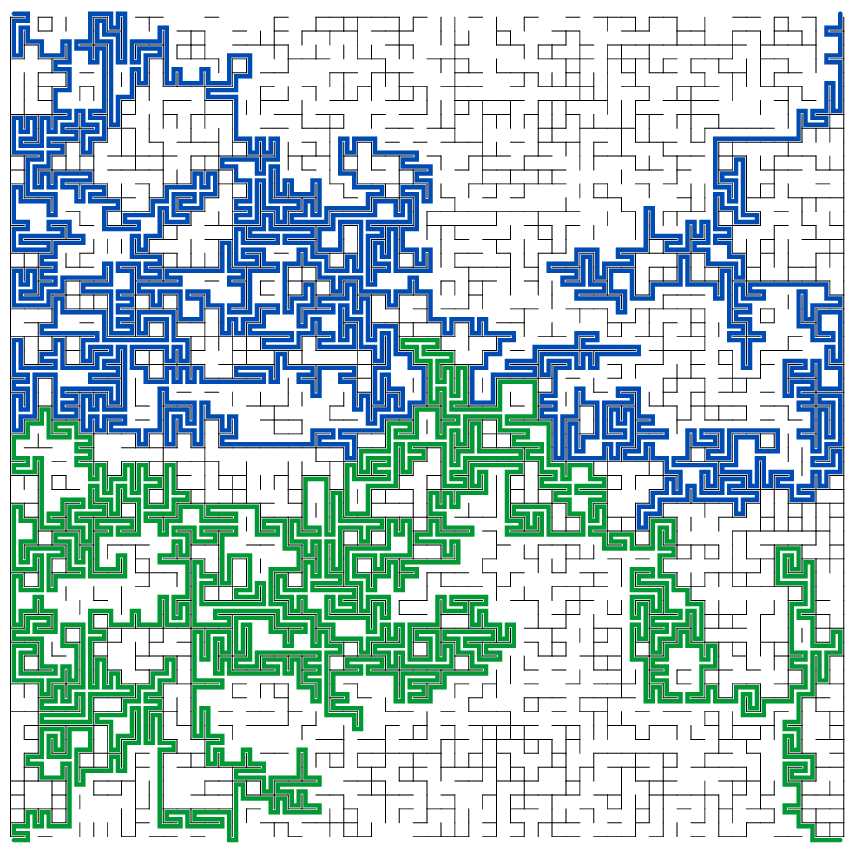}\qquad
\epsfysize=2in \epsffile{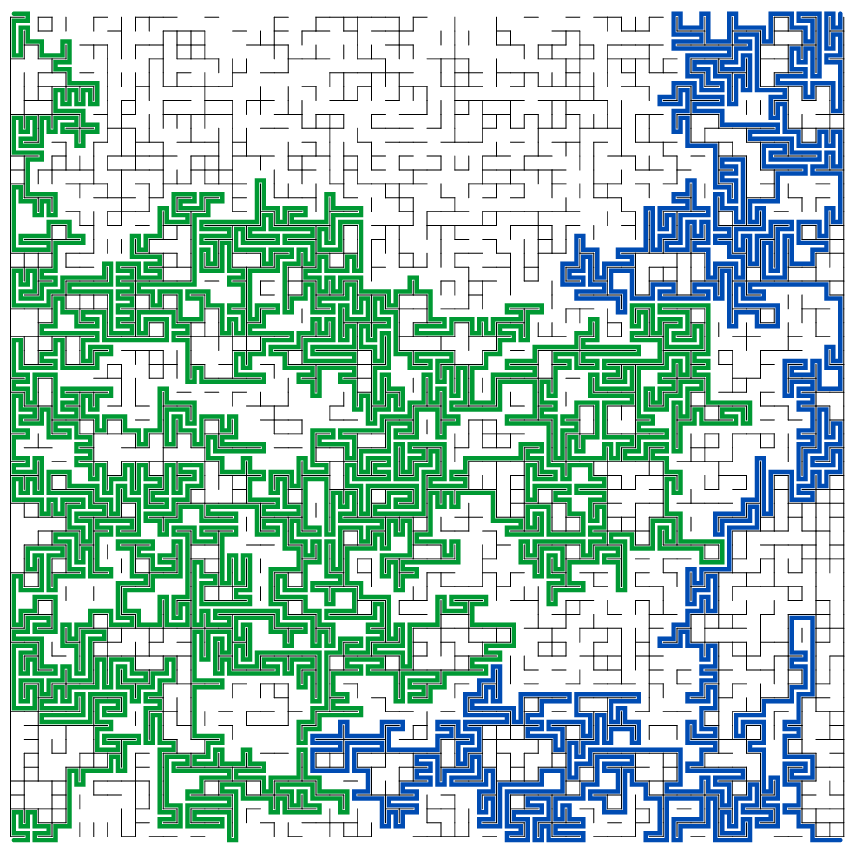}
}
\vskip0.2in
\hbox{
\epsfysize=2in \epsffile{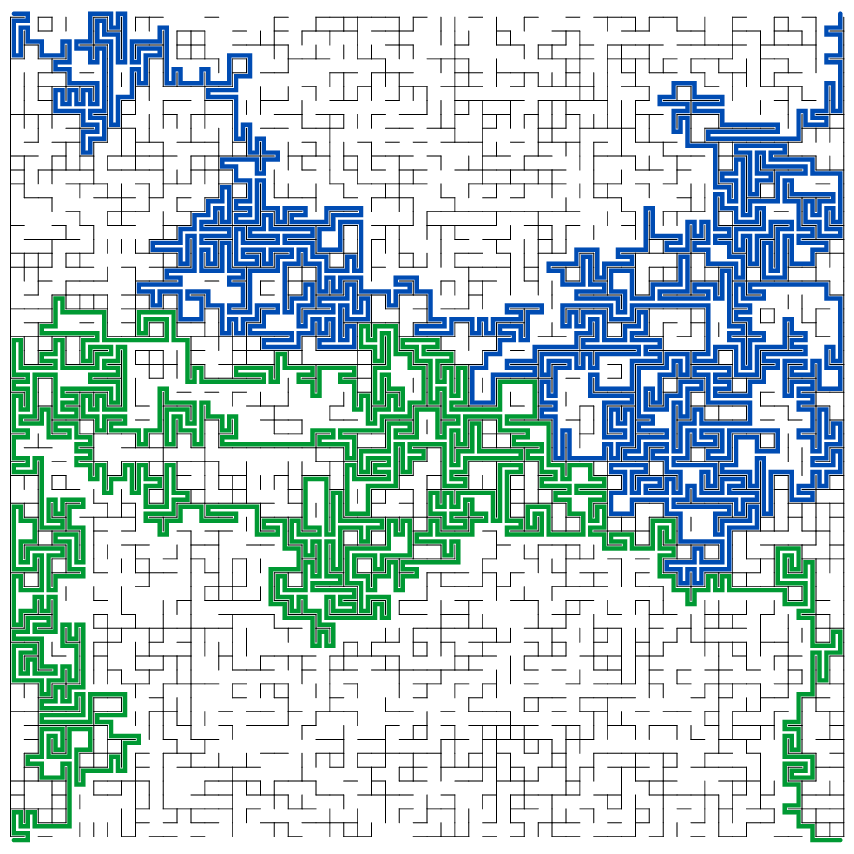}\qquad
\epsfysize=2in \epsffile{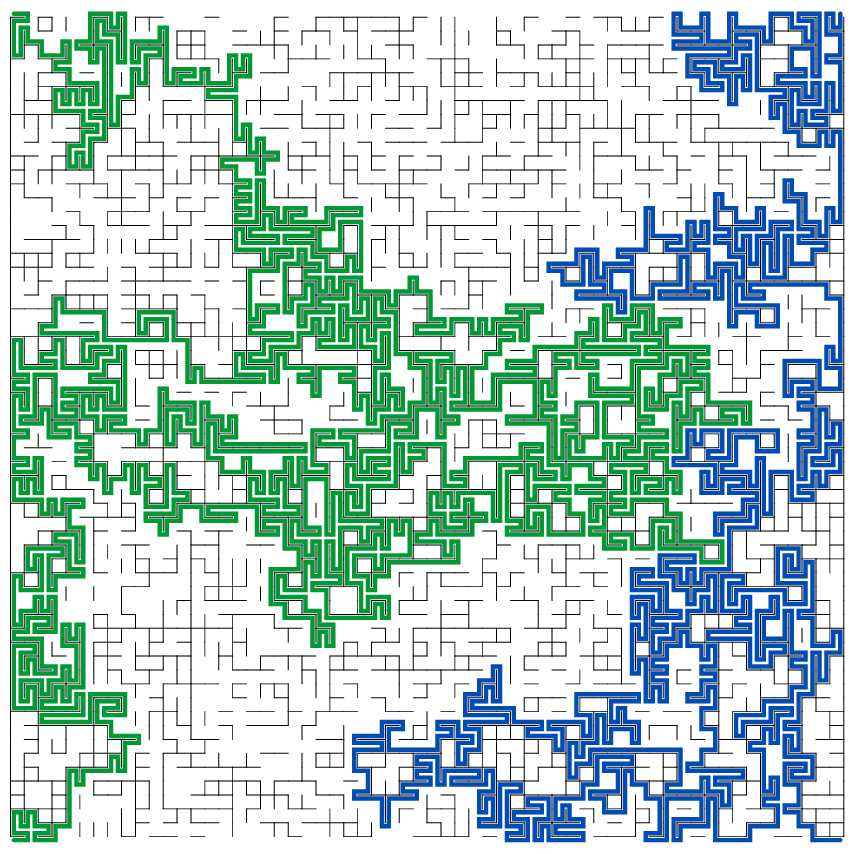}
}}
}
}
\begin{caption}{
\label{f.persample}
Interfaces in percolation on $\Z^2$.
Each successive pair of configurations are related by a noise of about $0.04$,
which results in about one in every $50$ bits being different.
The squares are of size about $60\times 60$. 
The perturbations follow each other cyclically, flipping the edges in three independently chosen subsets, $S_1,S_2, S_3, S_1, S_2, S_3$, arriving back where we started.
}
\end{caption}
\end{figure}

Figure~\ref{f.persample} illustrates a sequence of percolation configurations,
where each configuration is obtained from the previous one by applying some
noise. The effect on the interfaces can be observed.
Theorem~\ref{t.1} (or, in fact, its corollary (\ref{e.tight})) implies the following sharp noise sensitivity estimate
regarding such perturbations:

\begin{corollary}\label{c.noise}
Suppose that $y$ is an $\eps_R$-noisy version of $x$, where 
$\eps_R\in(0,1)$ may depend on $R$.
If $\lim_{R\to\infty} \Ess_R\,\eps_R=\infty$, then
\begin{equation}
\label{e.uncor}
\lim_{R\to\infty}
\Eb{f_R(x)\,f_R(y)}-\Eb{f_R(x)}\Eb{f_R(y)}=0\,,
\end{equation}
while if $\lim_{R\to\infty} \Ess_R\,\eps_R=0$, then
\begin{equation}
\label{e.corel}
\lim_{R\to\infty} \Eb{f_R(x)\,f_R(y)}-\Eb{f_R(x)^2}=0\,.
\end{equation}
\end{corollary}

The second of these statements is actually obvious from~(\ref{e.fcor}) and Jensen's inequality, 
and is brought here only to complement the first claim.
Of course,~\eqref{e.uncor} just means that having a crossing in $x$ is
asymptotically uncorrelated with having a crossing in $y$, while~\eqref{e.corel}
means that with probability going to $1$, a crossing in $x$ occurs if and only
if there is a crossing in $y$.
Although $f_R(x)^2=1$, we find the form of~\eqref{e.corel} more suggestive,
since this is the statement that generalizes to other situations.
Likewise, $\lim_{R\to\infty} \Eb{f_R(x)}=0$, but~\eqref{e.uncor}
is more suggestive.

In Corollary~\ref{c.noisequad}, we prove a generalization of Corollary~\ref{c.noise} 
for the crossing function $f_{R\Quad}$, where $\Quad$ is an arbitrary fixed ``quad'', 
i.e., a domain homeomorphic to a disk with four marked points on its boundary.

In a forthcoming paper on the scaling limit of dynamical percolation \cite{\DPSL} (see also \cite{\PivotalMeasure}),
we plan to show that for critical percolation on the triangular grid,
whenever $\lim_{R\to\infty} \Ess_R\,\eps_R$ exists and is in $(0,\infty)$,
then $\lim_{R} \Eb{f_R(x)\,f_R(y)}$ also exists and is
strictly between the limits of $\Eb{f_R(x)}^2$ and $\Eb{f_R(x)^2}$.

The following theorem proves Conjecture~5.1 from~\cite{\BKSnoise}.
With minor adaptations, it follows from the proof of Theorem~\ref{t.1}.

\begin{theorem}\label{t.horizontal}
Consider bond percolation on $\Z^2$. Let $x$ be a critical
percolation configuration, and let $z$ be another critical percolation
configuration, which equals $x$ on the horizontal edges, but is
independent from $x$ on the vertical edges.
Then having a left-right crossing in $x$ is asymptotically
independent from having a left-right crossing in $z$.
Moreover, the same holds true if ``horizontal'' and ``vertical''
are interchanged (but still considering left
to right crossing).
\end{theorem}

This is just a special case of Proposition~\ref{p.determ}, a quite general result on ``sensitivity to selective noise''.

\subsection{Applications to dynamical percolation}
\label{ss.introdyn}

First, recall that in dynamical percolation the random bits
determining the percolation configuration are refreshed according
to independent Poisson clocks of rate $1$.
Dynamical percolation was proposed by Itai Benjamini  in 1992, and later independently 
by H\"aggstr\"om, Peres and Steif, who wrote the first paper on the subject~\cite{\HaggstromPeresSteif}.  
Since then, dynamical percolation
and other dynamical random processes have been the focus of
several research papers; see the references in Section~\ref{ss.genback}.
In response to a question in~\cite{\HaggstromPeresSteif}
it was proved in~\cite{\SchrammSteif} that critical dynamical site percolation on
the triangular grid a.s.\ has times at which the origin is in an infinite
connected percolation component. Such times are called \lq\lq exceptional\rq\rq,
since they necessarily have measure zero. The paper~\cite{\SchrammSteif}
also showed that the dimension of the set of exceptional times
is a.s.\ in $\bigl[\frac16,\frac{31}{36}\bigr]$, and conjectured that it is a.s.\ $31/36$.
Likewise, it was conjectured there that the set of times at which
an infinite occupied as well as an infinite vacant cluster
coexist is a.s.\ $2/3$, but~\cite{\SchrammSteif} only proved that
$2/3$ is an upper bound, without establishing the existence of
such times. Similarly,~\cite{\SchrammSteif} proved that
the set of times at which there is an infinite percolation component
in the upper half plane has Hausdorff dimension at most $5/9$,
but did not prove that the set is nonempty.
There were numerous other lower and upper bounds of this type
in~\cite{\SchrammSteif}, some of them having to do with dynamical percolation
in wedges and cones, which will not be discussed in this paper.
We can now prove most of these conjectures regarding the Hausdorff
dimensions of exceptional times
in the setting of the triangular grid, and for each case corresponding to a monotone event,
the Hausdorff dimension a.s.\ equals the previously known upper bound. In particular, we have

\begin{theorem}\label{t.dyn}
In the setting of dynamical critical site percolation on the triangular grid, we have the
following a.s.\ values for the Hausdorff dimensions.
\begin{enumerate}
\hitem{i.plane}{1} The set of times at which there is an infinite cluster a.s.\ has Hausdorff dimension $31/36$.
\hitem{i.halfplane}{2} The set of times at which there is an infinite cluster in the upper half plane a.s.\
  has Hausdorff dimension $5/9$.
\hitem{i.multi}{3} The set of times at which an infinite occupied cluster and an  infinite vacant cluster coexist a.s.\
  has Hausdorff dimension at least $1/9$.
\end{enumerate}
\end{theorem}

The reason that in~\iref{i.plane} and~\iref{i.halfplane} the results agree with the conjectured upper bound from~\cite{\SchrammSteif}
is that the upper bound is dictated by $\Eb{|\Spec_f|}$ (which is generally not hard to compute),
while the tail estimate given in~\eqref{e.ssbounds} and its analogs give sufficient estimates to bound
the probability that $|\Spec_f|$ is much smaller than its expectation.
Here, $f$ is the indicator function of some crossing event, which may vary from one application to another. We cannot calculate the exact dimension in item~\iref{i.multi} because we use the monotonicity of $f$ in an essential way (though at only one point), and the event that both
vacant and occupied percolation crossings occur is not monotone. The lower bound $1/9$ comes from using the $5/9$ result in the upper and lower half planes separately,  in which case ``codimensions add'' by independence.

See Section~\ref{s.dyns} for further results and an explanation as to how these numbers are calculated.

\smallskip

The paper~\cite{\SchrammSteif} came quite close to proving that exceptional times exist for dynamical
critical bond percolation on $\Z^2$, but was not able to do it. Now, we close this gap.

\begin{theorem}\label{t.Z2dyn}
A.s.\ there are exceptional times at which dynamical critical bond percolation on $\Z^2$ has infinite clusters, and the Hausdorff dimension of the set of such times is a.s.\ positive.
\end{theorem}

A paper in preparation, \cite{\HPSIIC}, defines a natural local time measure on the set of exceptional times for percolation, and proves that the configuration at a ``typical exceptional time'', chosen with respect to this local time, has the distribution of Kesten's Incipient Infinite Cluster \cite{\KestenIIC}.

There is one more application to dynamical percolation that we will presently mention.
This has to do with the scaling limit of dynamical percolation, as introduced in~\cite{\SchrammICM}, and whose existence we plan to show in \cite{\DPSL} (see also \cite{\PivotalMeasure}). In this scaling limit, time and space are both scaled, and the relationship between their 
scaling is chosen in such a way that the event of the existence of a percolation crossing 
of the unit square at time $0$ and at one unit of time later have some fixed correlation
strictly between $0$ and $1$. Consequently, as space is shrinking, time is expanding.
We leave it as an exercise to the reader to verify that the ratio between the scaling of time
and of space can be worked out directly from the law of $|\Spec_{f_R}|$.
An easy consequence of~\eqref{e.tight} is that in the dynamical percolation scaling limit, the
correlation between having a left-right crossing of the square at time $0$ and at time $t$
goes to zero as $t\to\infty$; see (\ref{e.triangnoise}).
In fact, based on~\cite{\SchrammSmirnovNoise} and estimates such as~\eqref{e.tight}
and its generalizations to other domains,
it can be shown that the dynamical percolation scaling limit is ergodic w.r.t.~time.
These results answer Problem~5.3 from~\cite{\SchrammICM}.

\subsection{The scaling limit of the spectral sample}
\label{ss.introscaling}

The study of the scaling limit of $\Spec$ was suggested by Gil Kalai~\cite[Problem 5.2]{\SchrammICM}
(see also~\cite[Problem 5.4]{\BKSnoise}).
The idea is that we can think of $\Spec_{f_R}$ as a random subset of the plane,
and consider the existence of the weak limit as $R\to\infty$ of
the law of $R^{-1}\,\Spec_{f_R}$.
Boris Tsirelson~\cite{\TsirelsonStFlour} addressed this problem more generally within
his theory of noises, dealing with various functions
$f$ that are not necessarily related to percolation. 
It follows from Tsirelson's theory and from~\cite{\SchrammSmirnovNoise}
that the scaling limit of $\Spec_{f_R}$ exists.
In Section~\ref{s.spectralscaling}, we explain this, and prove 

\begin{theorem}\label{t.speclim}
In the setting of the triangular grid, the limit in law of $R^{-1}\,\Spec_{f_R}$ exists.
It is a.s.\ a Cantor set of dimension $3/4$.
\end{theorem}

The conformal invariance of the scaling limit of $\Spec_{f_R}$ in the setting
of the triangular grid is also proved in Section~\ref{s.spectralscaling}.
These results answer a problem posed by Gil Kalai \cite[Problem~5.2]{\SchrammICM}.

\subsection{A rough outline of the proof}
\label{ss.rough}
The proof of Theorem~\ref{t.1} does not follow the same general strategy as
the proof of the non-sharp bounds given in~\cite{\SchrammSteif}.
The lower bound for the left hand side in~\eqref{e.ssbounds} 
is rather easy, and so we only discuss here the proof of the upper bound.
Fix some $r\in[1,R]$ and subdivide the square $[0,R]^2$ into subsquares
of sidelength $r$ (suppose that $r$ divides $R$, say).
Let $\Spec(r)$ denote the set of these subsquares that intersect $\Spec_{f_R}$.
In Section~\ref{s.verysmall} we estimate the probability that $|\Spec(r)|=k$ when $k$ is 
small (for example, $k=O(\log (R/r))$).
The argument is based on building a rough geometric classification
of all the possible configurations of $\Spec(r)$, applying a bound for each
class, and summing over the different classes.
The bound obtained this way is
\begin{equation}
\label{e.Xbd}
\Pb{|\Spec(r)|=k}\le \exp\bigl(O(1)\,\log^2 (k+2)\bigr)\,
\Bigl(\frac{\Ess_R/R}{\Ess_r/r}\Bigr)^2,
\end{equation}
and has the optimal dependence on $R$ and $r$,
but a rather bad dependence on $k$.

Here is a naive strategy for getting from~\eqref{e.Xbd} to~\eqref{e.ssbounds}, which
does not seem to work.
Fix some $r\times r$ square $B$. Suppose that we are able to show that conditioned
on the intersection of $\Spec$ with some set $W$ in the complement of the $r$-neighborhood of $B$,
and conditioned on $\Spec$ intersecting $B$, we have a probability bounded away from
$0$ that $|\Spec\cap B|>\Ess_r$.
We can then restrict to a sublattice of $r\times r$ squares that are at mutual distance
at least $r$ and easily show by induction that the probability that
$\Spec$ intersects at least $k'$ of the squares in the sublattice but has size
less than $\Ess_r$ is exponentially small in $k'$. 
We may then take a bounded set of such sublattices, 
which covers every one of the $r\times r$ squares in our initial
tiling of $[0,R]^2$. Thus, using the exponentially small bound in $k'$ when $|\Spec(r)|$ is large enough, 
while~\eqref{e.Xbd} when $|\Spec(r)|$ is small, we obtain the required bound on $\Pb{0<|\Spec|<\Ess_r}$.
The reason that this strategy fails is that there are presently no good tools to understand the conditional
law of $\Spec \cap B$  given $\Spec \cap W$.
Refusing to give up, we observe that,
as explained in Section~\ref{ss.ssg}, the law of $B\cap\Spec$ conditioned on the ``negative information'' $\Spec\cap W=\emptyset$
can be described.  
Based on this, we amend the above strategy, as follows. We pick a random set $\Rs\subseteq\I$
independent from $\Spec$,
where each $i\in\I$ is put in $\Rs$ with probability about $1/\Ess_r$ independently.
The reason for using this sparse random set is that for any $r\times r$ square $B$, either $\Spec\cap B \cap \Rs$ is empty, and thus we gained only ``negative information'' about $\Spec$, or $|\Spec\cap B|$ is likely to be as large as $\Ess_r$.
So, as we will see in a second, we can hope to get a good upper bound on $\Pb{\Spec\ne\emptyset=\Rs\cap\Spec}$, which would almost immediately give a constant times the same bound on $\Pb{0<|\Spec|<\Ess_r}$.

%
In Section~\ref{s.mainlemma} we show that for an $r\times r$ square $B$ and the random $\Rs$ as above, 
if we condition on $\Spec\cap B\ne\emptyset$ and on $\Spec\cap W=\emptyset$, where $W\subseteq B^c$,
then with probability bounded away from $0$ we
have $\Spec\cap B'\cap \Rs\ne\emptyset$, where $B'$ is a square of $1/3$ the
sidelength that is concentric with $B$; namely,
\begin{equation}\label{e.pe}
\Pb{\Spec\cap B'\cap \Rs\ne\emptyset\md \Spec\cap W=\emptyset, \Spec\cap B\ne\emptyset}>a>0\,,
\end{equation}
for some constant $a$.
This is based on a second moment argument, but to carry it through we have to
resort to rather involved percolation arguments.
A key observation here is to interpret these conditional events for the spectral
sample in terms of percolation events for a coupling of two configurations (which are independent on the set $W$ 
but coincide elsewhere). An important step is to prove a quasi-multiplicativity property for arm-events
in the case of this system of coupled configurations.

Again, there is a simple naive strategy based on~\eqref{e.pe} and~\eqref{e.Xbd}
to get an upper bound for $\Pb{\Spec\ne\emptyset=\Rs\cap\Spec}$.
One may try to check sequentially if $B'\cap \Rs\cap\Spec\ne\emptyset$ for each of the
$r\times r$ squares $B$, and as long as a nonempty intersection has not been found, the probability
to detect a nonempty intersection is proportional to the conditional probability that
$\Spec\cap B\ne\emptyset$. However, the trouble with this strategy is that the conditional
probability of $\Spec\cap B\ne\emptyset$ varies with time, and the bound~\eqref{e.Xbd}
does not imply a similar bound for the sum of these conditional probabilities,
since each time the conditioning is different. Hence we cannot conclude that $\Spec\cap B\ne\emptyset$ happens many times during the sequential checking. And this issue cannot be solved by first conditioning on having a large total number of nonempty intersections, because we cannot handle such a ``positive conditioning''.

The substitute for this naive strategy is a large deviations estimate that we state and
prove in Section~\ref{s.YX}, namely, Proposition~\ref{pr.YX}.
This result is somewhat in the flavor of the Lov\'asz local lemma and the domination of product measures result of \cite{\LSSdomination} (which proves and uses an extension of the Lov\'asz lemma), since it gives
estimates for probabilities of events with a possibly complicated dependence structure, and more specifically, it says that certain positive conditional marginals ensure exponential decay for the probability of complete failure, similarly to what would happen in a product measure. However, our situation is more complicated than \cite{\LSSdomination}: we have two random variables $x,y\in\{0,1\}^n$ instead of one random field,  and we have only a much smaller set of positive conditional marginals to start with. In the application, $x_i$ is the indicator of the event that $\Spec$ intersects the
$i$'th $r\times r$ square, and $y_i$ is the indicator of the event that $\Spec\cap \Rs$ intersects that
square.
The assumption~\eqref{e.pe} then translates to
$$
\Pb{y_j=1\md y_i=0\, \forall i\in I}
\ge
a\,
\Pb{x_j=1\md y_i=0\,\forall i\in I}\,,\qquad j\notin I\subset[n]\,,
$$
and the proposition tells us that under these assumptions we have
\begin{equation}
\label{e.introXY}
\Pb{y=0 \md X> 0}\le a^{-1}\,\Eb{e^{-aX/e} \md X>0}\,,
\end{equation}
where $X=\sum_ix_i$.
In our application $X=|\Spec(r)|$,
and thus~\eqref{e.introXY} combines with~\eqref{e.pe}
to yield the desired bound.
The proof of Theorem~\ref{t.1} is completed in this way in Section~\ref{s.concent}.

\begin{remark}\label{r.pivotals}
As we mentioned earlier, the above results about the \concentration of the spectrum when normalized by its mean
(as well as the up to constants optimal results on the lower tail) are much harder to achieve than their analogs for the set of pivotal points, even though the two are intimately related. 
Indeed, the techniques of the paper easily adapt to the 
set $\Piv_{f_R}$ of pivotal points of the left-right crossing of the square $[0,R]^2$ and give \concentration results 
on the number of pivotal points $|\Piv_{f_R}|$. For instance,
they imply the following analog of Theorem ~\ref{t.1}: 
$$
\PB{0<|\Piv_{f_R}|<\E |\Piv_{f_r}|} \asymp  (R/r)^2 \alpha_6(r,R)\,.
$$
The appearance of $\alpha_6(r,R)$ in place of $\alpha_4(r,R)^2$ will be explained in Remark~\ref{r.annpivo}. 
In the case of the triangular grid, this gives the following estimate on the lower tail:
$$
\limsup_{R\to \infty} \PB{0<| \Piv_{f_R}| \le \lambda \,\E |\Piv_{f_R}| } = \lambda^{11/9+o(1)}\,,
$$
as $\lambda\to 0$. 
It turns out that in the case of the pivotal points $\Piv_{f_R}$ (where the i.i.d. structure of the percolation configuration
helps), there is a more direct route towards the \concentration of $|\Piv_{f_R}|/\E |\Piv_{f_R}|$ on $(0,\infty)$
than the route we needed to follow
for the spectrum $\Spec$. We do not give more details, since we will not use this pivotal \concentration in this paper.

One might think that this good control on the lower tail of the number of pivotals
should imply that if one waits long enough, the system must switch many times between having and not having
a left-right crossing, and then this should imply an almost complete decorrelation. However, assuming the first implication for now,
the second one still remains completely unclear: even if there are a lot of switches, it might well be that the system started, say,
from having the left-right crossing, will remember this for a long time in the sense that it will move back more quickly from not
having the crossing to having the crossing, than vice versa. In this case, at a given time it would be significantly more likely (by a constant factor) that an even number of switches have happened so far, i.e., the system would not have lost most of the correlation.
We do not see how to rule out this scenario by studying pivotals only, that is why the Fourier spectrum is so useful.
\end{remark}

\bigskip
\noindent {\bf Acknowledgments}.
We are grateful to Gil Kalai for inspiring conversations and for his observation~\eqref{e.gil}, 
and to Jeff Steif for insights regarding the dimensions of exceptional times
in dynamical percolation and for numerous comments on the manuscript. We thank Vincent Beffara for permitting us to include
Proposition~\ref{p.strong}, Alan Hammond, Pierre Nolin and Wendelin Werner for motivating and enlightening discussions, and Erik Broman and Alan Hammond for suggestions on how to improve readability at several places.
Finally, we are extremely grateful to the journal referee for his or her careful reading and many detailed comments, and to MathSciNet reviewer Antal J\'arai for drawing our attention to some mistakes in Subsection~\ref{ss.smallradial} that are corrected in the present version.

\section{Some basics} \label{s.basics}
\subsection{A few general definitions}\label{ss.gen}

In this paper we consider site percolation on the triangular grid as well as
bond percolation on $\Z^2$, both at the critical parameter $p=1/2$.

In the case of site percolation on the triangular grid $T$, a percolation configuration $\omega$
is just the set of sites which are open. However, we often think of $\omega$ as a coloring of the plane by two colors: in the hexagonal grid dual to $T$, a hexagon is colored white if the corresponding site is in $\omega$, while the other hexagons are colored black. If $A$ and $B$ are subsets of the plane, we say that there is a {\bf crossing} in $\omega$ from $A$ to $B$ if there is a continuous path with one endpoint in $A$ and the other endpoint in $B$ that is contained in the closure of the union of the white hexagons. Likewise, a {\bf dual crossing} corresponds to a path contained in the closure of the black hexagons.

In the case of bond percolation on $\Z^2$, there is a similar coloring of the plane
by two colors which has the \lq\lq correct\rq\rq\ connectivity properties.
In this case, we color by white all the points that are within
$L^\infty$ distance of $1/4$ from all the vertices of $\Z^2$
and all the points that are within $L^\infty$ distance of $1/4$ from the edges
in $\omega$, and color by black the closure of the complement of the white colored
points.
%

Regardless of the grid, the set of points whose color is determined by
$\omega_x$ will be called the {\bf tile} of $x$. In the case of the square
grid, we also have tiles with deterministic color, namely, each square
of sidelength $1/2$ centered at a vertex of $\Z^2$ and each square
of sidelength $1/2$ concentric with a face of $\Z^2$. Thus,
in either case we have a tiling of the plane by hexagons
or squares where each tile consists of a connected
set of points whose colors always agree.

A {\bf quad} $\Quad$ is a subset of the plane homemorphic to
the closed unit disk together with a distinguished pair
of disjoint closed arcs on $\p\Quad$.
We say that $\omega$ has a {\bf crossing} of $\Quad$ if the two
distinguished arcs can be connected by a white path inside
$\Quad$.

If $\ev A$ is an event, then the {\bf $\pm 1$ indicator} function of 
$\ev A$ is the function $2\cdot 1_{\ev A}-1$, which is
$1$ on $\ev A$ and $-1$ on $\neg\ev A$.
The $\pm1$ indicator function for the event
that a quad $\Quad $ is crossed will be
denoted by $f_{\Quad}$.

We use $\I$ to denote the set of bits in $\omega$;
that is, in the context of the triangular grid
$\I$ is the set of vertices of the grid, and in the context 
of $\Z^2$ it denotes the set of edges.
Although $\I$ is not finite in these cases,
the functions we consider will only depend on finitely many bits in $\I$,
and so the Fourier-Walsh expansion~\eqref{e.fourier} still holds.
Moreover, for $L^2$ functions depending on infinitely many bits
we still have~\eqref{e.fourier}, except that the summation is restricted
to finite $S\subseteq\I$.

Since we will be considering $\Spec$ as a geometric object,
we find it convenient to think of $\I$ as a discrete set in the plane.
In the context of the triangular grid, this is anyway the case,
but for the square grid we will implicitly associate each
edge of $\Z^2$ with its center; so $\I$ can be considered as
the set of centers of the relevant edges. This way, any subset of the 
plane also represents a subset of the bits. Note however that e.g.~the 
crossing function $f_\Quad$ usually depends on more bits than the ones 
contained in $\Quad$.

For $z\in\R^2$ and $r\ge 0$, the set $z+[-r,r)^2$ will be called the 
square of radius $r$ centered at $z$.
Furthermore, we let $B(z,r)$ 
denote the union of the tiles whose center is contained in $z+[-r,r)^2$, 
and will refer to $B(z,r)$ as a {\bf box} of radius $r$.
One reason for using these boxes (instead of round balls, say) is that
the plane can be tiled with them perfectly.

\subsection{Multi-arm events for percolation}\label{ss.multi}

In many different studies of percolation, the multi-arm events play a central role.
We now define these events (a word of caution --- there are a few different natural variants
to these definitions), and discuss the asymptotics of their probabilities.

\bigskip

Let $A\subset\R^2$ be some topological annulus in the plane, and let $j\in\N_+$.
If $j$ is even, then the {\bf $j$-arm event} in $A$ is the event that there are
$j$ disjoint monochromatic paths joining the two boundary components of $A$,
and these paths in circular order are alternating between white and black.
If $j$ is odd, the definition is similar, except that the order of the colors
is required to be (in circular order) alternating between white and black
with one additional white crossing.

In most papers, the restriction that the colors are alternating is relaxed
to the requirement that not all crossings are of the same color. 
Indeed, it is known that if $A$ is an annulus,
$A=B(0,R)\setminus B(0,r)$, then in the setting of critical site percolation
on the triangular grid the circular order of the colors effects the probability of
the event by at most a constant factor (which may depend on $j$), provided that
in the case $j>1$ there is at least one required crossing from each
color~\cite{\AizenmanDuplantierAharony}.
However, since it appears that the corresponding result for the square grid has
not been worked out, we have opted to impose the alternating colors restriction.

We let $\alpha_j(A)$ denote the probability of the $j$-arm event in $A$.
For the case $A= B(0,R)\setminus B(0,r)$, write $\alpha_j(r,R)$ for $\alpha_j(A)$.
Note that $\alpha_j(r,R)=0$ if $r <<j<R$.
We use $\alpha_j(R)$  as a shorthand for $\alpha_j(2\,j,R)$.
We will also adopt the convention that $\alpha_j(r,R)=1$ if $r\ge R$. 

We now review some of the results concerning these arm events.
The Russo-Seymour-Welsh (RSW) estimates imply 
that
\begin{equation}\label{e.RSW}
s^{-a_j}/C_j \leq \alpha_j(r,s\,r) \leq C_j s^{-1/a_j}
\end{equation}
for all $r>r(j)$ and $s>1$, where $r(j),C_j,a_j>1$ depend only on $j$.  
Another important property of these arm events is {\bf quasi-multiplicativity},
namely,
\begin{equation}\label{e.qm}
\alpha_j(R)/C_j\le \alpha_j(r)\,\alpha_j(r,R) \le C_j\,\alpha_j(R)
\end{equation}
for $r(j)<r<R$, where, again, $r(j),C_j>1$ depend only on $j$.
This was proved in \cite{\KestenScaling}; see \cite[Proposition A.1]{\SchrammSteif} and \cite[Section 4]{\NolinKesten} for concise proofs.
The above properties in particular give for $r<r'<R'<R$ that
\begin{equation}
\label{e.qRSW}
C_j^{-1}
\Bigl(\frac{R\,r'}{R'\,r}\Bigr)^{a_j^{-1}}
\alpha_j(r,R)\le
\alpha_j(r',R')
\le
C_j\,
\Bigl(\frac{R\,r'}{R'\,r}\Bigr)^{a_j}
\alpha_j(r,R)\,,
\end{equation}
with possibly different constants $C_j$.

\bigskip

Of the multi-arm events, the most relevant to this paper is the
$4$-arm event, due to its relation to pivotality for the crossing event in a quad $\Quad$. In particular, for closed $B\subset\R^2$,
we will use $\alpha_\square(B,\Quad)$ to denote the probability of having four arms in $\Quad\setminus B$, the white arms connecting $\p B$ to the two distinguished arcs on $\p\Quad$ and the black arms to the complementary arcs.
If $B\cap\p\Quad\ne\emptyset$, then the arms connecting $B$ to the arcs of $\p\Quad$ which $B$ intersects are considered as present.
Quasi-multiplicativity often generalizes easily to this quantity; for example, if $\Quad$ is an $R\times R$ square with two opposite sides being the distinguished arcs, $B$ is a
radius $r$ box anywhere in $\Quad$, and $x\in B$ is at a distance at least $c r$ from $\p B$, then
\begin{equation}\label{e.qmsq}
{\alpha_\square(x,\Quad)}/{\alpha_\square(B,\Quad)} \asymp \alpha_4(x,B) \asymp \alpha_4(r)\,,
\end{equation}
with the implied constants depending only on $c$.
(A more general version of this will also be proved in Section~\ref{ss.first}.)
Here and in the following, when we write $\alpha_\square(x,\Quad)$ or $\alpha_4(x,B)$, we are referring to the corresponding $4$-arm event from the tile of $x$ to $\p\Quad$ or $\p B$ (with or without paying attention to any distinguished arms on the boundary, respectively).

Let us also recall what is known about $\alpha_4$ quantitatively. For site percolation on the triangular lattice, by~\cite{\SmirnovWerner}, we have
\begin{equation}
\label{e.a4}
\alpha_4(r,R)= (r/R)^{5/4+o(1)}
\end{equation}
as $R/r\to\infty$ while $1\le r\le R$. 
Similar relations are known for $j\ne 4$~\cite{\LSWoneArm,\SmirnovWerner}.
For bond percolation on the square grid, we presently have weaker estimates; in particular,
\begin{equation}
\label{e.a4less2}
C^{-1}\, (r/R)^{2-\eps}\le \alpha_4(r,R) \le C\, (r/R)^{1+\eps}
\end{equation}
for some fixed constants $C,\eps>0$ and every $1\le r\le R$.
The left inequality
can be obtained by combining $\alpha_5(r,R)\asymp (r/R)^2$
(see \cite[Lemma 5]{\KestenSZh} or \cite[Corollary A.8]{\SchrammSteif}),
the RSW estimate $\alpha_1(r,R) < O(1)\,(r/R)^{\eps}$, and, finally, the relation
$\alpha_1(r,R)\,\alpha_4(r,R) \geq \alpha_5(r,R)$
(which follows from Reimer's inequality \cite{\Reimer},
or from Proposition~\ref{p.strong} in our Appendix).
The right hand inequality in~\eqref{e.a4less2}
follows from~\cite{\KestenScaling}; see also~\cite[Remark 4.2]{\BKSnoise}.
  
\subsection{The spectral sample in general}\label{ss.ssg}

This subsection derives some formulas and estimates for $\Pb{\Spec\subseteq A}$,
for the distribution of $\Spec\cap A$,
and for $\Pb{\Spec\cap A=\emptyset\ne\Spec\cap B}$.
We also briefly present an estimate of the variation distance between the
laws of $\Spec_f$ and of $\Spec_g$ in terms of $\|f-g\|$. 
(We generally use $\|\cdot\|$ to denote the $L^2$ norm.)
Moreover, the definition of the set of pivotals $\Piv$ is recalled and some relations
between $\Piv$ and $\Spec$ are discussed.
\bigskip

As before, let $\Omega:=\{-1,1\}^\I$, where $\I$ is finite.
Recall that for $f:\Omega\to\R$ with $\|f\|=1$, we consider the random variable
$\Spec_f$ whose law is given by $\Pb{\Spec=S}=\widehat f(S)^2$.
More generally, if $\|f\|>0$, we use the law given by
$$
\Pb{\Spec=S}=\widehat f(S)^2/\|f\|^2,
$$
but will also consider the un-normalized measure given by
$$
\Qb{\Spec=S}=\widehat f(S)^2.
$$
(If we wish to indicate the function $f$, we may write $\Q_f$ in place of $\Q$.)

Now suppose that $f,g:\Omega\to\R$.
We argue that if $\|f-g\|$ is small, then the law of $\Spec_f$ is close to the law of $\Spec_g$, as follows. 
If we want to compare the ``laws'' under $\Q$, then
\begin{equation}
\label{e.specapprox}
\begin{aligned}
\sum_{S\subseteq\I} \bigl|\widehat f(S)^2-\widehat g(S)^2\bigr|
&
=
\sum_S \bigl|\widehat f-\widehat g\bigr|\,\bigr|\widehat f+\widehat g\bigr|
\\&
\le 
\Bigl(
\sum_S(\widehat f-\widehat g)^2\Bigr)^{1/2}\Bigl(\sum_S(\widehat f+\widehat g)^2
\Bigr)^{1/2}
\\&
= \|f-g\|\,\|f+g\|\,,
\end{aligned}
\end{equation}
where the inequality is due to Cauchy-Schwarz and the final
equality is an application of Parseval's identity. 
For the laws under $\P$, when $\|f\|=\|g\|=1$ does not hold, a slightly more complicated version of this argument gives  that 
if $\|f-g\|\leq\eps \max\{\|f\|,\|g\|\}$ for some $\eps\in (0,1)$, then 
$$
\sum_{S\subseteq\I} \left|\frac{\widehat f(S)^2}{\|f\|^2}-\frac{\widehat g(S)^2}{\|g\|^2} \right|
\leq \frac{4\eps}{(1-\eps)^2}\,.
$$
We will not use this version, hence omit the details of its derivation.
%
%
\bigskip

For $A\subseteq\I$, let $\omega_A$ denote the restriction of $\omega$ to $A$, and
let $\F_A$ denote the $\sigma$-field of subsets of $\Omega$ generated by $\omega_A$.
We use the notation $A^c:=\I\setminus A$ for the complement of $A$.
Observe that for $A\subseteq\I$,
\begin{equation}
\label{e.FA}
\Eb{\chi_S\md\ev F_A}=\begin{cases} \chi_S & S\subseteq A\,,\\
0&\text{otherwise}\,.
\end{cases}
\end{equation}
It follows from this and~\eqref{e.fourier} that
$$
g:=\Eb{f \md \F_A}=\sum_{S\subseteq A}\widehat f(S)\, \chi_S\,.
$$ 
Thus $\widehat g(S)=\widehat f(S)$ for $S\subseteq A$, and $\widehat g(S)=0$ otherwise.
Therefore, Parseval's formula implies 
\begin{equation}\label{e.SinB}
\Qb{\Spec\subseteq A}=\Eb{\Es{f \md \F_A}^2}.
\end{equation}
In principle, this describes the distribution of $\Spec$ in terms of $f$,
and indeed we will extract information about $\Spec$ from this formula and its
consequences.

Using~\eqref{e.fourier} and~\eqref{e.FA}, one obtains for $S\subseteq A\subseteq \I$
$$
\Eb{f\,\chi_S\md\ev F_{A^c}}=\sum_{S'\subseteq A^c}\widehat f(S\cup S')\,\chi_{S'}\,.
$$
This gives 
$$
\EB{\Eb{f\,\chi_S\md\ev F_{A^c}}^2}=\sum_{S'\subseteq A^c}\widehat f(S\cup S')^2
=\Qb{\Spec\cap A=S}\,,
$$
which implies the following Lemma from~\cite{\LMN}.
Roughly, the lemma says that in order to sample the random variable 
$\Spec\cap A$, one can first pick a random sample of $\omega_{A^c}$,
and then take a sample from the spectral sample of the function
we get by plugging in these values for the bits in $A^c$.

\begin{lemma}\label{l.LMN}
Suppose that $f:\Omega\to\R$, and $A\subseteq \I$.
For $x\in \{-1,1\}^A$ and $y\in\{-1,1\}^{A^c}$, write
$g_y(x):=f\bigl(\omega(x,y)\bigr)$, where $\omega(x,y)$ is the
element of $\Omega$ whose restriction to $A$ is $x$ and whose
restriction to $A^c$ is $y$.
Then for every $S\subseteq A$, we have
$\Qb{\Spec_f\cap A=S} = \Eb{\widehat{g_y}(S)^2}=\Eb{\Bs{\Q_{g_y}}{\Spec_{g_y}=S}}$.
\QED
\end{lemma}

For any $\omega\in\Omega$ and any $A\subseteq \I$, let $\omega_A^{+}$
denote the element of $\Omega$ that is equal to $1$ in $A$ and
equal to $\omega$ outside of $A$. Similarly, let $\omega_A^{-}$
denote the element of $\Omega$ that is equal to $-1$ in $A$ and
equal to $\omega$ outside of $A$.
An $i\in\I$ is said to be {\bf pivotal} for $f:\Omega\to\R$ and $\omega$ if
$f(\omega^+_{\{i\}})\ne f(\omega^-_{\{i\}})$.
Let $\Piv=\Piv_f$ denote the (random) set of pivotals.

It is known from~\cite{\KKL} that for functions $f:\Omega\to\{-1,1\}$,
\begin{equation}\label{e.kkl}
\Eb{|\Spec|}  =\Eb{|\Piv|},
\end{equation}
Gil Kalai (private communication) further observed that it also holds for the second moment
\begin{equation}\label{e.gil}
\Eb{|\Spec|^2}  =\Eb{|\Piv|^2},
\end{equation}
but this does not extend to higher moments. (Note that in these formulas, 
the expectation $\E$ is with respect to different measures on the left and right hand sides.)
To prove~\eqref{e.kkl} and~\eqref{e.gil}, consider some $i,j\in\I$.
In Lemma~\ref{l.LMN}, if we take $A=\{i\}$, then $g_y$ is a
constant function (of $x$) unless $i\in\Piv$, while
$g_y=\pm\chi_{\{i\}}$ if $i\in\Piv$. Therefore,
the lemma gives 
\begin{equation}
\label{e.SP}
\Pb{i\in\Spec}= \Pb{\Spec\cap\{i\}=\{i\}}= \Pb{i\in\Piv},
\end{equation}
which sums to give~\eqref{e.kkl}.
Similarly, one can show that 
\begin{equation}\label{e.2ndmom}
  \Pb{i,j\in\Spec}=\Pb{i,j\in\Piv}
\end{equation}
by using Lemma~\ref{l.LMN} with $A=\{i,j\}$
to reduce~\eqref{e.2ndmom} to the case where $\Omega=\{-1,1\}^2$,
which easily yields to direct inspection.
Now~\eqref{e.gil} follows by summing~\eqref{e.2ndmom} over $i$ and $j$.

\bigskip
As we have discussed in the Introduction,~(\ref{e.fcor}) immediately shows that the distribution of the size $|\Spec_f|$ governs the sensitivity of $f:\Omega\to\{-1,1\}$ to $\eps$-noise: if $\eps\, \Eb{|\Spec|}$ is small, then the correlation is always close to 1, while the $\eps$ needed for decorrelation is given by the lower tail of $|\Spec|$. One can also ask finer questions, concerning ``sensitivity to selective noise'': which deterministic subsets $\determ \subseteq \I$ have the property that knowing the bits in $\determ$ (and having  unknown independent random bits in $\determc$) we can predict $f$ with probability close to 1, and which subsets $\determ$ give almost no information on $f$? In Subsection~\ref{ss.deterministic} we will see, using~(\ref{e.SinB}), that the first case ($\determ$ is ``almost decisive'') happens if{f} $\Pb{\Spec\subseteq\determ}$ is close to 1, while the second case ($\determ$ is ``almost clueless'') happens if{f} $\Pb{\emptyset\not=\Spec\subseteq\determ}$ is close to 0. For percolation, we will be able to understand both cases quite well, and will also see that although the pivotal and spectral sets, $\Piv$ and $\Spec$, are different in many ways, it is reasonable to conjecture that they have the same decisive and clueless sets (see Remark~\ref{r.polarset}). There is probably a similar phenomenon for many Boolean functions. 
\bigskip

We now derive estimates for $\Qb{\Spec\cap B\ne\emptyset=\Spec\cap W}$, when $B$ and $W$ are
disjoint subsets of the bits.
Define $\Lambda_B=\Lambda_{f,B}$ as the event that $B$ is pivotal for $f$. More precisely,
$\Lambda_B$ is the set of $\omega\in\Omega$ such that there is some $\omega'\in\Omega$
that agrees with $\omega$ on $B^c$ while $f(\omega)\ne f(\omega')$.
Also define $\lala(B,W):=\Pb{\Lambda_B\md \mathcal{F}_{W^c}}$.

\begin{lemma}\label{l.lambda}
Let $\Spec=\Spec_f$ be the spectral sample of some
$f:\Omega\to\R$, and let $W$ and $B$ be disjoint subsets of $\bits$.
Then
$$
\Qb{\Spec\cap B\ne\emptyset=\Spec\cap W}
\le 4\,\|f\|^2_\infty\,\Eb{\lala(B,W)^2}.
$$
\end{lemma}

\proof
{}From~\eqref{e.SinB}, 
\begin{equation} 
\label{e.SBA}
\begin{aligned}
\Qb{\Spec\cap B\ne \emptyset,\ \Spec\cap W =\emptyset} 
&
= \Qb{\Spec \subseteq W^c}-\Qb{\Spec\subseteq (W\cup B)^c}
\\ &
= \EB{\Es{f \md \F_{W^c}}^2-\Es{f \md \F_{(W\cup B)^c}}^2}
\\ &
= \EB{\Bigl(\Es{f \md \F_{W^c}}-\Es{f \md \F_{(W\cup B)^c}}\Bigr)^2} .
\end{aligned}
\end{equation}
On the complement of $\Lambda_B$, we have $f=\Eb{f\md \ev F_{B^c}}$. Therefore,
$$
-2\,\|f\|_\infty\,1_{\Lambda_B} \le
f-\Es{f\md  \mathcal{F}_{B^c}}\le
2\, \|f\|_\infty 1_{\Lambda_B}\,.
$$
Taking conditional expectations throughout, we get
$$
-2\,\|f\|_\infty\,\lala(B,W) \le
\Eb{f\md \ev F_{W^c}}-\Eb{\Es{f\md\ev F_{ B^c}}\md \ev F_{W^c}}\le
2\, \|f\|_\infty \lala(B,W)\,.
$$
Note that
$\Eb{\Es{f\md\ev F_{ B^c}}\md \ev F_{W^c}} =\Eb{f\md \ev F_{(B\cup W)^c}}$,
since our measure on $\Omega$ is i.i.d.
Thus, the above gives
$$
\Bigl|\Eb{f\md \ev F_{W^c}}- \Eb{f\md \ev F_{(B\cup W)^c}}\Bigr| \le 2\,\|f\|_\infty\,
\lala(B,W)\,.
$$
An appeal to~\eqref{e.SBA} now completes the proof.
\QED

Taking $W=\emptyset$ yields $\Qb{\Spec\cap B \neq \emptyset} \leq 4\|f\|_\infty \, \Pb{\Lambda_B}$, 
since $\Lambda_{B,\emptyset}=1_{\Lambda_B}$. Taking $W=B^c$ yields 
$\Qb{\emptyset\neq \Spec \subseteq B} \leq  4\|f\|_\infty \, \Pb{\Lambda_B}^2$, since 
$\Lambda_{B,B^c}=\Pb{\Lambda_B}$. See also Lemma~\ref{l.basic} and the discussion afterwards.

\begin{remark}\label{r.xlambda} In the special case when $f$ is monotone and $\pm1$-valued, and $B=\{x\}$ is a single bit, the lemma takes a more precise form, as we will prove in Subsection~\ref{subs.int}:
$\Pb{x\in\Spec,\ \Spec \cap W=\emptyset} = \Eb{\lala(x,W)^2}$.
\end{remark}

What turns out to be important in Section~\ref{s.mainlemma} is that,
in the context of percolation, the quantity
$\Eb{\lala(B,W)^2}$ can be studied and controlled when $B$ is
a box and $W\subseteq \bits\setminus B$ is arbitrary. Likewise,
in Section~\ref{s.verysmall}, we use a variant of Lemma~\ref{l.lambda}
in which we look at the event that $\Spec$ intersects a collection of
boxes and is disjoint from some collection of annuli.

\section{First percolation spectrum estimates} \label{s.firstest}

We will now consider the special case of $\Spec_f$ when $f=f_\Quad$ for a quad $\Quad\subset \R^2$.
As noted in Section~\ref{ss.gen}, we will be considering $\I$ and $\Spec\subseteq\I$ as
subsets of the plane. 
When $R>0$, we will use the notation $R\,\Quad$ to denote the quad obtained from $\Quad$
by scaling by a factor of $R$ about $0$.

\begin{lemma}[First and second moments]\label{l.moments}
Let $\Quad\subset\R^2$ be some quad and let $U$ be an
open set whose closure is contained in the interior of $\Quad$.
Let $\Spec:=\Spec_{f_{R\Quad}}$ be the spectral sample for
the $\pm1$-indicator function of crossing $R\Quad$.
There are constants $C,R_0>0$, depending only on $\Quad$ and $U$,
such that for all $R>R_0$
$$
C^{-1}\,R^2\,\alpha_4(R)
\le \Eb{|\Spec_{f_{R\Quad}} \cap R\,U|}
\le
C \,R^2\,\alpha_4(R)
$$
and
$$
\Eb{|\Spec_{f_{R\Quad}} \cap R\,U|^2}
\le
C\,
\Eb{|\Spec_{f_{R\Quad}} \cap R\,U|}^2.
$$
\end{lemma}

The reason for the appearance of $\alpha_4$ in the
first moment is the following. We know from~\eqref{e.SP}
that $\Pb{x\in\Spec}=\Pb{x\in\Piv}$ for $\Piv:=\Piv_{f_{R\Quad}}$.
In order for $x$ to be pivotal for $f_{R\Quad}$ it is necessary and
sufficient that there are white paths
in $R\Quad$ from the tile of $x$ to the two distinguished arcs on
$R\,\p\Quad$ and two black paths in $R\Quad$ from the tile
of $x$ to the complementary arcs of $R\p\Quad$. 
These four paths form the $4$-arm event in the annulus between
the tile of $x$ and $R\p\Quad$. Moreover, it is well-known
(and follows from quasi-multiplicativity arguments; see \cite{\KestenScaling, \WWperc}) that 
\begin{equation}\label{e.pivo4}
\forall x\in \I\cap R\,U:\qquad \Pb{x\in\Piv_{f_{R\Quad}}}=\alpha_\square(x,R\Quad)\asymp \alpha_4(R)\,.
\end{equation}
Here and in the proof below, we use the notation $g\asymp g'$ to mean that there
is a constant $c>0$ (which may depend on $U$ and $\Quad$),
such that $g\le c\,g'$ and $g'\le c\,g$.
Likewise, the $O(\cdot)$ notation will involve constants that may depend
on $\Quad$ and $U$.

\proofof{Lemma \ref{l.moments}}
From~\eqref{e.SP} and~(\ref{e.pivo4}) we get that
\begin{equation}\label{e.percmom1}
\forall x\in \I\cap R\,U:\qquad
 \Pb{x\in\Spec} \asymp \alpha_4(R)\,.
\end{equation}
The first claim of the lemma is obtained by summing over $x\in\I\cap R\,U$.

Now consider $x,y\in\I\cap R\,U$.
Let $a$ be the distance from $\closure U$ to $\p \Quad$. Thus $a>0$.
Then by~\eqref{e.2ndmom}, we have $\Pb{x,y\in\Spec}=\Pb{x,y\in\Piv}$.
In order for $x,y\in\Piv$ it is necessary that the $4$ arm
event occurs from the tile of $x$ to distance $(|x-y|/3)\wedge (a\,R)$ away,
and from the tile of $y$ to distance $(|x-y|/3)\wedge (a\,R)$ away, and
from the circle of radius $2\,|x-y|$ around $(x+y)/2$ to distance
$a\,R$ away (if $2\,|x-y|<a\,R$). By independence on disjoint subsets of $\I$,
this (together with the regularity properties of
the $4$-arm probabilities~\eqref{e.qRSW}) gives for $R$ sufficiently large 
\begin{equation*}
 \Pb{x,y\in\Spec} \le
O(1)\,\alpha_4(|x-y|)^2\,\alpha_4(|x-y|,R) 
\,.
\end{equation*}
Using the quasi-multiplicativity property of $\alpha_4$, this gives
\begin{equation}
\label{e.percmom2}
\forall x,y\in \I\cap R\,U:\qquad
 \Pb{x,y\in\Spec}
\le O(1)\,\alpha_4(R)^2/\alpha_4(|x-y|,R)\,.
\end{equation}
The number of pairs $x,y\in\I\cap U$ such that $|x-y|\in [2^n,2^{n+1})$
is $O(R^2)\,2^{2n}$, and is zero if $|x-y|>R\,\diam(\Quad)$.
Therefore, we get from~\eqref{e.percmom2} and the regularity property~\eqref{e.qRSW}
that
$$
\Eb{|\Spec\cap RU|^2}
\le O(R^2)\,\alpha_4(R)^2 \,\sum_{n=0}^{\log_2(R)+O(1)}\frac{ 2^{2n}}{\alpha_4(2^n,R)}\,.
$$
{}From (\ref{e.a4less2}) we get $2^{2n}/\alpha_4(2^n,R) \leq O(1)\,R^{2-\eps}\,2^{\eps n}$. Hence the sum over $n$ is at most $O(R^2)$, and we obtain the desired bound on the second moment.
\QED

Note that $\Eb{|\Spec \cap R\,U|}\to \infty$ as $R\to\infty$, which follows from Lemma \ref{l.moments}
and  ~\eqref{e.a4less2}. Moreover, by the standard Cauchy-Schwarz second-moment argument (also called the Paley-Zygmund inequality), the above lemma implies that for some constant $c>0$ (which may depend on $\Quad$ and $U$), 
$$
\PB{|\Spec\cap RU| > c\,\E{|\Spec\cap RU|}}>c\, \text{ for all $R>0$}.
$$

We also note the following lemma.

\begin{lemma} \label{l.basic}
Let $\Quad\subset\R^2$ be a quad, and set $\Spec=\Spec_{f_\Quad}$.
Let $B$ be some union of tiles such that $B\cap\Quad$ is nonempty and connected.
Then
\begin{equation}\label{e.SintersectsB}
\Ps{\Spec\cap B \ne \emptyset} \leq 4\,\alpha_\square(B,\Quad)\,,
\end{equation}
and
\begin{equation}\label{e.ScontainedinB}
\Ps{\emptyset\ne\Spec\subseteq B} \leq 4\, \alpha_\square(B,\Quad)^2.
\end{equation}
When $B$ is a single tile, corresponding to $x\in\bits$, we have
\begin{equation}
\label{e.singleBit}
\Pb{x \in \Spec} = \alpha_\square(x,\Quad)\hbox{\quad\text{and}\quad} 
\Pb{\Spec = \{x\} } = \alpha_\square(x,\Quad)^2.
\end{equation}
\end{lemma}

\proof
Note that $\Pb{\Lambda_B}=\alpha_\square(B,\Quad)$, since $\Lambda_B$ holds
if an only if the $4$-arm event from $B$ to the corresponding arcs
on $\p\Quad$ occurs.
Since
$\lala(B,\emptyset)=1_{\Lambda_B}$, the first claim follows from Lemma~\ref{l.lambda}
with $W=\emptyset$.
Similarly,~\eqref{e.ScontainedinB} follows by taking $W=B^c$.
The identity $\Pb{x\in\Spec}=\alpha_\square(x,\Quad)$ follows from~\eqref{e.SP}.
Finally, the right hand identity in~\eqref{e.singleBit} can be derived
from~\eqref{e.SBA} with $B=\{x\}$ and $W=\bits\setminus B$. Alternatively,
it also follows from
$\Pb{\{x\}=\Spec}=\Pb{x\in\Spec}^2$, which holds for arbitrary
monotone $f:\Omega\to\{-1,1\}$.
\QED

As we will see in Section \ref{s.mainlemma}, both inequalities in Lemma~\ref{l.basic} are actually approximate equalities when $B\subset\Quad$. The main reason for this is a classical arm separation phenomenon,
see e.g.~\cite[Appendix, Lemma A.2.]{\SchrammSteif}, which roughly says that conditioned on having four arms connecting $\p B$ to the appropriate boundary arcs on $\p\Quad$, with positive conditional probability, these arms are ``well-separated'' on $\p B$. On this event, a positive proportion of the $\omega_B$ configurations enable the crossing of $\Quad$, while a positive proportion disable all crossings. However, for radial crossings, the estimates from Lemma~\ref{l.lambda} are not sharp --- see Lemma~\ref{l.ComphRad} and the discussion afterwards.

Lemma \ref{l.basic} 
has the following immediate consequence. 
Let $\Quad$ be the $\nn\times \nn$ square with two opposite sides as distinguished boundary arcs. Then, for a box $B\subseteq \Quad$ of radius $r$ and a concentric sub-box $B'$ of radius $r/3$, if $B'\cap\bits\ne\emptyset$, then
\be
\EB{|\Spec \cap B'| \md \Spec \cap B \ne \emptyset} 
&=& \sum_{x\in B'}\frac{\Ps{x\in\Spec}}{\Ps{\Spec \cap B \ne \emptyset}} 
\geq \sum_{x\in B'} \frac{\alpha_\square(x,\Quad)}{4\alpha_\square(B,\Quad)}\nonumber\\
&\asymp& |B'|\alpha_4(r)  \asymp r^2 \alpha_4(r),\label{e.mean}
\ee
where we used the quasi-multiplicativity result~(\ref{e.qmsq}).
This result already suggests that $\Spec$ has self-similarity properties that a random fractal-like object should have, and it should be unlikely that it is very small. This idea will, in fact, be of key importance to us, and will be developed in Section~\ref{s.mainlemma}.

\section{The probability of a very small spectral sample} \label{s.verysmall}

\subsection{The statement}\label{ss.verysmallintro}

In this section, we study the Fourier spectrum of the $\pm 1$ indicator function $f$
of having a crossing of a square, or more generally, of a quad $\Quad$.

Divide the plane into a lattice of $r\times r$ subsquares, that is, $r\,\Z^2$,
and for any set of bits $S\subseteq\I$ define 
$$S_r:=\{\text{those $r\times r$ squares that intersect }S\}.$$ 
In particular, $\Spec_r$ is the set of $r$-squares whose intersection with the spectral sample $\Spec$ of $f$  is nonempty.
Following is an estimate for the probability that $\Spec$ is very small, or, more generally,
that $\Spec_r$ is very small.

\begin{proposition}\label{pr.verysmall} 
Let $\Spec$ be the spectral sample of $f=f_{[0,\nn]^2}$, the $\pm1$ indicator function of the left-right
crossing of the square $[0,\nn]^2$.
For $g(k):=2^{\coa\log_2^2 (k+2)}$, with $\coa>0$ large enough, and $\gamma_r(\nn):=(\nn/r)^2 \alpha_4(r,\nn)^2$,
$$
\forall k,\nn,r\in\N_+\qquad
\Pb{|\Spec_r|=k} \leq g(k)\; \gamma_r(\nn).
$$
\end{proposition}

The square prefactor $(\nn/r)^2$ in the definition of $\gamma_r$ reflects the two dimensionality
of the ambient space --- it corresponds to the number of different ways to choose a square of size
$r$ inside $[0,\nn]^2$. The factor $\alpha_4(r,\nn)^2$ comes from the second inequality in Lemma~\ref{l.basic}. In fact, since that lemma will turn out to be sharp up to a constant factor (when the $r$-square is not close to the boundary of $[0,\nn]^2$), the $k=1$ case of Proposition~\ref{pr.verysmall} is also sharp. When we use this proposition, this will be important, as well as the fact that the dependence on $k$ is sub-exponential.

Recall from~\eqref{e.a4} that for critical site percolation on the triangular lattice
$\gamma_r(\nn)=(r/\nn)^{1/2+o(1)}$, as $\nn/r \to\infty$.
Also, note that for critical bond percolation on $\Z^2$
we have $\gamma_r(\nn)<O(1)\,(r/\nn)^{\epsilon}$ for some $\epsilon>0$
by~\eqref{e.a4less2}.

For either lattice, the $r=1$ case of the proposition implies that for arbitrary $C,a>0$, we have $\lim_{R\to \infty} \Pb{0<|\Spec|< C(\log R)^a}=0$.
This is already stronger than the $a=1$ and small $C$ result of \cite{\BKSnoise}, whose proof used more analysis but less combinatorics. 

In order to demonstrate the main ideas of the proof of the proposition in a slightly
simpler setting in which considerations having to do with 
the boundary of the square do not appear, we will first state and prove 
a \lq\lq local\rq\rq\ version of this proposition. Then we will see in Subsection~\ref{ss.boundary} that the boundary of the square indeed has no significant effect; the main reason for this is that the
spectral sample ``does not like'' to be close to the boundary. (Let us note here that this phenomenon holds in any quad,  
as we will see in Subsection~\ref{ss.quadconc}. However, the exact computations that are needed to get the bounds of Proposition~\ref{pr.verysmall} apply only to the case of the square, whose boundary is easy to handle.)

In Subsection~\ref{ss.smallradial}, we will prove a radial version of Proposition~\ref{pr.verysmall}, which will be similar to the square case, using one main additional trick.

\subsection{A local result}\label{ss.local}

\begin{proposition}\label{pr.local}
Consider some quad $\Quad$, and let $\Spec$ be the spectral sample of $f=f_\Quad$, the $\pm1$ indicator function for the crossing event in $\Quad$.
Let $U'\subset U\subset \Quad$, let $R$ denote the diameter of $U$, let $a\in (0,1)$,
and suppose that the  distance from $U'$ to the complement of $U$ is at least $a\,R$.
Let $\mathcal S(r,k)$ be the collection of all sets $S\subseteq\bits$ such that
$\bigl|(S\cap U)_r\bigr|=k$ and $S\cap (U\setminus U')=\emptyset$.
Then for $g$ and $\gamma_r$ as in Proposition~\ref{pr.verysmall}, we have
$$
\forall k,r\in\N_+\qquad
\Pb{\Spec\in\mathcal S(r,k)} \le c_a\,g(k)\,\gamma_r(R)\,,
$$
where $c_a$ is a constant that depends only on $a$.
\end{proposition}

We preface the proof of the proposition with a rough sketch of the main ideas.
When $S\in\mathcal S(r,k)$ and $k$ is small, the
set $(S\cap U)_r$ has to consist of one or very few \lq\lq clusters\rq\rq\ of
$r$-squares that are small (since their total cardinality is $k$, and the elements of each cluster are close to each other) and well separated from each other (otherwise they would not be distinct clusters).
The probability that $(\Spec\cap U)_r$ has just one cluster, contained
in a specific small box $B$ of Euclidean diameter anywhere between $r$ and $g(O(k))r$, can be estimated by~(\ref{e.ScontainedinB}) of Lemma~\ref{l.basic}, at least if we assumed $U=\Quad$. For general $U$, we will soon prove a generalization of (\ref{e.ScontainedinB}), Lemma~\ref{l.Comph}.
Then, one may sum over an appropriate collection of such small boxes $B$ to get a reasonable bound
for the probability that $\diam(\Spec\cap U)$ is small while $\Spec\cap U\ne\emptyset$.
To deal with the case where $(\Spec \cap U)_r$ has a few different well-separated clusters,
we will again use Lemma~\ref{l.Comph}.
The more involved part of the proof will be to classify the possible cluster
structures of $\Spec$ and sum up the bounds corresponding to each possibility.

\proofof{Proposition \ref{pr.local}} Let $\Ann$ be a finite collection of disjoint (topological) annuli in the plane;
we call this an {\bf annulus structure}. We say that a set $S \subset \R^2$ is {\bf compatible} with $\Ann$ (or vice versa) if it is contained in $\R^2\setminus\bigcup\Ann$ and intersects the inner disk of each annulus in $\Ann$.  Define $h(\Ann)$ as the probability that each annulus in $\Ann$ has the four-arm event.
By independence on disjoint sets, we have
\begin{equation}\label{e.hann}
h(\Ann)=\prod_{A\in\Ann}h(A)\,.
\end{equation}
 Suppose that $\ANN$ is a set of annulus structures $\Ann$ such that each
annulus $A\in\Ann$ is contained in $\Quad$, and $\ANN$ has the property that each $S\in\mathcal S(r,k)$
must be compatible with at least one $\Ann \in \ANN$. We claim that any such set $\ANN$ satisfies
\begin{equation}\label{e.ANN}
\Pb{\Spec\in \mathcal S(r,k)} \le \sum_{\Ann \in \ANN} h(\Ann)^2.
\end{equation}
For this, it is clearly sufficient to verify the following lemma, which is a generalization of~(\ref{e.ScontainedinB}) from Lemma~\ref{l.basic}.

\begin{lemma}\label{l.Comph}
For any annulus structure $\Ann$ with $\bigcup \Ann\subset \Quad$, 
$$
\Pb{\Spec \text{ is compatible with }\Ann} \le h(\Ann)^2.
$$
\end{lemma}

\proof 
We write $f(\window,\eta)$ for the value of $f$, where $\window$ is the configuration inside $\bigcup \Ann$ and $\eta$ is the configuration outside.
For each $\window$, let $F_\window$ be the function of $\eta$ defined by $F_\window(\eta):=f(\window,\eta)$. If $\window$ is such that there is an annulus $A\in\Ann$ without the 4-arm event, then the connectivity of points outside the outer boundary
of $A$ does not depend on the configuration inside the inner disk of $A$, and therefore $F_\window$ does
not depend on any of the variables in the inner disk of $A$. Thus, if a subset of variables $S$ is disjoint from $\bigcup \Ann$ (so that we can talk about $\widehat F_\window(S)$) and intersects this inner disk, then the corresponding Fourier coefficient vanishes: $\widehat F_\window(S)=\Es{F_\window\chi_S}=0$. Therefore, if we let $W$ be the linear space of functions spanned by $\{\chi_S : S$ compatible with $\Ann\}$, and $P_W$ denotes the orthogonal projection onto $W$, then $P_W F_\window\ne 0$ implies the $4$-arm event of $\window$ in every $A\in\Ann$.

Now, observe that $P_W f=P_W\Es{f\md \eta}$, because for all
$g\in W$ we have $\Eb{\Es{f\md\eta}\,g}=\Eb{\Es{f\,g\md\eta}}=\Es{f\,g}$.
Next, by the independence of $\eta$ and $\window$ from each other, we have 
$\Es{f\md \eta}=\E^\window[F_\window]$, where the right hand side
is an expectation w.r.t.~$\window$, hence is still a function of $\eta$.
Thus, $P_W f= P_W \E^\window[F_\window]=\E^\window[P_W F_\window]$.
Consequently,
$$ 
\Pb{\Spec\text{ is compatible with }\Ann} = \| P_W f \|^2 = \| \E^\window[P_W F_\window] \|^2 \leq  
\big(\E^\window\bigl[\| P_W F_\window \| \bigr]\big)^2,
$$
where the last step follows from the triangle inequality for $\|\cdot\|$.
For every $\window$, the function $F_\window$ is bounded from above by 1 in absolute value. Therefore,
$\|F_\window\| \leq 1$ for every $\window$.
This implies $\| P_W F_\window \| \leq 1$ for every $\window$ (the norm is the $L^2$ norm and
$P_W$ is an orthogonal projection). Hence, we get
$$
\Pb{\Spec\text{ is compatible with }\Ann} \leq \P^\window \bigl[ P_W F_\window \ne 0\bigr]^2 \leq h(\Ann)^2.
$$
This proves the lemma.\qed


A trivial but important instance of (\ref{e.ANN}) is when $\ANN=\{\emptyset\}$: the empty annulus structure $\emptyset$ is compatible with any $S$, while $h(\emptyset)^2 = 1$.

Now, for each set $S\in\mathcal S(r,k)$ 
 we construct a compatible annulus structure $\Ann(S)$, such that the set $\ANN=\ANN(r,k)=\{\Ann(S) : S\in\mathcal S(r,k)\}$ will be small and effective enough for (\ref{e.ANN}) to imply Proposition \ref{pr.local}. The main idea for this construction is that the spectral sample tends to be clustered together, which can already be foreseen in Lemma \ref{l.Comph}:
each additional thick annulus (corresponding to some part of the spectral sample far from all other parts) decreases the weight $h(\Ann)^2$ by quite a lot --- more than what can be balanced by the number of essentially different ways that this can happen. (We will make this vague description of clustering more precise after the end of the proof, in Remark~\ref{r.diam}.)
\bigskip

We now prepare to define
the annulus structure $\Ann(S)$ corresponding to an $S\in\mathcal S(r,k)$, based on the geometry of
$S$. For this definition, we need to first define what we call {\bf clusters} of $S$.
Set $V=V_S:= (S\cap U)_r$, the set of squares in the $r\,\Z^2$ lattice which meet $S\cap U$.
For $j\in\N_+$ let $G_j$ be the graph on $V$, where two squares are joined
if their Euclidean distance is at most $2^j\,r$, and let $G_0$ be the graph on $V$ with
no edges.
If $j\in\N_+$ and $C\subseteq V$ is a connected component
of $G_j$, but is not connected in $G_{j-1}$, then $C$ is called a level $j$ cluster of $S$.
The level $0$ clusters are the connected components of $G_0$, that is, the singletons
contained in $V$.
The level of a cluster $C$ is denoted by $j(C)=j_S(C)$. 
The Euclidean diameter of a cluster $C$ is clearly at most $2^{j(C)+2}\,r\,|C|$.

We will now associate to each cluster $C$ of $S$ an ``inner'' and an ``outer bounding square'', whose difference will be the annulus of $C$. We will have to make sure that the annuli we are constructing are all disjoint from the set $S$, hence the inner bounding square should be large enough to contain the points of $C$, while the outer square should be small enough not to intersect other clusters. However, we cannot simply take the smallest and largest such squares for each $C$, because we do not want to get too many different annulus structures, i.e., the bounding squares should not depend too sensitively on $C$. One consequence of this crudeness is that some of the clusters may have an empty annulus associated to them: this will happen when the other clusters are too close.

Let $C$ be a cluster in $S$ at some level $j=j(C)$. 
We choose a point of the plane, $z\in [C]$, in an arbitrary but fixed way (say the lowest among the leftmost points of $C$), where $[C]$ denotes the set of points of the plane covered by the $r$-squares in $C$.
Let $z'$ be a point with both coordinates divisible by $2^j\,r$, which is closest to $z$,
with ties broken in some arbitrary but fixed manner.
Define the {\bf inner bounding square} $B(C)$ as the square with edge-length
$|C|\,2^{j+4}\,r$
centered at $z'$ (with edges parallel to the coordinate axes).
Note that $[C]\subseteq B(C)$ and the distance from $[C]$ to $\partial B(C)$ is at least
$2^{j+3}\,|C|\,r-2^j\,r-2^{j+2}\,|C|\,r \geq 2^j\,r$. 

If $C\ne V$ is a cluster, then the smallest cluster in $V$ that properly contains $C$
will be called the {\bf parent} of $C$ and denoted by $C^p$.
In this case, let the {\bf outer bounding square} $\bar B(C)=\bar B_S(C)$
of $C$ be the square concentric with $B(C)$ having
sidelength $(2^{j(C^p)-4}\,r)\wedge (a\,R/4)$.
For the case $C=V$, we let $\bar B(V)$ be the square concentric with $B(V)$
having sidelength $a\,R/4$.

It is not necessarily the case that $\bar B(C)\supset B(C)$, nor even that $\bar B(C) \supset [C]$. 
Also, it may happen that for two disjoint clusters $C, C'$ we have $B(C)\cap B(C')\not=\emptyset$.
However, we do have the following important properties of these squares, which are easy to verify:
\begin{enumerate}
\hitem{i.a}{1} if $C'\subsetneqq C$ is a subcluster of $C$, then $\bar B(C')\subseteq B(C)$;
\hitem{i.b}{2} if $C$ and $C'$ are disjoint clusters, then $\bar B(C)\cap \bar B(C')=\emptyset$;
\hitem{i.c}{3} $B(C)$ depends only on $C$; and
\hitem{i.d}{4} $\bar B(C)$ depends only on $B(C)$ and $j(C^p)$.
\end{enumerate}


Define $A(C)=A_S(C):= \bar B(C)\setminus B(C)$. Note that if $A(C)\not=\emptyset$, then $|C|2^{j(C)+4}r < aR/4$;
therefore, using that $C\subseteq U'_r$, that the distance between $z$ and $z'$ is at most $2^j r$, that $\bar B(C)$ has sidelength
at most $aR/4$, and that the distance of $U'$ from $\partial U$ is at least $aR$, we get that $A(C)\subseteq U$.
This means that the annulus $A(C)=A_S(C)$ we have constructed is disjoint from $S$, both inside and outside $U$. 

Define an annulus structure $\Ann_1(S)$ associated with $S$ by
$$
\Ann_1(S):= \bigl\{ A_S(C) : C\text{ is a cluster in }S,\, A_S(C)\ne \emptyset \bigr\}.
$$
By properties \iref{i.a} and \iref{i.b} above, the different annuli in $\Ann_1(S)$ are disjoint. See Figure \ref{fig:annuli}. It is also clear that $\Ann_1(S)$
is compatible with $S$. However, we still need to modify $\Ann_1(S)$ for it to be useful.

\begin{figure}[htbp]
\centerline{\epsfysize=2 true in \epsffile{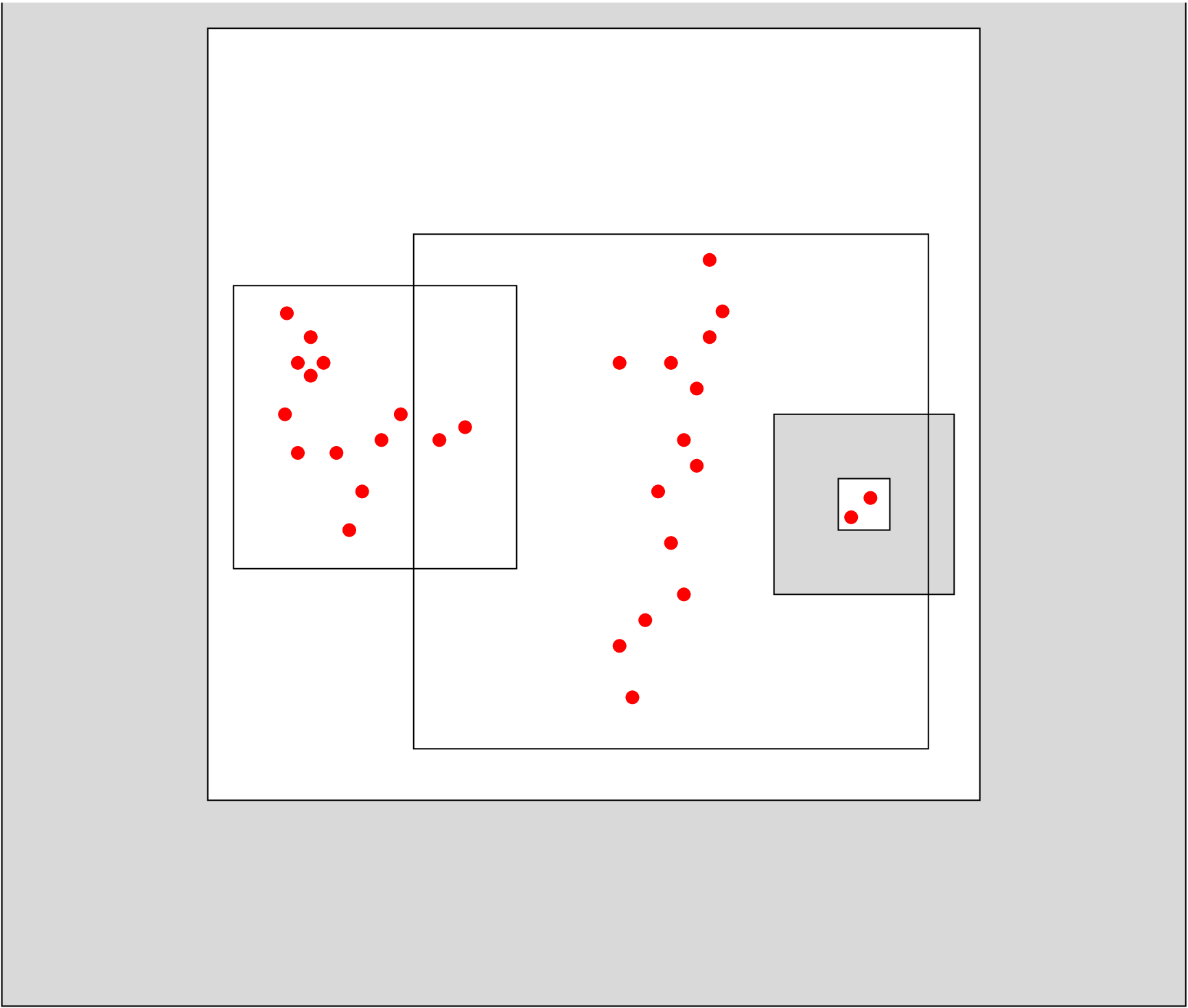}} 
\caption{ \label{fig:annuli}
A cluster with three child clusters, two of which have empty annuli; only the inner bounding squares are shown for those two.}
\end{figure}

It follows from~\eqref{e.a4less2} that
the function $\gamma_r(2^j\,r)$ is decaying exponentially, in the sense that there are
absolute constants $c_0,c_1>0$ such that 
\begin{equation}
\label{e.gamdecay}
\gamma_r(2^{j'}r)\le c_0\, 2^{-c_1 (j'-j)}\,\gamma_r(2^j\,r)\qquad\text{if }j'> j\ge 0\,.
\end{equation}
It would be rather convenient in the proof below to have $\gamma_r(2^{j+1}\,r)\le \gamma_r(2^j\,r)$.
However, we do not want to try to prove this. Instead, we use the function 
$\bgamma_r(\rho):=\inf_{\rho'\in[r,\rho]}\gamma_r(\rho')$ in place of $\gamma$.
Clearly, $\bgamma$ also satisfies~\eqref{e.gamdecay} and
$\bgamma_r(\rho)\le\gamma_r(\rho)\le O(1)\,\bgamma_r(\rho)$.

A cluster $C$ will be called {\bf overcrowded} if  
$$
g(|C|)\, \bgamma_r(2^{j(C)}\,r) > 1\,, 
$$
with the constant $\coa>0$ in the definition of $g(k)$ to be determined later.
In particular, clusters at level $0$ are overcrowded.
 For each overcrowded cluster, we remove from $\Ann_1(S)$ all the annuli corresponding to its proper subclusters. The resulting annulus structure will be denoted $\Ann(S)$, still compatible with $S$. Finally, we set
$\ANN=\ANN(r,k,U,U'):=\{\Ann(S) : S\in\mathcal S(r,k)\}$.
\bigskip

We will show that, for some constant $c_0>0$,
\begin{equation}\label{e.ANNnk}
\sum_{\Ann \in \ANN} h(\Ann)^2 \leq O(1)\,(k/a)^{c_0}\, g(k) \; \bgamma_r(R)\,.
\end{equation}
(Note that there are usually a lot of sets $S$ with the same $\Ann(S)$ in $\ANN$, or even with the same $V_S$. Indeed, a main point of our construction of $\ANN$ is to use as few annulus structures as possible. In particular, the above sum is really over different annulus structures, without multiplicities coming from the different sets $S$.)
This and~\eqref{e.ANN} together imply Proposition~\ref{pr.local}, with possibly
a different choice for the constant $\coa$. 

For each $S$, there is a natural tree structure on the clusters that correspond to the annuli of $\Ann(S)$.
 The root is $V$ itself, and
the parent $C^p$ of each cluster $C\ne V$ is also its parent in the tree. The leaves are the overcrowded clusters.
We will use this tree structure to build the bound (\ref{e.ANNnk}) inductively: we will write each term $h(\Ann)^2$ as a product of weights corresponding to smaller annulus structures, with the trivial annulus structures of overcrowded clusters as the base step.
(The effect of doing this induction on the notation is that the letter $k$ (or $k_i$) will be used not only for the size of $V_S$, but of subclusters, too.)
The reason for introducing the annulus structures $\Ann(S)$ instead of $\Ann_1(S)$ is that inside overcrowded clusters, the factors we would get from extra annuli would not balance the number of possible ways they can be added, so it  seems better to use no annuli for them at all.

\bigskip

For a cluster $C$ of $S$ let $\Ann'(C)$ denote the subset of $\Ann(S)$ corresponding to the proper subclusters of $C$.
Note that $\Ann'(C)$ depends only on $C$; that is, it is not affected by a modification of $S$,
as long as $C$ remains a cluster of $S$ and it does not acquire or lose an overcrowded ancestor.
Also note that $\Ann'(V)=\Ann(S)\setminus \{A_S(V)\}$, where $V=V_S$.

 Fix some $j,k\in\N_+$.  If $B=B(C)$, where $|C|=k$ and $j(C)=j$, then
  the coordinates of the four corners of $B$ are divisible by $2^j\,r$, and its side-length is $k\,2^{j+4}r$.
 For such $B$ (which might correspond to more that one pair of $k$ and $j$), define
$$
\ANNp(B,k,j):=\Bigl\{\Ann'(V_S) : S\in\mathcal S(r,k),\, B(V_S)=B,\,j(V_S)=j\Bigr\},
$$
and
$$
H(j,k):= \sup_B\,\sum_{\Ann\in \ANNp(B,k,j)} h(\Ann)^2.
$$
(Note that the sum on the right may depend on $B$, since it may happen,
for example, that $B\not\subseteq U'$.)

Define $J(k):=\max\{j\in\N: g(k)\,\bgamma_r(2^j\,r)>1\} = O(\coa)\, \log_2^2 (k+2)$.
If $j\leq J(k)$, then every $S\in\mathcal S(r,k)$ with $j(V_S)=j$ and $B(V_S)=B$ has $V_S$ overcrowded, hence $\Ann'(V_S)=\emptyset$. Therefore, if there exist such sets $S$, then we have $\ANNp(B,k,j)=\{\emptyset\}$ and thus $H(j,k)=1$; otherwise,
we have $\ANNp(B,k,j)=\emptyset$ and thus $H(j,k)=0$. That is,
\begin{equation}\label{e.smallj}
\forall j\le J(k)
\qquad
 H(j,k)\le 1
\,.
\end{equation}
This is the bound we will use for the overcrowded clusters at the base steps of the induction.

For general $j$ and $k$, we will show by induction on $j$ (for all $j\in\N$) that 
\begin{equation}\label{e.induct}
\forall k\geq 1
\qquad
H(j,k) \leq \bigl(g(k) \; \bgamma_r(2^j\,r)\bigr)^{1.99}\,.
\end{equation}

\bigskip

Before proving (\ref{e.induct}), let us demonstrate that it implies (\ref{e.ANNnk}).
If $j(V_S)=j$ (and $S\in\mathcal S(r,k)$), then the number of possible choices for $B(V_S)$ is at most
$\bigl(2R/(2^j\,r)\bigr)^2$. 
If $A(V_S)\ne\emptyset$, then the probability of the 4-arm event in $A(V_S)$ is at most
$O(1)\,\alpha_4\bigl(k\,2^{j}\,r,a\,R\bigr)$, by (\ref{e.qRSW}). If $A(V_S)=\emptyset$, then the probability of the 4-arm event is 1, 
while $k\,2^{j+4}\,r \geq aR/4$, hence the previous bound is $O(1)$, and thus remains true. Therefore,
$$
\sum_{\Ann \in \ANN} h(\Ann)^2 \leq \, O(1) \, \sum_{j=0}^{\lceil \log_2(2R/r)\rceil} 
\left(\frac{R}{2^jr}\right)^2\alpha_4\bigl(k\,2^{j}\,r,a\,R\bigr)^2\,H(j,k)\,.
$$
Now, we use the quasi-multiplicativity of the $4$-arm event, (\ref{e.qRSW}), \eqref{e.smallj} and \eqref{e.induct} 
 to rewrite this as
\be
\sum_{\Ann \in \ANN} h(\Ann)^2 
&\le&
 \sum_{j=0}^{\lceil\log_2 (2R/r)\rceil} O(1)\, (k/a)^{c_0}\frac{\bgamma_r(R)}{\bgamma_r(2^jr)}\,H(j,k)
\label{e.ANNjsum}\\
& &\hskip -1.5 cm \le \ 
O(1)\,(k/a)^{c_0}\,
 \bgamma_r(R) \Bigl(
\sum_{j\leq J(k)}\frac{1}{\bgamma_r(2^j r)}  + 
g(k)^{1.99} \, \sum_{j > J(k)} \bgamma_r(2^j\,r)^{0.99} \Bigr),\nonumber
\ee
for some finite constant $c_0$.
Because~\eqref{e.gamdecay} holds for $\bgamma$,
the first sum is dominated by a constant times the summand corresponding to $j=J(k)$ and
the second sum is dominated by a constant times the summand corresponding to $j=J(k)+1$.
Since $\bgamma_r(2^{J(k)}\,r) > 1/g(k)$ and $\bgamma_r(2^{J(k)+1}\,r) \le 1/g(k)$,
 we get the bound claimed in (\ref{e.ANNnk}).

\bigskip

In order to complete the proof of Proposition~\ref{pr.local}, all that remains is to establish~\eqref{e.induct}.
The proof of this inequality will be inductive and somewhat similar to the proof of~\eqref{e.ANNnk}
from~\eqref{e.induct}, but there are some important differences. In passing from~\eqref{e.induct}
to~\eqref{e.ANNnk}, we ``used up'' some of the exponent $1.99$ --- it has become $1$. Such a loss would not be sustainable if it had to be repeated in every inductive step. What saves us in the following proof of~\eqref{e.induct}
is that clusters that are not singletons have more than one child cluster. (In other words, any non-leaf vertex of the cluster tree has at least two children.) This results in almost squaring the estimate at each inductive step, which makes the proof work.

We now proceed to prove~\eqref{e.induct} by induction.
Since $\bgamma$ is non-increasing, the claim holds for all $j\le J(k)$, because of~\eqref{e.smallj}. (In particular, for $j=0$ and any $k\geq 1$.) 
We now fix some $j$ and $k$ with $j>J(k)$, and assume that~\eqref{e.induct} holds for all smaller values of $j$.
 Fix some square $B$ such that $\ANNp(B,k,j)\ne\emptyset$.
Observe that if $B(V_S)=B$, then $S$ has some $d=d(S)\ge 2$
children in its tree of clusters (we know $d\ne 0$ because $V_S$ is not overcrowded, since $j>J(k)$). Fix some $d\in\{2,3,\dots\}$, $k_1,\dots,k_d$, $j_1,\dots,j_d$ and sub-squares $B_1,\dots,B_d$ such that there is some $S\in\mathcal S(r,k)$
with $j(V_S)=j$, $B(V_S)=B$, and having cluster children $C^1,\dots,C^d$ with $|C^i|=k_i$, $j(C^i)=j_i$ and $B(C^i)=B_i$. 
Note that $A_i:=A_S(C^i)$ does not depend on the choice of $S$ satisfying the above
(by property~\iref{i.d} of squares noted previously and by having a fixed $j$).

Let $\ANNpp=\ANNpp(d,k_1,\dots,k_d,j_1,\dots,j_d,B_1,\dots,B_d)$ denote the set of all elements
of $\ANNp(B,k,j)$ that arise from such an $S$.
Note that every $\Ann\in\ANNpp$ is of the form
$\{A_1,\dots,A_d\}\cup\bigcup_{i=1}^d\Ann_i$,
where the $\Ann_i\in \ANNp(B_i,k_i,j_i)$ are $d$ disjoint annulus substructures, and the $A_i$'s are fixed by $\ANNpp$. (It is possible that some of the $A_i$ here are empty annuli, see, e.g., Figure~\ref{fig:annuli}.) For such an  $\Ann$, we have by~\eqref{e.hann}
$$
h(\Ann)= \prod_{i=1}^d \bigl(h(A_i)\,h(\Ann_i)\bigr)\,.
$$
Hence,
\begin{equation*}
\begin{aligned}
\sum_{\Ann\in\ANNpp} h(\Ann)^2
&
\le \sum_{\Ann_1\in\ANNp(B_1,k_1,j_1)}\;
\sum_{\Ann_2\in\ANNp(B_2,k_2,j_2)}
\cdots
\sum_{\Ann_d\in\ANNp(B_d,k_d,j_d)}\;
\prod_{i=1}^d \bigl(h(A_i)\,h(\Ann_i)\bigr)^2
\\&
=
\prod_{i=1}^d\Bigl(h(A_i)^2\,
\sum_{\Ann\in\ANNp(B_i,k_i,j_i)}
h(\Ann)^2
\Bigr)
\\&
=
\prod_{i=1}^d\bigl(h(A_i)^2\,H(j_i,k_i)\bigr) .
\end{aligned}
\end{equation*}
Now, $h(A_i)=O(1)\,\alpha_4(k_i\,2^{j_i}\,r,2^{j}\,r)$, where we use the convention
that $\alpha_4(\rho,\rho')=1$ if $\rho\ge \rho'$.
Furthermore, given $B$, $d$, $j_1,\dots,j_d$ and $k_1,\dots,k_d$, there are no more than $O(1)\,(k\,2^{j-j_i})^2$ possible choices for $B_i$.
Hence, summing the above over all such choices, we get
$$
H(j,k)
\le
\sum_{d=2}^k \sum_{(j_1,\dots,j_d) \atop (k_1,\dots,k_d)} 
\prod_{i=1}^d
\Bigl(
O(k\,2^{j-j_i})^2\,
\alpha_4(k_i\,2^{j_i}\,r,2^{j}\,r)^2\, H(j_i,k_i)\Bigr).
$$
The inductive hypothesis~\eqref{e.induct} for each of the pairs $(j_i,k_i)$ implies
$H(j_i,k_i)\le \bgamma_r(2^{j_i}\,r)\,g(k_i)$ when $j_i> J(k_i)$, since we decreased the exponent from 1.99 to 1, with a base less than 1. This upper bound also holds when $j_i\le J(k_i)$, because of~\eqref{e.smallj}. 
We now use this, together with the quasi-multiplicativity and the RSW estimate (\ref{e.qRSW}), as before,
to obtain 
\begin{equation}
\label{e.gammasum}
\begin{aligned}
H(j,k)
&
\le 
\sum_{d=2}^k \sum_{(j_1,\dots,j_d) \atop (k_1,\dots,k_d)} 
\prod_{i=1}^d
\Bigl(O(k)^2\, k_i^{O(1)}\,\frac{\bgamma_r(2^jr)}{\bgamma_r(2^{j_i}r)}\, \bgamma_r(2^{j_i}\,r)\,g(k_i)\Bigr)
\\&
\le
\sum_{d=2}^k \sum_{(j_1,\dots,j_d) \atop (k_1,\dots,k_d)} 
\prod_{i=1}^d
\Bigl(O(k)^{O(1)}\,\bgamma_r(2^j\,r)\,g(k_i)\Bigr)
\\&
\le
\sum_{d=2}^k \bigl(O(k)^{O(1)}\,j\,\bgamma_r(2^j\,r)\bigr)^d
 \sum_{(k_1,\dots,k_d)}\prod_{i=1}^d g(k_i)\,.
\end{aligned}
\end{equation}
Since $\log_2^2(x+2)$ is concave on $[0,\infty)$, it follows that when $k_1+\cdots+k_d=k$,
we have $\prod_{i=1}^d g(k_i)\le g(k/d)^d$. 
Since for a fixed $d$ the number of possible $d$-tuples $(k_1,\dots,k_d)$ is clearly bounded by $k^d$, the above gives
\begin{equation*}
H(j,k)
\le
\sum_{d=2}^k \bigl(c\,k^c\,j\,\bgamma_r(2^j\,r)\,g(k/d)\bigr)^d,
\end{equation*}
for some constant $c$.
Noting that $g(k/d)\le g(k/2)$, and setting 
$$
\Phi=\Phi(j,k):=c\,k^{c}\,j\,\bgamma_r(2^j\,r)\,g(k/2),
$$
we then get $ H(j,k) \le \Phi^2+\Phi^3+\Phi^4+\cdots= \Phi^2/(1-\Phi)$, provided that $\Phi<1$.
Consequently, the proof of~\eqref{e.induct} and of Proposition~\ref{pr.local} are completed by the following lemma.
\QED

\begin{lemma}
For all $\eps\in(0,1/2)$, if $\coa$ in the definition of $g$ is chosen sufficently large (depending
only on $\eps$),
then for all $k\in\N_+$ and all $j>J(k)$, 
\begin{equation}
\label{e.Aest}
\Phi(j,k) < \bigl(g(k)\,\bgamma_r(2^j\,r)\bigr)^{1-\eps}/2
\end{equation}
and $\Phi(j,k)<1/2$. 
\end{lemma}

\begin{proof}
The estimate $\Phi<1/2$ follows from~\eqref{e.Aest}, since $g(k)\,\bgamma_r(2^j\,r)\le 1$.
Write
\begin{equation}
\label{e.Qrat}
\Phi/\bigl(g(k)\,\bgamma_r(2^j\,r)\bigr)^{1-\eps} =
c\,k^{c}\,\bgamma_r(2^j\,r)^\eps\,j\,g(k/2)^\eps\,\Bigl(g(k/2)/g(k)\Bigr)^{1-\eps}.
\end{equation}
Since~\eqref{e.gamdecay} holds for $\bgamma$ and
$\bgamma_r(2^{J(k)+1}\,r)\,g(k)\le 1$, we have
$\bgamma_r(2^j\,r)^\eps g(k/2)^\eps \le
\bgamma_r(2^j\,r)^\eps g(k)^\eps \le
O(1)\, 2^{-\eps\,c_1(j-J(k))}$.
Hence,
$j\,\bgamma_r(2^j\,r)^\eps g(k/2)^\eps \le O_\eps\bigl(J(k)+1\bigr)$.
As we noted before, $J(k)=O(\coa)\,\log_2^2(k+2)$.
Hence,~\eqref{e.Qrat} gives
$$
\Phi/\bigl(g(k)\,\bgamma_r(2^j\,r)\bigr)^{1-\eps} 
\le
O_\eps(\coa)\,\log_2^2(k+2)\,k^{c}\,\Bigl(g(k/2)/g(k)\Bigr)^{1/2}.
$$
It is easy to verify that the right hand side tends to $0$ as $\coa\to\infty$,
uniformly in $k\in\N_+$. This completes the proof. \qed
\end{proof}

For later use, let us point out that having an exponent larger than 1 in (\ref{e.induct}) was important 
when we derived the bound \eref{e.ANNnk}, but not for the induction.
In (\ref{e.gammasum}), we used only the bound (\ref{e.induct}) with the exponent 1. 
This was sufficient because of $d\ge 2$.

\begin{remark}\label{r.diam}
Our proof shows a {\bf clustering effect} for spectral samples of 
very small size. Firstly, a positive proportion of our main upper bound~(\ref{e.ANNnk}) comes from annulus structures with a single overcrowded cluster, as it is clear from~(\ref{e.ANNjsum}). Moreover, the contribution from sets with large diameter is small: for any $d\le \frac R r$, 
the calculation in (\ref{e.ANNjsum}) implies immediately that if $k\leq O(1)\log_2 d$, then 
$$
\PB{\Spec\in\mathcal S(r,k),\ \text{diam}(\Spec) > rd} \leq \gamma_r(R)\,\gamma_r(rd)^{1+o(1)}\,d^{\,o(1)}
$$
(of course, the exponent $0.99$ in~\eqref{e.ANNjsum} can be modified to $1+o(1)$).
Recall that (\ref{e.gamdecay}) says $\gamma_r(rd) < O(1)\,d^{-c_1}$ for some $c_1>0$.
Since we will see in Section~\ref{s.mainlemma} that Lemma~\ref{l.basic} is sharp, and will handle the boundary issues in Subsection \ref{ss.boundary}, we easily obtain that $\Pb{1\le |\Spec_r|\le k}\ge \Pb{ |\Spec_r| =1} \ge O(1) \gamma_r(R)$. Therefore the above bound gives for $k\leq O(1)\log_2 d$ that 
$$
\PB{\text{diam}(\Spec) > rd \, \Big| \, 1\leq |\Spec_r|\leq k} \leq \gamma_r(rd)^{\,1+o(1)}\,.
$$ 
To illustrate this formula (with $r=1$), if one is looking for spectral samples of size less than $\log R$,
then they have small diameter:
$$
\PB{\text{diam}(\Spec) > R^{\alpha} \, \Big| \, 1\leq |\Spec|\leq \log R} \le \gamma(R^{\alpha})^{1+o(1)},
$$
which for the triangular lattice gives $R^{-\alpha/2+o(1)}$.
\end{remark}

\begin{remark}\label{r.annpivo} 
Our strategy proving Proposition \ref{pr.verysmall} also works for pivotals, showing 
$$
\Pb{|\Piv_r|=k} \leq g(k)\,\alpha_6(r,\nn)\,(\nn/r)^2\,.
$$
The only difference is that we need to replace the factor $\alpha_4(r,\nn)^2$, coming from Lemma~\ref{l.basic} and its generalization Lemma~\ref{l.Comph}, with $\alpha_6(r,\nn)$. The reason for this factor is that having pivotals in the inner disk of an annulus but no pivotals in the annulus itself corresponds to the 6-arm event in the annulus. On $\Z^2$ it is known that $\alpha_6(r,\nn)\,(\nn/r)^2 \leq O(1)\,(r/\nn)^\eps$ for some $\eps>0$ \cite[Corollary A.8]{\SchrammSteif}; on the triangular lattice,
$\alpha_6(r,\nn)\,(\nn/r)^2=(r/\nn)^{11/12+o(1)}$ \cite{\SmirnovWerner}.
Thus the clustering effect for pivotals is expressed here in the following way:
$$
\PB{\text{diam}(\Piv) > R^{\alpha} \, \Big| \, 1\leq |\Piv|\leq \log R} \le R^{2\alpha} \alpha_6(R^{\alpha})^{1+o(1)},
$$
which for the triangular lattice gives $R^{-\frac {11}{12}\alpha +o(1)}$.
\end{remark}

\subsection{Handling boundary issues}\label{ss.boundary}

\proofof{Proposition~\ref{pr.verysmall}}
Since the proof is rather similar to that of Proposition~\ref{pr.local}, 
we will just indicate the necessary modifications.

First of all, we need the {\bf half-plane} and {\bf quarter-plane} {\bf $j$-arm events}. So let   $\alpha_j^+(r,R)$ be the probability of having $j$ disjoint arms of alternating colors connecting $\p B(0,r)$ to $\p B(0,R)$ inside the half-annulus $(B(0,R)\setminus B(0,r)) \cap (\R\times\R_+)$. Similarly, $\alpha_j^{++}(r,R)$ is the probability of having $j$ arms of alternating colors connecting $\p B(0,r)$ to $\p B(0,R)$ inside the quarter-annulus $(B(0,R)\setminus B(0,r)) \cap (\R_+\times\R_+)$. As before, we let $\alpha_j^+(R):=\alpha_j^+(2j,R)$ and similarly for $\alpha_j^{++}$. The RSW and quasi-multiplicativity bounds (\ref{e.RSW}, \ref{e.qm}) hold for these quantities, as well.

The reason for these definitions is that if $x$ is a point on one of the sides of $[0,\nn]^2$ such that its distance from the other sides is at least $r'$, then the percolation configuration inside $B(x,r)$ has an effect on the crossing event only if we have the 3-arm event in the half-annulus $(B(x,r')\setminus B(x,r)) \cap [0,\nn]^2$. Similarly, if $x$ is one of the corners of $[0,\nn]^2$, and $r' \leq \nn$, then we need the 2-arm event in the quarter-annulus $(B(x,r')\setminus B(x,r)) \cap [0,\nn]^2$ in order for the configuration inside $B(x,r)$ to have any effect.

In the annulus structures we are going to build, we will have to consider clusters that are close to a side or even to a corner of $[0,\nn]^2$. These will be called side and corner clusters (defined precisely below). To understand the contribution of such clusters, we define 
$$\gamma^+_r(\rho):= (\rho/r)\, \alpha_3^+(r,\rho)^2 \quad\text{and}\quad \gamma^{++}_r(\rho):=\alpha_2^{++}(r,\rho)^2.$$
The function $\gamma^+_r$ plays a role similar to $\gamma_r$, but in relation
to the side clusters. Similarly, $\gamma^{++}_r$ relates to the corner clusters.
The motivation for the linear prefactor of $\rho/r$ in the definition of
$\gamma^+_r$ is that (up to a constant factor) the number of different ways to
choose a square of some fixed size whose center is on a line segment of length $\rho$ and whose position on the line segment is divisible by $\rho'$ is $\rho/\rho'$, when $\rho>\rho'>0$.
Such a prefactor is not necessary in the case of $\gamma^{++}_r$, because
it corresponds to a corner, and there is no choice in placing a square of a fixed size into a fixed corner.

As we will see, the key reason for the boundary $\p[0,\nn]^2$ to play no significant role in the behavior of the spectral sample is that there is a $\delta>0$ and a constant $c$ such that when $r\le\rho'\le\rho$
\begin{equation}\label{e.gammadelta}
\frac{\gamma^{+}_r(\rho)}
{\gamma^{+}_r(\rho')}
\leq c\,
\Bigl(\frac{\gamma_r(\rho)}
{\gamma_r(\rho')}\Bigr)^{1+\delta} \quad\text{and}\quad 
\frac{\gamma^{++}_r(\rho)}
{\gamma^{++}_r(\rho')}
 \leq c\,
\Bigl(\frac
{\gamma_r(\rho)}
{\gamma_r(\rho')}\Bigr)
^{1+\delta}.
\end{equation}

We now prove these inequalities. Firstly, it is known
that
\begin{equation}
\label{e.a+3}
\alpha^+_3(\rho',\rho) \asymp (\rho'/\rho)^2,
\end{equation}
see \cite[First exercise sheet]{\WWperc}. Hence $\gamma^{+}_r(\rho)/\gamma^+_r(\rho') \asymp (\rho'/\rho)^3$. Secondly, observe that 
\begin{equation}
\label{e.a++2}
\alpha_2^{++}(\rho',\rho)^2 \leq \, \alpha_4^{+}(\rho',\rho) \leq \, \alpha_3^{+}(\rho',\rho)\,\alpha_1^{+}(\rho',\rho)\leq O(1)\,\alpha_3^{+}(\rho',\rho) \, (\rho'/\rho)^\eps,
\end{equation}
with some $\eps>0$, where the second step used Reimer's inequality \cite{\Reimer} (or color-switching and the BK inequality), and the third step used RSW for the 1-arm half-plane event. Therefore, $\gamma^{++}_r(\rho)/\gamma^{++}_r(\rho') \leq O(1) (\rho'/\rho)^{2+\eps}$. On the other hand,~(\ref{e.a4less2}) implies that $\gamma_r(\rho)/\gamma_r(\rho') \geq (\rho'/\rho)^{2-\eps'}$ for some $\eps'>0$. Combining these upper and lower bounds we get~(\ref{e.gammadelta}).

Let us note that for the triangular lattice we actually know that
\begin{equation}\label{e.a++2tri}
\alpha_2^{++}(\rho',\rho) =  (\rho'/\rho)^{2+o(1)},
\end{equation}
since $\alpha_2^{++}(\rho',\rho)=\alpha_2^{+}(\rho',\rho)^{2+o(1)}$ by the conformal invariance of the scaling limit, while $\alpha_2^{+}(\rho',\rho)\asymp \rho'/\rho$ is known on both lattices by RSW arguments, see again \cite[First exercise sheet]{\WWperc}.
\bigskip

After these preparations, we define $\mathcal S(r,k)$ as the set of all $S\subseteq\I$ such that
$|S_r|=k$. The set $V=V_S$ is defined as $S_r$.
As it turns out, we will need to limit the diameters of the clusters.
For that purpose, set
$\bar j=\bar j_k:=\lfloor \log_2(\nn/(k\,r))\rfloor -5$ and $J:=\{0,1,\dots,\bar j\}$.
Clearly, we may assume without loss of generality that $\bar j>0$.
The clusters at level $0$ are once again the sets of the type $C=\{x\}$, where $x\in V$.
If $j\in J$, then an {\bf interior cluster} at level $j$ is defined as
a connected component $C\subseteq V$ of $G_j$
such that the distance from $C$ to $\p ([0,\nn]^2)$ is larger than $2^j\,r$
and $C$ is not connected in $G_{j-1}$.
The interior clusters at level $0$ are just the connected components of
$G_0$; that is, the singletons in $V$. 
Inductively, we define the {\bf side clusters}:
a connected component $C\subseteq V$ of $G_j$ is
a side cluster at level $j\in J$ if it is within
distance $2^j\,r$ of precisely one of the four boundary edges
of $[0,\nn]^2$ and it is not a side cluster at any level
$j'\in\{1,2,\dots,j-1\}$.
Likewise, a {\bf corner cluster} at level $j\in J$ is a
connected component $C\subseteq V$ of $G_j$ that is within distance $2^j\,r$ of precisely two adjacent
boundary edges of $[0,\nn]^2$, but is not a corner cluster at any level $j'\in\{1,2,\dots,j-1\}$.
Finally, the unique {\bf top cluster} has level $\bar j+1$ (by definition) and consists of all of $V$.
With these choices, when $j\in J$ every connected component of $G_j$ is either
an interior, side, or corner cluster, and this is the reason for setting the upper bound $\bar j$.

Note that the top cluster contains at least one cluster (which could be a
side cluster or a corner cluster or an interior cluster),
a corner cluster contains at least one side or interior cluster, and possibly
also a corner cluster, a side cluster contains
an interior cluster or at least two side clusters (but no corner clusters),
and an interior cluster at level $j>0$
contains at least two interior clusters (and no side or corner clusters).

The inner and outer bounding squares associated with the clusters are defined 
as before, except that the squares associated with side clusters
are centered at points on the corresponding edge
and squares associated with corner clusters are centered at the corresponding corner,
which is the meeting point of the two sides of $[0,\nn]^2$ closest to the
cluster.  There are no squares associated with the top cluster.
The annulus associated with each cluster
(other than the top cluster, which does not have its own annulus)
is the annulus between its outer square and its inner square, provided that the outer square
strictly contains the inner square.


We define $\bgamma_r^+(\rho):=\inf_{\rho'\in[r,\rho]}\gamma^+_r(\rho)$ and similarly for $\bgamma_r^{++}$.
The exponential decay (\ref{e.gamdecay}) holds for these functions as well. As before, clusters (even side or corner clusters) are considered overcrowded if they satisfy $g(|C|)\,\gamma_r(2^j\,r)>1$, and the annulus structure $\Ann(S)$ is defined as above.

There are some modifications necessary in the definition of $h(\Ann)$ in order
for~\eqref{e.ANN} to still hold in the present setting.
For a side annulus $A$ define $h(A)$ as the probability
of the $3$-arm crossing event within $A\cap[0,\nn]^2$ between
the two boundary components of $A$, and for a corner annulus
define $h(A)$ as the probability of the $2$-arm crossing event
within $A\cap[0,\nn]^2$ between the two boundary components of $A$.
With these modifications,~\eqref{e.ANN} still holds, and the proof is
essentially the same. 

The definition of $H(j,k)$ is similar to the one in the proof of Proposition~\ref{pr.local},
but now the supremum only refers to interior squares.
(In the definition of $ \ANNp(B,k,j)$, we presently restrict
to such $S$ so that $V_S$ is an interior cluster at level $j$,
in addition to being the top cluster, and when we write $B(V_S)$, 
it is understood as the inner square defined for an interior cluster.)
We also define $H^+(j,k)$, $H^{++}(j,k)$, which refer to the supremum over
side and corner squares, respectively. We also set 
$$
H^\square(\nn,k):= \sum_{S\in\mathcal S(r,k)} h\bigl(\Ann(S)\bigr)^2\,,
$$
and our goal is to show that this quantity is at most $g(k)\,\gamma_r(\nn)$.

We first prove a similar bound on $H^+(j,k)$. It is convenient to
separate the annulus structures $\Ann'(S)$ where $B(V_S)=B$ and $B$ is
a side square into those where $V_S$ has a single child cluster and
into those where $V_S$ has at least two child clusters.
In the case of a single child cluster, that child has to be an interior
cluster, and using~\eqref{e.induct}, the argument giving~\eqref{e.ANNnk}
now gives the bound
$$
\sum_{\Ann} h(\Ann)^2\le
O(1)\,k^{O(1)} g(k) \; \bgamma_r(2^j\,r)\,,
$$
where the sum extends over such (single interior cluster child) annulus structures.
By increasing the constant $\coa$ in the definition of $g$, we may then incorporate
the factor $O(1)\,k^{O(1)}$ into $g$.
The bound on the sum over the annulus structures with at least two child clusters 
can now be established by induction on $j$, in almost the same way that~\eqref{e.induct}
was proved by induction. The main difference here is that the children can
fall into two types, which slightly complicates the calculations but adds no
significant difficulties.
(Indeed, since each square among $B_i$ at level $j_i$ ($i=1,2,\dots,d$)
can either correspond to a side cluster or an interior cluster, each factor
in the first row of~\eqref{e.gammasum} presently needs to be replaced by
$$
O(k)^2\, k_i^{O(1)}\Bigl(
\frac{\bgamma_r(2^jr)}{\bgamma_r(2^{j_i}r)}
+
\frac{\bgamma^+_r(2^jr)}{\bgamma^+_r(2^{j_i}r)}
\Bigr)\bgamma_r(2^{j_i}\,r)\,g(k_i)\,.
$$
Thanks to~\eqref{e.gammadelta}, the fraction featuring $\bgamma^+$ is dominated
by a constant times the other fraction and is therefore certainly
inconsequential.)
A point worth noting is that in the inductive prove of~\eqref{e.induct} we only used
the inductive hypothesis with exponent $1$ in place of $1.99$.
In our present situation, this is what we have at our disposal in the base step of the induction, since the induction now has to be started from overcrowded clusters or side clusters with a single child cluster. 

Summing up the two types, we conclude (by changing $\coa$ again, if necessary) that
$H^+(j,k)\le g(k)\,\bgamma_r(2^j\,r)$.
Of course, in this type of argument we should not change
$\coa$ at every induction step, for then it may end up depending on $\nn$.
But we have not committed any such offence.

We can show $H^{++}(j,k) \leq g(k)\,\bgamma_r(2^j\,r)$ similarly. We separate the annulus structures
into those with a single child at the top level that is an interior cluster,
those with a single child that is a side cluster, and those with several
children at the top level. The first type is handled as in the bound for $H^+(j,k)$.
The second type is handled similarly, but now we do not have the exponent $1.99$ as
in~\eqref{e.induct}, but only the exponent $1$ that we showed for $H^+(j,k)$. 
So, we use instead the fact that $\delta>0$ in~\eqref{e.gammadelta}.
The argument bounding the third type uses induction as in the
multi-child case of $H^+(j,k)$.

Finally, the bound for $H^\square(\nn,k)$ follows in the same way, using our previous bounds for $H^+(j,k)$ and $H^{++}(j,k)$ together with~\eqref{e.gammadelta}. The last small difference is that the child clusters of the top cluster (at some levels $j_i$) have outer bounding squares of size $\asymp r2^{\bar j}$, but the number of ways to place each of these clusters is $\asymp (\nn/(r2^{j_i}))^2$, instead of $\asymp (2^{\bar j}/2^{j_i})^2$. But $\nn/(r2^{\bar j})\asymp k$, so this discrepancy gives only an $O(k^2)$ factor for each child cluster, which can be absorbed into $g(k)$ in the usual way.
This completes the proof.
\QED

\subsection{The radial case}\label{ss.smallradial}

In this subsection, we will consider the spectral sample $\Spec=\Spec_f$, where $f$ is the indicator function
(not the $\pm1$-indicator function) of the $\ell$-arm event in the
annulus $[-R,R]^2\setminus[-\ell,\ell]^2$ and
$\ell=1$ or $\ell\in 2\,\N_+$. Thus, $\Eb{f}=\Eb{f^2}\asymp \alpha_\ell(R)$. Instead of the probability measure for $\Spec$ that we have worked with so far, it will be easier notationally to use the un-normalized measure $\Qb{\Spec=S}:=\widehat f(S)^2$.

For any $S\subset \R^2$, we let $S^*:=S \cup \{0\}$, 
  and define $S_r$ as before. In particular, $\Spec^*_r$ is the set of $r$-squares whose intersection with $\Spec^*$ is nonempty. We are going to prove the following analogue of Proposition \ref{pr.verysmall}:

\begin{proposition}\label{pr.verysmallRad} 
Let $\ell=1$ or $\ell\in2\,\N_+$,
and let $\Spec$ denote the spectral sample of the
indicator function of the $\ell$-arm event
in the annulus $[-R,R]^2\setminus [-\ell,\ell]^2$.
Then there is some $\cod=\cod_\ell>0$ such that
$$
\Qb{|\Spec_r|=k} \leq g^*(k)\; \alpha_\ell(r,R)^2\, \alpha_\ell(r).
$$
holds with $g^*(k):=2^{\cod \log_2^2(k+2)}$
for all $k\in\N_+$ and all $R\ge r\ge \ell$.
\end{proposition}

It might be surprising at first glance that no four-arm probabilities show up in this upper bound. See~(\ref{e.envelope}) below for a rough explanation.

\proof The main difference from the square crossing case is that we will use {\bf centered annulus structures}, which have two kinds of annuli: annuli centered at the origin ($0$), for which we are interested in the $\ell$-arm event, and annuli with outer square disjoint from 0, for which we are interested in the 4-arm event.
  Each centered annulus structure is required to have an annulus centered at $0$
 whose inner square does not contain any other annuli.
The {\bf inner radius} of the annulus structure is defined as the inner radius of this innermost centered annulus.
For a centered annulus structure $\Ann$, we define $h^*(\Ann)$ to be the probability of having the 4-arm event in the annuli with outer square disjoint from $0$ and the $\ell$-arm event in the annuli centered at $0$. Now, we have the following analogue of Lemma \ref{l.Comph}. 

\begin{lemma}\label{l.ComphRad}
For any centered annulus structure $\Ann$ with inner radius $r_\Ann$, 
$$
\Qb{\Spec^* \text{ is compatible with }\Ann} \le \alpha_\ell(r_\Ann) \, h^*(\Ann)^2.
$$
\end{lemma}

\proof Similarly to the proof of Lemma \ref{l.Comph}, we divide the set
of relevant bits into parts: $\window$ is the configuration inside $\bigcup \Ann$, while $\eta_0$ is the configuration inside the inner disk of the smallest centered annulus, and $\eta_1$ is the configuration neither in $\window$ nor in $\eta_0$. As before, $F_\window$ is the function defined by $F_\window(\eta_0,\eta_1):=f(\window,\eta_0,\eta_1)$; furthermore, $W$ is the linear space of functions spanned by $\{\chi_S : S^* \text{ compatible with }\Ann\}$, and $P_W$ denotes the orthogonal projection onto $W$. Now, $P_W F_\window\ne 0$ implies the $4$-arm event in every 
interior non-centered $A\in\Ann$, the $\ell$-arm event in every centered $A\in\Ann$,
and the $3$-arm event in every boundary (or corner) annulus.
(Note that this uses the fact that when $\ell\ne 1$ we are considering the alternating arms event.
In particular, we are restricted to $\ell\in \{1\}\cup 2\,\N_+$.)
 Moreover, for any $\window$, we have $F_\window(\eta_0,\eta_1)=0$ if $\eta_0$ does not have the $\ell$-arm event. 
Thus, $\|P_W F_\window\|^2 \leq \|F_\window\|^2 = \Es{F_\window^2} \leq \alpha_\ell(r_\Ann)$. Altogether, similarly to the proof of Lemma \ref{l.Comph},
\begin{align*} 
\Qb{\Spec^*\text{ is compatible with }\Ann} &= \| P_W f \|^2  \leq  \E^\window\bigl[\| P_W F_\window \| \bigr]^2 \\
&\leq  \P^\window\bigl[ P_W F_\window \not= 0\bigr]^2\,\alpha_\ell(r_\Ann)  \leq  h^*(\Ann)^2\,\alpha_\ell(r_\Ann), 
\end{align*}
which proves the lemma.\qed

Note that a centered annulus structure compatible with $\Spec$ is also compatible with $\Spec^*$, but not necessarily vice versa, hence the lemma is stronger with $\Spec^*$ than it would be with $\Spec$. This strengthening is crucial, as shown by the following example: for an $r$-square $B$ at distance $t\in (r,R)$ from $0$, by the remark after Lemma~\ref{l.lambda}, we have $\Qb{\emptyset\not=\Spec \subseteq B} \leq O(1)\,\Pb{\Lambda_B}^2 \asymp \big(\alpha_\ell(1,R)\alpha_4(r,t)\big)^2$, a bound that we would not be able to reproduce from the weaker (starless) version of the above lemma. Furthermore, when $B$ is centered at $0$, then the bound $O(1)\,\alpha_\ell(r)\alpha_\ell(r,R)^2$ given by Lemma~\ref{l.ComphRad} is stronger than the $O(1)\,\alpha_\ell(r,R)^2$ bound of Lemma~\ref{l.lambda}. (Unlike Lemma~\ref{l.Comph}, which was only a generalization of Lemmas~\ref{l.lambda} and~\ref{l.basic}.)

These bounds provide a back-of-the-envelope explanation how the result of Proposition~\ref{pr.verysmallRad} arises, at least for $k=1$:
\begin{align}
\Qb{|\Spec_r|=1} &\leq O(1)\, \alpha_\ell(r)\alpha_\ell(r,R)^2 + O(1)\sum_{s=1}^{O(R/r)} s \cdot \big(\alpha_\ell(1,R)\alpha_4(r,s r)\big)^2 \nonumber \\ 
&\leq O(1)\, \alpha_\ell(r)\alpha_\ell(r,R)^2 + O(1)\,\alpha_\ell(1,R)^2\,, \label{e.envelope}
\end{align}
where we used that $s\,\alpha_4(r,sr)^2 \leq O(1)\,s^{-1-\eps}$, by~(\ref{e.a4less2}). The first term being dominant also shows that a small $\Spec_r$ should typically be close to $0$. (Which is another manifestation of the four-arm events playing a small role here.)

Back to the actual proof, analogously to Subsection \ref{ss.local}, for each $S\subset [-R,R]^2$ with $|S_r|=k$ we will build a centered annulus structure $\Ann(S)$ that is compatible with $S^*$ (but not necessarily with $S$ itself!)\ and that has $r_{\Ann(S)} \geq r$. Furthermore, the collection $\ANN^*(r,k)$ of all these centered annulus structures will be small enough to have
\begin{equation}\label{e.ANNRk}
\sum_{\Ann \in \ANN^*(r,k)} h^*(\Ann)^2 \leq g^*(k) \; \alpha_\ell(r,R)^2.
\end{equation}
The combination of (\ref{e.ANNRk}) with Lemma \ref{l.ComphRad} proves Proposition \ref{pr.verysmallRad}.

To construct the annulus structure $\Ann(S)$, we take $V=V_S$ to be $S^*_r$, set $\bar j:=\lfloor\log_2(R/(kr))\rfloor-5$, and define the clusters exactly as in Subsection \ref{ss.boundary}. A cluster is called centered if it contains 0. In constructing the inner and outer bounding squares, we use the additional rule that for centered clusters $C$ the inner bounding square must be centered at 0. (Note that this is just a special case of the ``arbitrary but fixed way'' of choosing a vertex $z \in [C]$.)

The centered analogue of $\gamma_r(\rho)$ is now $\gammas_r(\rho):=\alpha_\ell(r,\rho)^2$. We again let $\bgammas_r(\rho):=\inf_{\rho'\in [r,\rho]}\gammas_r(\rho')$, and we note the exponential decay (\ref{e.gamdecay}) for $\bgammas_r(r2^j)$ in $j$. A non-centered cluster $C$ is called overcrowded if $g(|C|)\, \bgamma_r(r2^{j(C)}) > 1$, with the function $g(k)$ of Proposition~\ref{pr.verysmall}. A centered cluster is overcrowded if $g^*(|C|)\, \bgammas_r(r2^{j(C)}) > 1$. We define $J(k)$ as before, using $g(k)$, and similarly define $J^*(k)$, using $g^*(k)$.
In particular, $\bgammas_r(r2^{J^*(k)})\,g^*(k)\asymp 1$.

As before, we are removing all the annuli corresponding to the proper subclusters of overcrowded clusters. We have almost arrived at our centered annulus structure $\Ann(S)$, except for one technical issue: it might happen that the smallest {\it nonempty} centered annulus contains non-centered annuli in its inner square, which we did not allow. (This situation can be so bad that there are no nonempty centered annuli at all.) In order to be able to apply Lemma~\ref{l.ComphRad}, we need to fix this problem; therefore, whenever this happens, we add back the largest centered outer square that does not contain any other annulus, and consider it as an annulus of zero width. With this, we now have $\Ann(S)$. The collection of these for all $S\subseteq\bits$ with $|S_r| = k$ is $\ANN^*(r,k)$. 

We define $H(j,k)$ similarly as before, but now only for interior inner bounding squares that do not contain 0. (So, it is a sum of $h(\Ann)=h^*(\Ann)$ terms over a collection of annulus structures $\Ann=\Ann'(S)$.) We similarly define the quantities $H^+(j,k)$ and $H^{++}(j,k)$ for side and corner squares not containing 0. Finally, we let $H^*(j,k)$ be the analogous quantity (using $h^*$) where the inner bounding square is required to be centered. We will show that there is some constant $\delta>0$, depending only on $\ell$, such that, for $j\in \{J^*(k),\dots,\bar j\}$,
\begin{equation}\label{e.Hstarjk}
H^*(j,k) \leq \bigl(g^*(k) \; \bgammas_r(r2^j)\bigr)^{1+\delta}\,.
\end{equation}
This implies (\ref{e.ANNRk}) (with a possibly larger constant $\cod$) in exactly the same way as in Subsection~\ref{ss.local} the bound (\ref{e.induct}) implied (\ref{e.ANNnk}) (with the small additional care regarding the cutoff $\bar j$ that we have seen in Subsection~\ref{ss.boundary}).

As usual, we prove (\ref{e.Hstarjk}) by induction on $j$. We may assume that $j > J^*(k)$.
Suppose that $\Ann$ is a centered annulus structure contributing to the sum $H^*(j,k)$, where $\Ann=\Ann'(S)$ for 
some $S\subseteq\bits$ with $|S_r|=k$. Let $j_*$ be $j(C^*)$,
where $C^*$ is the largest proper centered subcluster of $S_r^*$, and let $k_*=|C^*\cap S_r|$.
Every such $\Ann$ can be formed as a union of a centered annulus structure $\Ann^*$ for $(j_*,k_*)$,
the annulus $A(C^*)$, which might actually be missing if the ``scales'' $j_*$ and $j$ are very close to each other, and the annulus structure $\Ann^\sharp$ formed by dropping from $\Ann$ all the annuli that are part of some centered cluster. Moreover, 
$$
h^*(\Ann)^2\asymp (k+2)^{O(1)}\,h^*(\Ann^{*})^2\,h(\Ann^\sharp)^2\,
\alpha_\ell(r\,2^{j_*},r\,2^j)^2,
$$
where the last factor should not be there if  $A(C^*)$ is empty, but it is of constant order in that case anyway, hence the approximate equality is still valid. Then, the sum over such $\Ann$ with $j_*$ and $k_*$ fixed is bounded by
$$
(k+2)^{O(1)}\,
\alpha_\ell(r\,2^{j_*},r\,2^j)^2
\Bigl(\sum_{\Ann^*} h^*(\Ann^*)^2\Bigr)\Bigl(\sum_{\Ann^\sharp} h(\Ann^\sharp)^2\Bigr),
$$
where the sums run over the appropriate collections of annulus structures.
The first sum is bounded by $H^*(j_*,k_*)$, and the proof of~\eqref{e.ANNnk},
possibly incorporating boundary clusters,
shows that the second factor is bounded by $ g(k-k_*) \, \bgamma_r(r\,2^j)$,
with possibly a different choice of the constant $\coa$ implicit in $g$.
(Note that the annulus structure $\Ann^\sharp$ may have just one annulus whose
outer square is roughly at the scale corresponding to $j$. This case is handled by the computation in~(\ref{e.envelope}), and this is the reason for the estimate being of the type~\eqref{e.ANNnk}, rather than
the type estimated in~\eqref{e.induct}.)
The induction hypothesis therefore gives
\begin{multline*}
H^*(j,k) \leq
  (k+2)^{O(1)} \sum_{k_*,j_*}  
 \alpha_\ell(r2^{j_*},r2^j)^2\,  g^*(k_*) \, \bgammas_r(r2^{j_*})\,  g(k-k_*)\, \bgamma_r(r2^j)\\
\leq  (k+2)^{O(1)} \, g(k)\, g^*(k) \; j\, \bgamma_r(r2^j) \; \bgammas_r(r2^j)\,.
\end{multline*}

If $\delta>0$ is small enough, then $j\,\bgamma_r(r2^j) \leq O(1)\, \bgammas_r(r2^j)^\delta$. Given this $\delta$, if $\cod$ is large enough, we also have $O(1)\,(k+2)^{O(1)} g(k) \leq g^*(k)^{\delta}$. Thus, our last upper bound is at most $\big( g^*(k)\, \bgammas_r(r2^j)\big)^{1+\delta}$, so (\ref{e.Hstarjk}) is proved, and our proof of Proposition \ref{pr.verysmallRad} is complete. \qed

\section{Partial independence in the spectral sample} \label{s.mainlemma}

\subsection{Setup and main statement}\label{ss.maintro}

Let $\Spec$ denote the spectral sample of the $\pm1$ indicator function of having a percolation left-right crossing in $[0,R]^2$
(in either of our two favorite lattices).
In order to prove that $|\Spec|$ is rarely much smaller than its mean 
it would be useful to have some independence of the following kind: if $B_1, B_2$ are two
distant squares, then we would expect that
\begin{multline*}
\Pb{\Spec \cap B_1=S_1 \md \Spec\cap B_1 \neq \emptyset ,\, \Spec\cap B_2=S_2} \asymp\qquad \\
\Pb{\Spec \cap B_1=S_1 \md \Spec\cap B_1 \neq \emptyset}.
\end{multline*}
It turns out that it is hard to control such correlations. Nevertheless, we will prove a weaker independence result that will be enough for our purposes.
\bigskip

Consider some box $B$ of radius $r$ inside $[0,R]^2$.
(Recall from Section~\ref{ss.gen} that a box $B(x,r)$ of radius $r$ is the union of tiles whose centers are in $x+[-r,r)^2$.)
We want to understand the behavior of $\Spec$ in $B$. Because of boundary issues, we will actually look at $\Spec$
in a smaller concentric box $B'$, of radius $r/3$.

We saw in (\ref{e.mean}) that $ O(1)\,\Eb{|\Spec \cap B'| \md \Spec \cap B \ne \emptyset} \geq r^2\, \alpha_4(r)$.
In this section, we will strengthen this by proving that $|\Spec\cap B'|$ is at least of this size
with a uniform positive probability, moreover, this remains true when we add $\Spec\cap W=\emptyset$ to the conditioning, where
$W$ is an arbitrary set in the complement of $B$:
\begin{equation}
\label{e.weakind}
\Pb{|\Spec \cap B'|\geq c\, r^2 \alpha_4(r) \md \Spec\cap B \neq \emptyset,\; \Spec\cap W =\emptyset } > a\,,
\end{equation}
with some fixed constants $c,a>0$.
However, the following stronger statement is closer to what we actually need.

\bpr\label{pr.SQBW}
Let $\Spec$ be the spectral sample of the $\pm1$-indicator function of the
left-right crossing event in $\Quad = [0,R]^2$.
Let $B$ be a box of some radius $r$. Let $B'$ be the concentric box with radius $r/3$,
and assume that $B'\subset\Quad$ (note that $B$ does not need to be in $\Quad$).
We also assume that $r\ge \bar r$, where $\bar r>0$ is some universal constant.
 Fix any set $W\subset\R^2\setminus B$, and let
$\Rs$ be a random subset of $\bits$ that is independent from $\Spec$, where each element of $\bits$ is
in $\Rs$ with probability $1/(\alpha_4(r)\,r^2)$ independently.
(By~\eqref{e.a4less2}, $\alpha_4(r)\,r^2\ge 1$ if $\bar r$ is sufficiently large.)
Then
$$
\Pb{\Spec\cap B'\cap \Rs \ne \emptyset \md \Spec\cap B \ne \emptyset,\ \Spec\cap W=\emptyset} > a\,,
$$
where $a>0$ is a universal constant.
\epr 

The estimate~\eqref{e.weakind} follows immediately from the proposition, since
\begin{multline*}
\Pb{\Spec\cap B'\cap \Rs \ne \emptyset \md \Spec\cap B',\ \Spec\cap B \ne \emptyset,\ \Spec\cap W=\emptyset} 
\\
=
1-\bigl(1-r^{-2}\,\alpha_4(r)^{-1}\bigr)^{|\Spec\cap B'|}.
\end{multline*}

\smallskip

It is important to note that in the proposition the constant $a>0$ is  independent of the position
of the box $B$ relative to the square $[0,R]^2$.
Such a uniform control over the domain would be harder to achieve in the case of general quads.
Instead, after proving this uniform result for the square, we will prove a local version (Proposition~\ref{pr.SQBWloc}) for general quads. We will also prove a radial version (Proposition~\ref{pr.SQBWrad}), which will be important for the application to exceptional times of dynamical percolation.

The proof of the proposition is straightforward once we have the following bounds on 
the first and second moments.
Recall the definition of $\llwb$ right before Lemma~\ref{l.lambda}.
\bpr[First moment]
\label{pr.first}
Assume the setup of Proposition~\ref{pr.SQBW}.
There is an absolute constant $c_1>0$ such that for any $x\in B'\cap\bits$,
\begin{equation}\label{e.1stm}
\Pb{x\in\Spec,\ \Spec \cap W=\emptyset} \geq c_1\, \Eb{\llwb^2}\, \alpha_4(r).
\end{equation}
\epr

\bpr[Second moment]
\label{pr.second}
Let $\Spec$ be the spectral sample of $f=f_\Quad$, where $\Quad\subset\R^2$ is some arbitrary quad.
Let $z\in\Quad$ and $r>0$. Set $B:= B(z,r)$ and $B':=B(z,r/3)$.
Suppose that $B(z,r/2)\subset\Quad$ and that $B$ and $W$ are disjoint.
Then for every $x,y\in B'\cap \bits$ we have
\be\label{e.2ndm}
\Pb{x,y\in\Spec,\ \Spec \cap W=\emptyset} \leq c_2\, \Eb{\llwb^2}\, \alpha_4(|x-y|)\,\alpha_4(r)\,,
\ee
where $c_2<\infty$ is an absolute constant.
\epr

\proofof{Proposition~\ref{pr.SQBW} {\rm (assuming the first and second moment estimates)}}
Consider the random variable 
$$
Y:=|\Spec \cap B' \cap \Rs|\,\1_{\{\Spec\cap W=\emptyset\}}.
$$ 
Since $\Rs$ is independent from $\Spec$
and $\Ps{x\in \Rs}=1/(\alpha_4(r)\,r^2)$, we obtain by summing~(\ref{e.1stm})
over all $x\in B'\cap \bits$ that 
$O(1)\,\Es{Y}\geq \Eb{\llwb^2}$. On the other hand, summing~(\ref{e.2ndm}) over all $x,y\in B'\cap\bits$, similarly to the second moment estimate in Lemma~\ref{l.moments}, gives 
\begin{align*}
\Es{Y^2} &\le
\overbrace{O(1)\,\Eb{\llwb^2}\,\alpha_4(r)\,r^2\,\Pb{x\in\Rs}}^{\text{diagonal term}}+
\overbrace{O(1)\,\Eb{\llwb^2}\,\alpha_4(r)^2\,r^4\,\Pb{x\in\Rs}^2}^{\text{off-diagonal term}}\\
& \leq O(1)\,\Eb{\llwb^2},
\end{align*}
by our choice of $\Pb{x\in\Rs}$. Note that this choice for $\Pb{x\in \Rs}$ is of the smallest
possible order that does not make the diagonal term the leading contribution.
Now, by Cauchy-Schwarz,
\begin{equation}\label{e.BWY}
\Pb{Y > 0}\geq \frac{\Es{Y}^2}{\Eb{Y^2}} \ge
\frac{\Es{\llwb^2}^2}{O(1)\,\Es{\llwb^2}} =
\frac{\Eb{\llwb^2}}{O(1)}.
\end{equation}
The proposition now follows from Lemma~\ref{l.lambda}.  \QED

\begin{remark}\label{r.asymp}
$\Pb{\Spec \cap B \ne \emptyset,\ \Spec\cap W=\emptyset}$ is obviously not smaller than the left hand side of (\ref{e.BWY}). Therefore, (\ref{e.BWY})
and Lemma~\ref{l.lambda} imply that in the present setting
\begin{equation}\label{e.SLasymp}
\Pb{\Spec \cap B \neq \emptyset,\ \Spec \cap W =\emptyset}\asymp \Eb{\llwb^2}.
\end{equation}
The definition of $\lala(B,W)$ easily gives
\begin{equation}\label{e.LA}
\Eb{\lala(B,\emptyset)^2} =  \alpha_\square(B,\Quad) \quad\text{and}\quad
\lala(B,B^c) = \alpha_\square(B,\Quad)\,.
\end{equation}
Combining these with (\ref{e.SLasymp}), we get that
for $B$ as above,
approximate equalities hold in Lemma~\ref{l.basic}, i.e., 
\begin{equation}\label{e.SBasymp}
\Pb{\Spec \cap B \neq \emptyset}\asymp \alpha_\square(B,\Quad) \quad\text{and}\quad
\Pb{\emptyset\ne \Spec \subseteq B}\asymp \alpha_\square(B,\Quad)^2\,.
\end{equation}
\end{remark}

\subsection{Bounding the second moment}\label{ss.main12}
Due to the way in which $\lala(B,W)$ was defined, it is generally easier to obtain
$\lala(B,W)$ as an upper bound up to constants, than as a lower bound up to constants.
Consequently, the second moment estimate is easier to prove, and for this reason we start 
with that.

\proofof{Proposition~\ref{pr.second}}
Let $\theta$ denote the restriction of $\omega$ to
the complement of $W\cup\{x,y\}$. Then Lemma~\ref{l.LMN} gives
\begin{equation}\label{e.Wxy}
\begin{aligned}
\Pb{x,y\in\Spec,\Spec\cap W=\emptyset}
&=
\Pb{\Spec\cap ( W\cup\{x,y\})=\{x,y\}} \\
&=
\EB{
\Eb{\chi_{\{x,y\}}(\omega)\,f(\omega) \md \theta}^2}\,.
\end{aligned}
\end{equation}
Set 
$$
g(\theta):= \Eb{\chi_{\{x,y\}}(\omega)\,f(\omega) \md \theta}.
$$
Then $\Eb{g^2}$ is the quantity that we need to estimate.
Since $B\cap W=\emptyset$, the information in $\theta$ includes the
configuration in $B\setminus\{x,y\}$.
If $\theta$ does not have the $4$ arm event from the tile of $x$
to distance $|x-y|/4$, then flipping $\omega_x$ does not effect
$f(\omega)$, and hence $g(\theta)=0$.  A similar statement holds for $y$.
Also, if the box $\tilde B$ of radius $2\,|x-y|$ centered at $(x+y)/2$
does not intersect $\partial B$, then $g(\theta)=0$ unless $\theta$ has the $4$ arm event
in the corresponding annulus $B\setminus \tilde B$.
Let $A_x,A_y$ and $A_{x,y}$ denote the indicator functions for the $4$-arm
event in the corresponding $3$ annuli, where we take $A_{x,y}=1$ if 
$\tilde B\cap\partial B\ne\emptyset$.
Then we have $g(\theta)=0$ if $A_x\,A_y\,A_{x,y}=0$.

We now argue that $|g(\theta)|\le \lala(B,W)$. For this purpose, write
$$
g= 
 \Eb{\chi_{\{x,y\}} f\md \ev F_{ (W\cup \{x,y\})^c}}
=
 \Eb{\Es{\chi_{\{x,y\}} f\md \ev F_{ \{x,y\}^c}}\md \ev F_{W^c}},
$$
where we used that our measure is i.i.d. Clearly, 
 $\bigl|\Es{\chi_{\{x,y\}} f\md \ev F_{ \{x,y\}^c}}\bigr|\le 
1_{\Lambda_{\{x,y\}}}\le 1_{\Lambda_B}$,
where $\Lambda_{\cdot}$ is as defined above Lemma~\ref{l.lambda}.
Taking conditional expectation given $\ev F_{W^c}$ then gives
$$
\bigg|
\EB{
 \Es{\chi_{\{x,y\}} f\md \ev F_{ \{x,y\}^c}}\md \ev F_{W^c}}
\bigg|
\le
\EB{
\bigl|
 \Es{\chi_{\{x,y\}} f\md \ev F_{ \{x,y\}^c}}\bigr|\md \ev F_{W^c}}
\le \lala(B,W)\,.
$$
Since the left hand side is $|g|$, we get $|g|\le \lala(B,W)$.

Putting together the above, we arrive at
$ |g(\theta)|\le A_x\,A_y\,A_{x,y}\,\lala(B,W)$.
Thus, $ g(\theta)^2\le A_x\,A_y\,A_{x,y}\,\lala(B,W)^2$.
Independence on disjoint sets then gives
$$
\Eb{g^2} \le \alpha_4(|x-y|/4)^2\,\alpha_4 (2\,|x-y|,r/3)\,\Eb{\lala(B,W)^2}\,.
$$
The proposition now follows from the familiar properties of $\alpha_4$.
\QED

\subsection{Reformulation of first moment estimate}\label{subs.int}

Before proving the first moment estimate (Proposition \ref{pr.first}), we explain how
it can be reformulated as a quasi-multiplicativity property analogous to the
quasi-multiplicativity property of the $j$-arm events~\eqref{e.qm}.
Recall that
$$
\Eb{\lala(B,W)^2}= \EB{ \Pb{\Lambda_B\md \mathcal{F}_{W^c}}^2}.
$$
It is not apriori clear how to work with $\Eb{\lala(B,W)^2}$, but here is a useful observation about this quantity.
Let $\omega'$ and $\omega''$ be two critical percolation configurations which coincide on $W^c$
but are independent on $W$.
Let $\AA_\square(B,\Quad)$ denote the set of percolation configurations $\omega$ for which the $4$-arm event
occurs in the annulus $\Quad\setminus B$ with the appropriately colored arms terminating on the correct
boundary arcs of $\Quad$; that is, the primal (white) arms terminating on the two distinguished arcs of
$\p\Quad$ and the dual (black) arms terminating on the two complementary arcs.
Then
$$
\Eb{\lala(B,W)^2}= \Pb{ \omega',\omega''\in \Lambda_B} = \Pb{\omega',\omega''\in \AA_\square(B,\Quad)};
$$
that is, $\Eb{\lala(B,W)^2}$ is just the probability that the corresponding $4$-arm event occurs
in both $\omega'$ and $\omega''$. 
Lemma~\ref{l.LMN} gives
$$
\Pb{x\in\Spec,\ \Spec \cap W=\emptyset} = \EB{\Es{f\,\chi_x \md \ev F_{(W\cup\{x\})^c}}^2}.
$$
Now, if $f$ is a monotone increasing function taking values in $\{-1,1\}$, then 
\begin{equation}
\label{e.increasing}
\Es{f\,\chi_x\md \ev F_{ \{x\}^c}} = 1_{\Lambda_{\{x\}}},
\end{equation}
and
$$ \Es{f\,\chi_x \md \ev F_{(W\cup\{x\})^c}} =
\EB{\Es{f\,\chi_x\md \ev F_{\{x\}^c}}\md \ev F_{W^c}} = \Eb{1_{\Lambda_{\{x\}}}\md \ev F_{W^c}} = \lala(x,W)\,.
$$
Hence,
$$
\Pb{x\in\Spec,\ \Spec \cap W=\emptyset} 
= \Eb{\lala(x,W)^2} = \Pb{ \omega',\omega''\in \AA_\square(x,\Quad)},
$$
where $\AA_\square(x,\Quad)$ has the obvious meaning. Likewise, since $W\cap B=\emptyset$,
we have $\omega'=\omega''$ in $B$, and so,
$$
\alpha_4(r) \asymp \Pb{\omega',\omega''\in \AA_4(x,B)},
$$
where $\AA_4(x,B)$ is the $4$-arm event (which does not pay attention to any distinguished
arcs on $\p B$).
Hence~\eqref{e.1stm} can be rewritten as
\begin{equation}
\label{e.1}
\Pb{ \omega',\omega''\in \AA_\square(x,\Quad)}
\ge c_1\,
\Pb{\omega',\omega''\in \AA_4(x,B)} \, \Pb{ \omega',\omega''\in \AA_\square(B,\Quad)}.
\end{equation}
To see that this is indeed a quasi-multiplicativity property,
observe that if we take $W=\emptyset$ and replace the
events with $\AA_\square$ by the corresponding events with $\AA_4$,
then this is essentially the same as the case $j=4$ in the
left inequality of~\eqref{e.qm}.
  
It turns out that with a few extra twists, a proof
which gives the quasi-multiplicativity estimates~\eqref{e.qm}
generalizes to give~\eqref{e.1}. This will be explained in the next
subsections.

\begin{remark}
Proposition~\ref{pr.SQBW} generalizes to the radial setting,
in which we consider the event of a crossing from the origin to a large distance away.
However, at present it does not generalize to the radial $2$-arm event where
a vacant crossing and an occupied crossing occur simultaneously.
The only argument in the proof that does not generalize to the $2$-arm event is~\eqref{e.increasing},
which is not true for non-monotone functions.
Instead, we have
\begin{equation}
\label{e.nonmonotone}
\Es{f\,\chi_x\md \ev F_{ \{x\}^c}} = 1_{M_x^+}-1_{M_x^-},
\end{equation}
where $M_x^+$ is the event that $x$ is monotonically pivotal (i.e., $f(\omega_{\{x\}}^+)=1=-f(\omega_{\{x\}}^-)$)
and $M_x^-$ is the event that $x$ is anti-monotonically pivotal. The problem with such functions is that for the 
first moment we would need to bound from below $\Eb{\Es{1_{M_x^+}-1_{M_x^-}\md \mathcal{F}_{W^c}}^2}$. This expression is
easily controlled from above by $\Eb{\lambda_{x,W}^2}$, but not from below due to cancellations between $M_x^+$ and
$M_x^-$. These cancellations are far from being negligible, thus there is no hope to get $O(1) \Eb{\Es{1_{M_x^+}-1_{M_x^-}\md \mathcal{F}_{W^c}}^2} \geq \Eb{\lambda_{x,W}^2}$ for general $W$. For instance, if $W=\{x\}^c$ and $f$ is an even function ($f(-\omega)=
f(\omega)$) like the $2$-arm indicator function for site percolation on the triangular grid,
then
$\Pb{\Spec=\{x\}}=\Eb{\Eb{1_{M_x^+}-1_{M_x^-}}^2}=0$.
This ``unfortunate'' cancellation between the events $M_x^+$ and $M_x^-$ is the reason of the breakdown of our methods for such events.
\end{remark}

\subsection{Quasi-multiplicativity for coupled configurations}\label{subs.qm}

Rather than proving specifically the inequality~\eqref{e.1}, we first
address a related statement which is somewhat cleaner.
In the following, $W$ is any fixed subset of $\bits$, and $\omega',\,\omega''$
are the above coupled configurations, which are independent in $W$ and agree on $\bits\setminus W$.
The annulus $B(0,R)\setminus B(0,r)$ will be denoted by $A(r,R)$.
Let $j\in\N_+$ be either $1$ or an even number and
let $\AA_j(r,R)$ denote the set of configurations $\omega$ that satisfy the alternating
$j$-arm event in the annulus $A(r,R)$.
Set
$$
\b_j^W(r,R):=\Pb{\omega',\omega''\in \AA_j(r,R)}.
$$
We will prove the following quasi-multiplicativity result:

\begin{proposition}[Quasi-multiplicativity]\label{pr.qm}
Let $j\in \N_+$ be either one or an even integer, and let $W\subseteq\bits$. Then
$$
 \b_j^W(r_1,r_2)\, \b_j^W(r_2,r_3) \leq C_j\,\b_j^W(r_1,r_3)
$$
holds for every $0<r_1<r_2<r_3$ satisfying $r_2\ge \bar r_j$,
where $C_j$ and $\bar r_j$ are finite constants depending only on $j$
(and in particular, not on $W$).
\end{proposition}

Note that the opposite inequality with $C_j=1$ holds by
independence on disjoint sets.

To prepare for the proof of the proposition, we first need to prove a few
lemmas. The first observation is the following monotonicity property:
\begin{equation}
\label{e.mono}
\b_j^{W_2}(r,R) \le \b_j^{W_1}(r,R)\qquad\text{if }W_1\subseteq W_2\,.
\end{equation}
Indeed, since
$$
\b_j^W(r,R) =
\EB{\Ps{\omega \in \AA_j(r,R)\md \ev F_{W^c}}^2},
$$
the claimed monotonicity follows by the orthogonality property
of martingale increments.

\smallskip

The case $j=1$ in Proposition~\ref{pr.qm} easily follows from
the Russo-Seymour-Welsh theorem and from the Harris-FKG inequality.
In the following, we will restrict to the case $j=4$, since the other
even values of $j$ are essentially the same.

\smallskip

Let $\delta$ be some small positive constant, and let $r_0>0$.
We say that $r\geq r_0$ is $\delta$-{\bf good} if $\b_4^W(r_0,2\,r)\ge \delta \b_4^W(r_0,r)$.
Of course, this notion of good depends on $W,\delta$ and $r_0$.

\begin{lemma}\label{l.largerthangood}
For any $\delta>0$, there exist $\bar r=\bar r(\delta)>0$ and $c=c(\delta)>0$ (both depending only on $\delta$),
such that for any $W\subseteq\bits$ and any $r_0>0$: if one assumes that $r\geq r_0\vee\bar r$ is $\delta$-good
then for every $r'>r$
$$
\b_4^{W\cup A(r,r')}(r_0,r')\ge c\, \b_4^W(r_0,r)\,(r/r')^d,
$$
where $d$ is a universal constant.
\end{lemma}


The proof of this lemma will rely on Lemmas~A.2 and~A.3 from~\cite{\SchrammSteif}.

\proof
Assume that $r$ is $\delta$-good. Then $\b_4^W(r_0,2\,r)\ge\delta\,\b_4^W(r_0,r)$.
Set
\begin{align*}
X' & :=\Pb{\omega'\in \AA_4(r_0,2\,r)\md \omega'_{B(0,r)}},
\\
X'' & :=\Pb{\omega''\in \AA_4(r_0,2\,r)\md \omega''_{B(0,r)}}.
\end{align*}
Then
\begin{equation}
\label{e.bxx}
\begin{aligned}
\b_4^W(r_0,2\,r)
&
= \Pb{\omega',\omega''\in\AA_4(r_0,2\,r)}
\\&
= \EB{\Ps{ \omega',\omega''\in\AA_4(r_0,2\,r)\md \omega'_{B(0,r)},\omega''_{B(0,r)}}}.
\end{aligned}
\end{equation}
Now, since 
\begin{align*}
&
\Ps{ \omega',\omega''\in\AA_4(r_0,2\,r)\md \omega'_{B(0,r)},\omega''_{B(0,r)}}
\\ &\qquad
\le \Ps{ \omega'\in\AA_4(r_0,2\,r)\md \omega'_{B(0,r)},\omega''_{B(0,r)}}
=X',
\end{align*}
and a similar relation holds with $X''$, we have
$$
\Ps{ \omega',\omega''\in\AA_4(r_0,2\,r)\md \omega'_{B(0,r)},\omega''_{B(0,r)}}\le X'\wedge X''=:\tilde X,
$$
where $\tilde X=X'\wedge X''$ denotes the minimum of $X'$ and $X''$.
Because $r$ is $\delta$-good,~\eqref{e.bxx} now gives $\Eb{\tilde X} \ge \delta\,\b_4^W(r_0,r)$.
Since $\{\tilde X>0\}\subseteq\{\omega',\omega''\in\AA_4(r_0,r)\}$, and the
latter event has probability $\b_4^W(r_0,r)$, this gives
\begin{equation}
\label{e.tX}
\Eb{\tilde X\md \omega',\omega''\in\AA_4(r_0,r)}\ge \delta\,.
\end{equation}

Now let $\tilde \omega'$ and $\tilde\omega''$ be two percolation configurations
that have the same law as $\omega$ that are independent of each other outside
of $B(0,r)$ and inside $B(0,r)$ they satisfy 
$\tilde\omega'=\omega'$ and $\tilde\omega''=\omega''$.
Let $s'$ be the least distance between the endpoints on
$\p B(0,2\,r)$ of any pair of disjoint interfaces of $\tilde\omega'$
that cross the annulus $A(r,2\,r)$.
(Take $s'=\infty$ if there is at most one such interface.)
We claim that $r/s'$ is tight, in the following sense:
for every $\eps>0$ there is a constant $M=M_\eps$,
depending only on $\eps$, such that
$\Pb{r/s'> M}<\eps$.
This is proved, for example, in~\cite[Lemma~A.2]{\SchrammSteif}.
We use this with $\eps=\delta/2$.
Thus, we have
\begin{equation}
\label{e.Md}
\Pb{s'< r/M}<\delta/2\,.
\end{equation}
This property will be referred to below as the \lq\lq separation of arms\rq\rq\ phenomenon.
  
Assume now that $r\ge 100\,M=:\bar r$. Then when $s'\ge r/M$, we know
that $s'$ is substantially larger than the lattice mesh.
Observe that the distance between the endpoints on
$\p B(0,2\,r)$ of any two disjoint interfaces
of $\tilde\omega'$ that cross $A(r_0,2\,r)$ is at least
$s'$ (since every such interface also crosses $A(r,2\,r)$), 
and if $\tilde\omega'\in \AA_4(r_0,2\,r)$ then there exist at least four such interfaces.
Let $L_k$ denote the sector $\{\rho\,e^{i\theta}:\rho>0, \theta\in[\pi\,k/4,\pi\,(k+1)/4]\}$.
Let $\ev Z'$ denote the event that in $\tilde\omega'$ for each $k\in\{0,2,4,6\}$ there is a 
crossing from $\p B(0,r_0)$ to $\p B(0,8\,r)$ in $L_k\cup A(r_0,4\,r)$, which
is white when $k\in\{0,4\}$ and black when $k\in\{2,6\}$.
By the proof of~\cite[Lemma~A.3]{\SchrammSteif},
we know (see Figure \ref{f.Zp}) that there is a constant $c_0=c_0(M)>0$ such that 
\begin{equation}
\label{e.Zp}
\Pb{\ev Z' \md \tilde\omega'_{B(0,r)},\, \tilde\omega'\in \AA_4(r_0,2\,r),\,s'\ge r/M} \ge c_0\,.
\end{equation}

\begin{figure}[htbp]
\centerline{
\AffixLabels{\epsfysize=3.2in \epsffile{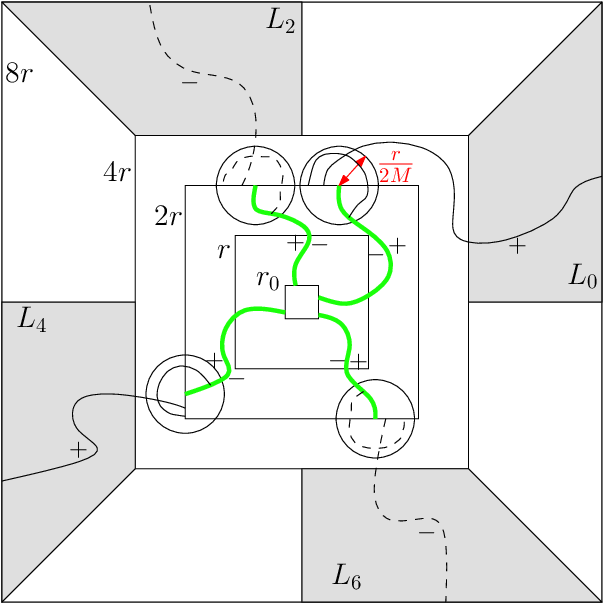}}
}
\begin{caption}{\label{f.Zp}
How to use ``separation of arms'' in order to get equation (\ref{e.Zp}).}
\end{caption}
\end{figure}

Note that $s'$ is independent from $\tilde\omega'_{B(0,r)}=\omega'_{B(0,r)}$.
Therefore,~\eqref{e.Md} gives
$$
\begin{aligned}
&\Pb{s'\ge r/M,\, \tilde\omega'\in\AA_4(r_0,2\,r) \md \omega'_{B(0,r)}}
\\&\qquad
\ge
\Pb{\tilde\omega'\in\AA_4(r_0,2\,r) \md \omega'_{B(0,r)}}
-\Pb{s'< r/M\md \omega'_{B(0,r)}}
\\&\qquad
\ge \tilde X-\delta/2\,.
\end{aligned}
$$
Together with~\eqref{e.Zp}, this shows that 
$$
\Pb{\ev Z'\md \omega'_{B(0,r)}} \ge
c_0\,(\tilde X-\delta/2)\,.
$$
Now let $\ev Z''$ be defined as $\ev Z'$, but with $\tilde\omega''$
replacing $\tilde\omega'$. Since $\tilde\omega'$ and $\tilde\omega''$
are conditionally independent given $(\omega'_{B(0,r)},\omega''_{B(0,r)})$,
we get
\begin{equation}
\label{e.ZZ}
\Pb{\ev Z',\,\ev Z''\md \omega'_{B(0,r)},\,\omega''_{B(0,r)}}
\ge c_0^2\, \bigl((\tilde X-\delta/2)_+\bigr)^2\,,
\end{equation}
where $(x)_+$ denotes $x\vee 0$.
Since $\bigl((\tilde X-\delta/2)_+\bigr)^2$
is a convex function of $\tilde X$, we get
from Jensen's inequality and~\eqref{e.tX}
$$
\Eb{\bigl((\tilde X-\delta/2)_+\bigr)^2\md \omega',\omega''\in\AA_4(r_0,r)}\ge \delta^2/4\,.
$$
Thus, taking the expectation of both sides of~\eqref{e.ZZ} gives
\begin{equation}
\label{e.ZpZpp}
\Pb{\ev Z',\,\ev Z''}
\ge c_0^2\,(\delta^2/4)\,\Pb{\omega',\omega''\in\AA_4(r_0,r)}
= c_0^2\,(\delta^2/4)\, \b_4^W(r_0,r)\,.
\end{equation}
This clearly implies the statement of the lemma in the case
$r'\le 8\,r$. Assume therefore that $r'>8\,r$.
Note that $\ev Z'\cap\ev Z''$ is monotone increasing inside
$(L_0\cup L_4)\setminus B(0,6\,r)$ and monotone decreasing in
$(L_2\cup L_6)\setminus B(0,6\,r)$.
Hence, it is positively correlated with the event
$\tilde{\ev Z}$ that for each
of $\tilde \omega'$ and $\tilde\omega''$ there are white paths separating
$\p B(0,6\,r)$ from $\p B(0,8\,r)$ in each of
$L_0$ and $L_4$, and similar black paths in $L_2$ and $L_6$, and
moreover, there are black paths in each of $L_2$ and $L_6$ joining
$\p B(0,6\,r)$ and $\p B(0,r')$
and white paths in each of $L_0$ and $L_4$ joining
$\p B(0,6\,r)$ and $\p B(0,r')$.
By the Russo-Seymour-Welsh theorem (see Figure~\ref{f.lemma5.7}), $\Pb{\tilde{\ev Z}}\ge c_1\,(r/r')^d$
for some absolute constants $c_1>0$ and $d<\infty$.


\begin{figure}[htbp]
\centerline{
\AffixLabels{\epsfysize=3.2in \epsffile{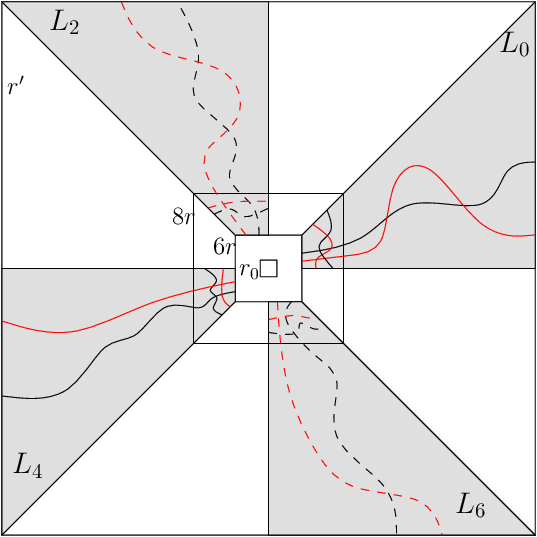}}
}
\begin{caption}{\label{f.lemma5.7}
A realization of the event $\tilde{\ev Z}$. The black color corresponds to arms in $\tilde \omega'$, while the red color corresponds to arms in $\tilde \omega''$.
}
\end{caption}
\end{figure}

Taking $c:= c_1\,c_0^2\,\delta^2/4$,
we obtain
$\Pb{\ev Z',\,\ev Z'',\,\tilde{\ev Z}}\ge c\,\b_4^W(r_0,r)\,(r/r')^d$.
The lemma follows, since 
$\ev Z'\cap\ev Z''\cap\tilde{\ev Z}\subseteq 
\{\tilde\omega',\tilde\omega''\in\AA_4(r_0,r')\}$.
\QED

\begin{lemma}\label{l.good}
There are absolute constants $\delta_0>0$ and $\bar R>0$ such that
\begin{equation}
\label{e.good1}
\b_4^W(r_0,2\,\rho)\ge \delta_0\,\b_4^W(r_0,\rho)
\end{equation}
holds for any $r_0>0$ and any $\rho\ge r_0\vee\bar R$. Furthermore,
\begin{equation}
\label{e.good2}
\b_4^W(\rho/2,r_0)\ge \delta_0\,\b_4^W(\rho,r_0)
\end{equation}
holds for any $\bar R\le \rho \le r_0$.
\end{lemma}

\proof
We start by proving the first claim.
Let us consider some $\delta \in (0, 2^{-(d+1)})$ (where $d$ is the universal constant from Lemma~\ref{l.largerthangood}),
whose value will be determined later.
Let $\bar r=\bar r(\delta)$ and $c(\delta)$ be as in Lemma~\ref{l.largerthangood}.
Let $r_0>0$ some radius.
Let $r\ge \bar r\vee r_0$, and assume for now that
$r$ is $\delta$-good. Then, by Lemma~\ref{l.largerthangood} and
the monotonicity property~\eqref{e.mono}, we have
\begin{equation}
\label{e.forecast}
\b_4^{W}(r_0,r')\ge c(\delta)\, \b_4^W(r_0,r)\,(r/r')^d,
\end{equation}
for every $r' > r$.
Set $\rho_k:=2^k\,r$, and let $\bar k:=\inf\{k\in\N_+: \rho_k\text{ is $\delta$-good}\}$, with $\bar k:=\infty$ if this set is empty.
If $m\in\N_+$ and $m\le \bar k$, then
by the definition of $\bar k$ and
by~\eqref{e.forecast} with $r':=\rho_{m}$, we have
$$
\begin{aligned}
\b_4^W(r_0,\rho_1)\,\delta^{m-1}
&
\ge \b_4^W(r_0,\rho_m)
\\&
\ge c(\delta)\,\b_4^W(r_0,r)\,2^{-dm}
\\&
\ge c(\delta)\,\b_4^W(r_0,\rho_1)\,2^{-dm}.
\end{aligned}
$$
Now, since $\delta<2^{-d-1}$, the above gives
$2^{-(d+1)(m-1)}\ge c(\delta)\,2^{-dm}$, which implies
$2^{d+1}\ge c(\delta)\,2^m$.
Hence, $\bar k$ is bounded from above by some finite constant depending only on $\delta$ 
(recall that $d$ is a universal constant).
We may conclude that $\delta$-good radii appear in scales with bounded gaps,
since the same argument may be applied with $r$ replaced by $\rho_{\bar k}$.
If $\rho$ is in the range $(r,\rho_{\bar k})$, then we have the estimate
$$
\frac{\b_4^W(r_0,2\,\rho)}{\b_4^W(r_0,\rho)} \ge
\frac{\b_4^W(r_0,2\,\rho_{\bar k})}{\b_4^W(r_0,r)} 
\overset{\eqref{e.forecast}}\ge c(\delta)\,2^{-d (\bar k+1)},
$$
which means that $\rho$ is $\delta_1$-good with
$\delta_1:= \delta_1(\delta) = c(\delta)\,2^{-d(\bar k+1)}$.
The same statement applies to any $\rho\ge r$, since, above $r$, the
$\delta$-good radii appear in scales with bounded gaps.

The proof of~\eqref{e.good1} is nearly complete, since we only have to initialize the above recursion. 
Given $r_0>0$, we want to find $\delta>0$ small enough and some $\bar R \ge \bar r = \bar r(\delta)$ 
that does not depend on $r_0$,
such that $\bar R \vee r_0$ will be $\delta$-good (recall that the definition of $\delta$-good 
in Lemma~\ref{l.largerthangood} above does depend on $r_0$). 
Once this is established, we can start the induction 
with $r:= \bar R \vee r_0$, and~\eqref{e.good1} will be satisfied for all $\rho\ge r$, with $\delta_0:=\delta_1$.

Clearly, by RSW, $\b_4^W(r,2r)$ is bounded from below once $r\ge R_0$, where $R_0$ is some absolute constant.
We now fix $\delta:=\delta(R_0) \in (0,2^{-(d+1)})$ such that for any $r\ge R_0$, $\b_4^W(r,2r) \ge \delta$.
Now that $\delta$ is fixed, define $\bar R:= \bar r(\delta)\vee R_0$. 
This certainly ensures that the requirement $\bar R \geq \bar r$ is fulfilled.

We now distinguish between two cases. If $r_0\ge \bar R$, then we want to start the induction with $r:=r_0$, 
which is fine, because $r_0\geq \bar R \geq R_0$, hence $r_0$ is $\delta$-good (for $r_0,W$) by the choice
of $\delta$. We thus obtain the desired ~\eqref{e.good1} for any $\rho \ge r=r_0$,
with $\delta_0:= \delta_1(\delta)$. 

Let us now deal with the case $0<r_0< \bar R$. Let $\rho \geq \bar R$. Note that for any coupled configurations 
$\omega', \omega''$ satisfying $\{ \omega', \omega'' \in \AA_4(\bar R, 2\rho) \}$, there always exists 
a pair of configurations $\tilde \omega', \tilde \omega''$ in the ball $B(\bar R)$ such that the concatenations
($\tilde \omega', \omega'$) and ($\tilde \omega'', \omega''$) both satisfy $\AA_4(r_0, 2\rho)$. 
Note that, for this to be entirely rigorous, one would need to choose the boxes $B(r)$ in a proper way to exclude 
discrete effects, see Figure~\ref{f.latticevisible}. On the triangular 
lattice, choosing the balls according to the triangular graph distance does it, for instance. 

\begin{figure}[htbp]
\begin{center}
\includegraphics[width=.25\textwidth]{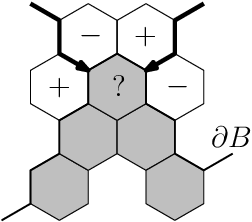}
\end{center}
\caption{If we had chosen boxes with such boundaries, then in this case whatever the color of the ?-hexagon is,
we would not be able to continue both interfaces.}\label{f.latticevisible}
\end{figure}

This implies that 
for any $1<r_0<\bar R\le \rho$, 
$$
\begin{aligned}
\b_4^W(r_0,2\rho) &\ge \left(\frac 1 2\right)^{2 | B(\bar R)|} \,\b_4^W(\bar R, 2\rho) \\
& \ge  \left(\frac 1 2\right)^{2 | B(\bar R)|}\,  \delta_1(\delta)\, \b_4^W(\bar R,\rho) \\
& \ge \left(\frac 1 2\right)^{2 | B(\bar R)|} \, \delta_1(\delta)\, \b_4^W(r_0,\rho)\,,
\end{aligned}
$$
where the second inequality follows from $r_0=\bar R$ in the above analysis.
Hence the first claim of the Lemma is now proved with $\delta_0:= 2^{-2 |B(\bar R)|} \delta_1(\delta)$.

To prove~\eqref{e.good2}, we follow a similar argument but with annuli
growing towards $0$ rather than towards infinity. For this,
a corresponding analogue of Lemma~\ref{l.largerthangood} is needed.
Since the proofs in this case are essentially the same, they are omitted.
\QED

\proofof{Proposition~\ref{pr.qm}}
As remarked above, we only prove the case $j=4$, since $j=1$ is
very easy and the proof for the case $j=4$ applies to all even $j$, 
with no essential changes.

If $r_2\le 4\, r_1$ or $r_3\le 4\, r_2$, then the claim follows from
Lemma~\ref{l.good}. Hence, assume that $r_1<r_2/4$ and $4\,r_2<r_3$.
By the monotonicity property~\eqref{e.mono}, it suffices to prove
$$
O(1)\,\b_4^{W\cup A(r_2/4,4\,r_2)}(r_1,r_3)\ge \b_4^W(r_1,r_2/4)\,\b_4^W(4\,r_2,r_3)\,.
$$
This follows from the proof of Lemma~\ref{l.largerthangood}:
we just need to apply the same argument twice, once going outwards
from $0$ and using~\eqref{e.good1} with $r_0:=r_1$ and $\rho:=r_2/4$
to verify that $r_2/4$ is $\delta_0$-good,
and once going inwards towards $0$ and using~\eqref{e.good2}
with $r_0:=r_3$ and $\rho:= 4\,r_2$.
The easy details are left to the reader.
\QED

Although this will not be needed in the following, we note that
the following generalization of Proposition~\ref{pr.qm} to arbitrary sequences of crossings holds.
(This can be proved by combining the above arguments with the proof of \cite[Proposition A.5]{\SchrammSteif}.)

\bpr\label{pr.qmgen}
Let $j\in \N_+$ and fix a color sequence $X\in \{\mathrm{black, white}\}^j$. For any set $W\subseteq\bits$,
the probabilities for 
the existence in both coupled configurations $\omega'$ and $\omega''$,
of $j$ crossings whose colors match this sequence in counterclockwise
order satisfy the inequalities in Proposition~\ref{pr.qm}. 
\epr

\subsection{Proof of first moment estimate}\label{ss.first}

\proofof{Proposition~\ref{pr.first}}
As remarked in Subsection~\ref{subs.int}, the proof of the first moment
estimate reduces to proving~\eqref{e.1}.
We will now explain how the proof of Proposition~\ref{pr.qm} needs to be adapted to
give~\eqref{e.1}. As in the case of Proposition~\ref{pr.qm}, one needs to show that arms 
``tend to separate'' when one conditions on the event we are interested in (i.e.~$\AA_\square(x,\Quad)$ here).
In Proposition~\ref{pr.qm}, the conditioning was very symmetric around $x$, while for Proposition~\ref{pr.first}, if the point $x$ happens to be close to the boundary 
$\p \Quad$, then, under the conditioning, boundary effects will have a strong influence. In order to deal with this influence of the boundary, 
it is convenient to locate the point $x$ first with respect to its closest edge (basically inducing a three-arm event) and then with respect to its closest corner (inducing a two-arm event).

Fix some $x\in\bits$ that is relevant for the left-right crossing in $\Quad=[0,R]^2$.
Let $x^+$ be the closest point to $x$ on $\p\Quad$ and let $x^{++}$ be the closest
point to $x$ among the four corners of $\Quad$. Set $R^+:= \|x-x^+\|_\infty$ and
$R^{++}:=\|x-x^{++}\|$.
We now define
$$
B_r:=\begin{cases}
B(x,r), & \bar r\le r\le R^+/8, \\
B(x^+,r), & 8\,R^+ \le r\le R^{++}/8,\\
B(x^{++},r), & 8\,R^{++} \le r\le R/8,
\end{cases}
$$
where $\bar r>1$ is some fixed constant, and let $\mathcal R$ denote the set of
$r$ for which $B_r$ is defined; that is,
$\mathcal R:= [\bar r,R^+/8]\cup[8\,R^+,R^{++}/8]\cup[8\,R^{++},R/8]$.
Given any $r\in\mathcal R$, define $\tilde r:=\inf (\mathcal R\cap [2\,r,R])$.
Then we say that $r$ is $\delta$-good if
$\Pb{\omega',\omega''\in \AA_4(x,B_{\tilde r})} \ge
\delta\,
\Pb{\omega',\omega''\in \AA_4(x,B_r)}$.
The proof that there is a universal constant $\delta$ such that
every $r\in\mathcal R$ satisfying $2\,r\le\sup\mathcal R$ is
$\delta$-good proceeds like the proof of~\eqref{e.good1} with a few minor changes.
The fact that some of the boxes considered are not concentric with each other
is of no consequence.  The only significant modification needed is that in the
argument corresponding to Lemma~\ref{l.largerthangood}, if $B_r\cap\p\Quad\ne\emptyset$,
then the interfaces considered are in the intersection of the corresponding
annulus and $\Quad$ and the definition of $s'$ needs to be modified.
In the adapted proof, $s'$ is defined as the least distance between any two distinct
points that are either endpoints on
$\p B_{\tilde r}$ of the interfaces or points in the intersection $\p B_{\tilde r}\cap \p\Quad$.
The remaining details are left to the reader, as is the similar proof
that 
$\Pb{\omega',\omega''\in \AA_\square(B_{\breve r},\Quad)} \ge
\delta\,
\Pb{\omega',\omega''\in \AA_\square(B_r,\Quad)}$
when $\breve r =\sup(\mathcal R\cap [\bar r,r/2])\ge\bar r$.
The proof of~\eqref{e.1} then follows as in the proof of Proposition~\ref{pr.qm}.
\QED

\begin{remark}\label{r.agglomerates}
In Sections~\ref{s.concent} and~\ref{s.noise}, we will need to apply Proposition~\ref{pr.SQBW} to a set of boxes
$B$ that form a grid covering $\Quad$.  Since we need each $B'$ to be contained in $\Quad$,
there is some care needed in placing the grid of boxes. In fact, for some radii $r$,
this is actually impossible. There are several alternative solutions to this problem.
The easiest solution is to restrict $r$ to the set of radii that admit grids of boxes that
cover $\Quad$ well. This happens, for example, when $r$ divides $R$.
 However, this solution has the drawback of not being easily adaptable
to other settings, for example, to the setting in which $\Quad$ is a rectangle
or some smooth perturbation of a rectangle. For this reason, we now describe a somewhat different
solution. Let $V$ be a maximal set of points in $\Quad$ such that the distance between
any pair of distinct points in $V$ is at least $r$ and the distance between any $v\in V$
to the closest point on $\p \Quad$ is at least $r$.
Consider the intrinsic metric $d_\Quad$ on $\Quad$, where $d_\Quad(x,x')$ is the
infimum length of any curve in $\Quad$ connecting $x$ and $x'$.
Let $(T_v:v\in V)$ denote the Voronoi tiling associated with $V$ and with this
metric, and let $B_v$ denote the union of the lattice tiles meeting $T_v$.
If we assume that $\Quad$ is \lq\lq reasonably nice\rq\rq,
then the maximal $d_\Quad$-diameter of any $B_v$ is $O(r)$. (Of course, we assume $r>1$.)
This will be the case, for example, when $\Quad$ is a rectangle whose smaller sidelength
is larger than $2\,r$, or more generally, if $\Quad=R\Quad_0$ for a piecewise smooth quad $\Quad_0$ and $R>c(\Quad_0) r$.
Now observe that the disk of radius $r/4$ around each $v\in V$ is contained
in the interior of $T_v$ and
is bounded away from $\p\Quad$. We may define $B'_v$ as the union of 
the lattice tiles that meet this disk.
The statement and proof of Proposition~\ref{pr.SQBW} hold
with $B_v$ and $B_v'$ replacing $B$ and $B'$, though the constants will depend
on the upper bound we have for $\diam(B_v)/r$.
\end{remark}

\subsection{A local result for general quads}\label{ss.piquad}

In this subsection, we prove the following local result, which is a key step in estimates for noise sensitivity in the case of general quads.

\bpr\label{pr.SQBWloc}
Let $\Quad\subset \R^2$ be some quad, and let $U$ be an open set whose closure is contained in the interior of $\Quad$.
For $R>0$, let
 $\Spec:= \Spec_{f_{R\Quad}}$ be the spectral sample of $f_{R\Quad}$, the $\pm 1$ indicator function
for the crossing event in $R\Quad$. Then, there is a constant $\bar r=\bar r(U,\Quad)$
such that
for any box $B\subset R\,U$ of radius $r \in [\bar r,R\diam(U)]$ and any set $W$ with $W\cap B=\emptyset$, we have
\begin{align*}
\Pb{\Spec_{f_{R\Quad}} \cap B' \cap \Rs \neq \emptyset \md
\Spec_{f_{R\Quad}}\cap B\neq \emptyset,\, \Spec_{f_{R\Quad}}\cap W =\emptyset}\geq
 a(U,\Quad)\,,
\end{align*}
where $B'$ is concentric with $B$ and has radius $r/3$,
the random set $\Rs$ is defined as in Proposition~\ref{pr.SQBW}, and $a(U,\Quad)>0$ is a constant that depends only on $U$ and $\Quad$.
\epr

\proof
Here, the main new issue to deal with is that the quad $\Quad$ is general;
but, in contrast to the situation in Proposition~\ref{pr.first},
the box $B$ is bounded away from $\p\Quad$, which simplifies parts
of the proof. 

The second moment estimate (Proposition~\ref{pr.second})
applies in the present setup. We now prove the corresponding first moment
estimate.

In the following discussion, the constants are allowed to depend on $U$ and
$\Quad$.
We start by proving the analogue of~\eqref{e.good2}.
Let $K$ be a compact set contained in the interior of $\Quad$ and containing
the closure of $U$ in its interior.
Let $M$ denote the set of all squares $S\subseteq K$ that intersect $\closure{U}$
while the concentric square of twice the radius is not contained
in the interior of $K$. Then $M$ is a compact set of squares in the natural topology, 
and the radius of the squares
in $M$ is bounded away from zero.
Fix some $S\in M$.
Let $B(R\,S)$ denote the union of the lattice tiles that meet $R\,S$.
A simple application of the Russo-Seymour-Welsh theorem
shows that
$\liminf_{R\to\infty} \Pb{\omega',\omega''\in\AA_\square\bigl(B(R\,S),R\,\Quad\bigr)}>0$.
Moreover, the same estimate holds in a neighborhood of $S$; that is, there is a set
$V\subset M$ that contains $S$ and is open in the topology of $M$, and there is a
constant $R_0=R_0(V,U,\Quad)>0$ such that
$$
\inf_{R\ge R_0}\,\inf_{S'\in V}\,
\Pb{\omega',\omega''\in\AA_\square\bigl(B(R\,S'),R\,\Quad\bigr)}> 0\,.
$$
Since $M$ is compact, this cover of $M$ by open subsets $V$ has a finite subcover,
and therefore there is some constant $R_1=R_1(U,\Quad)$ such that
$$
\inf_{R\ge R_1}\,\inf_{S\in M}\,
\Pb{\omega',\omega''\in\AA_\square\bigl(B(R\,S),R\,\Quad\bigr)}> 0\,.
$$
It is clear that there is some constant $b>1$ such that for every
$S\in M$ the concentric square whose radius is $b$ times the radius
of $S$ is still contained in $\Quad$. The above then shows that there
is a constant $\delta>0$ such 
that for all $R\ge R_1$ and all $S\in M$,
\begin{equation}
\label{e.bgood}
\Pb{\omega',\omega''\in\AA_\square\bigl(B(R\,S),R\,\Quad\bigr)}\ge \delta\,
\Pb{\omega',\omega''\in\AA_\square\bigl(B(R\,S_b),R\,\Quad\bigr)}\,,
\end{equation}
where $S_b$ denotes the square concentric with $S$ whose radius
is $b$ times the radius of $S$.
Let $\hat M$ denote the set of squares that are contained in and
concentric with some square in $M$.
Once we have~\eqref{e.bgood} for all $S\in M$, we can conclude
as in the proof of~\eqref{e.good2} in Lemma~\ref{l.good} that the same holds
with possibly a different constant $\delta$
for every $S\in \hat M$ such that $\diam(R\,S)\ge \bar r$,
for some constant $\bar r>0$. For this, the powers of $2$ that
were used in the proof of~\eqref{e.good1} and~\eqref{e.good2}
(for example, for the definition
of the notion of \lq\lq good\rq\rq) need to be replaced by powers of $b$,
but this is of little consequence. We also need here a version of Lemma~\ref{l.largerthangood} for the events $\AA_\square(B(RS_b),R\Quad)$, but that can be proved the same way as the original version, using \cite[Lemmas A.2, A.3]{\SchrammSteif} and (\ref{e.bgood}). Finally, the restriction that $R\ge R_1$ may be avoided by taking $\bar r$ sufficiently large. Thus, the analogue of~\eqref{e.good2} is established.

Based on~\eqref{e.good1} and the above analogue of~\eqref{e.good2}, we
obtain the analogue of~\eqref{e.1} for the current setup,
yielding the first moment estimate. (Note that we do not need to 
adapt the outward bound \eqref{e.good1} to this local result.)
The proof of the current proposition from the first and second
moment estimates follows as in the proof of Proposition~\ref{pr.SQBW}.
\QED

\subsection{The radial case}

For the study of the set of exceptional times for dynamical percolation, we will need some \concentration for the spectral samples of the ``radial'' indicator function.
For this purpose, the following analog of Proposition~\ref{pr.SQBW} for the radial setting
will be useful.

\bpr\label{pr.SQBWrad}
Let $f=f_R$ be the $0$-$1$ indicator function of the existence of a white crossing
between the two boundary components of the annulus
$[-R,R]^2\setminus [-1,1]^2$, and let $\Spec=\Spec_f$ be its spectral sample with 
law $\Pb{\Spec= S}=\widehat f(S)^2/\|f\|^2$.
Also let $W\subseteq\bits$.
Let $B$ be a box of some radius $r$ that does
not intersect $W$ and let $B'$ be the concentric box with radius $r/3$.
Suppose that $B'\subset [-R,R]^2$ and $B\cap [-4\,r,4\,r]^2=\emptyset$.
We also assume that $r\ge \bar r$, where $\bar r>0$ is some universal constant.
Then
\be
\Pb{\Spec\cap B'\cap \Rs\neq \emptyset \md \Spec\cap B \neq \emptyset,\,\Spec\cap W=\emptyset}>a,\nonumber
\ee
where $\Rs$ is as in Proposition~\ref{pr.SQBW} and $a>0$ is a universal constant.
\epr

\proof
The proof will be similar to the above proofs.
For this reason, we will be brief and leave many details to the reader.
Let $z$ be the center of the box $B$, and set $r_1:=|z| > 4r$.
Assume first that $r_1<R/3$.
We then consider the three annuli
$B(0,r_1/3)\setminus B(0,1)$, $B(0,R)\setminus B(0,3\,r_1)$
and $B(z,r_1/3)\setminus B$.
In order for $\Lambda_B$ to hold, it is necessary that
the $1$-arm event occurs in the first two annuli and
that the $4$-arm event occurs in the third annulus. 
Thus, by Lemma~\eqref{l.lambda},
\begin{align}
\Qb{B\cap\Spec\ne\emptyset = \Spec\cap W} 
&
 \le 4\, \Eb{\lala(B,W)^2} 
\nonumber\\
&
\le
4\, \Pb{\omega',\omega''\in \AA_1\bigl(B(0,1),B(0,r_1/3)\bigr)}
\times 
\nonumber\\&\qquad
 \Pb{\omega',\omega''\in \AA_1\bigl(B(0,3\,r_1),B(0,R)\bigr)}
\times 
\nonumber\\&\qquad
\Pb{\omega',\omega''\in \AA_4\bigl(B,B(z,r_1/3)\bigr)}. 
\label{e.lalaell}
\end{align}


Now, using the same technology of quasi-multiplicativity for coupled configurations and the separation of arms as we did before (starting with the argument of Subsection~\ref{subs.int}), one can prove the following first moment estimate for any $x\in B'$,
\begin{align}
\Qb{x\in \Spec, \Spec\cap W=\emptyset} &\geq c_1\, \b_1^W(1,r_1/3)\, \b_4^W(x,r_1/3) \, \b_1^W(3r_1, R)\nonumber\\
&\geq c_2\, \b_1^W(1,r_1/3)\, \alpha_4(r) \, \b_4^W(r,r_1/3) \b_1^W(3r_1,R)\nonumber\\
&\geq c_3 \, \alpha_4(r)\, \Eb{\lala(B,W)^2},\label{e.ellfirst}
\end{align}
where (\ref{e.lalaell}) is used for the last step. 
One can easily prove the analogous second moment estimate, and the claim now follows as in the proof of Proposition~\ref{pr.SQBW}
(it is the same second moment type of argument, except one has to renormalize the measure $\Q$ to get the probability measure $\P=\P_{f_R}$).

Suppose now that $r_1>2\,R/3$. In this case, we need to consider a different system of annuli.
Let $d$ denote the distance from $B$ to $\p B(0,R)$ and let $z'$ denote a closest point to
$B$ on $\p B(0,R)$.
In the annulus $B(0,R/3)\setminus B(0,1)$ we consider the $1$-arm event,
in the annulus $B(z,r+d/2)\setminus B$ we consider the $4$-arm event,
and in the intersection of $B(0,R)$ with $B(z',R/3)\setminus B(z',5\,r+d)$
(assuming that this is nonempty),
we consider the $3$-arm event between $\p B(z',R/3)$ and $\p B(z',5\,r+d)$.
Again, the claim follows.

In the intermediate case $R/3\le r_1\le 2\,R/3$,
we need to consider the $1$-arm event
in $B(0,R/6)\setminus B(0,1)$ and the $4$-arm event
in $B(z,R/6)\setminus B$, and the claim likewise follows. \QED

\section{A large deviation result}\label{s.YX}

In order to deduce Theorem~\ref{t.1} from the results of
Sections~\ref{s.mainlemma}
and~\ref{s.verysmall}, we will need the following general result. 
Let us note that not only the statement bears a vague resemblance to \cite{\LSSdomination}, as explained in Section~\ref{ss.rough}, but also the proof method of averaging using an independent random sample $I\subseteq [n]$.

\begin{proposition}\label{pr.YX}
Let $n\in\N_+$, let $x$ and $y$ be random variables in $\{0,1\}^n$,
and set $X:=\sum_{j=1}^n x_j$ and $Y:=\sum_{j=1}^ny_j$.
Suppose that a.s.\ $y_i\le x_i$ for each $i\in[n]$
 and that there is a constant $a\in(0,1]$ such that for each $j\in[n]$ and every
$I\subset[n]\setminus\{j\}$ we have
\begin{equation}\label{e.yx}
\Pb{y_j=1\md y_i=0\, \forall i\in I}
\ge
a\,
\Pb{x_j=1\md y_i=0\,\forall i\in I}\,.
\end{equation}
Then
\begin{equation}\label{e.y0}
\Pb{Y=0 \md X>0}\le a^{-1}\,\Eb{e^{-aX/e} \md X>0}.
\end{equation}
\end{proposition}

For completeness, we will also show below that
\begin{equation}\label{e.st}
\Pb{Y\le t} \le
\Pb{X<(e/a)\,(t+s)}+
\bigl(e^{t-1}/s\bigr)\,\Eb{e^{-a\,X/e}}
\end{equation}
holds for every $t\ge 0$ and $s>0$.
However, we do not have an application for this inequality.

\proof
We may write our assumption~\eqref{e.yx} as follows:
\begin{equation}\label{e.ma}
\Pb{ y_j=1,\,y_i=0\,\forall i\in I}
\ge a\,
\Pb{ x_j=1,\,y_i=0\,\forall i\in I}\,\1_{\{j\notin I\}},
\end{equation}
which is now true even when $j\in I$: it simply says $0\geq a\, 0$. 
This gives us many inequalities, which we will average out in a useful manner.
Fix $\lambda\in(0,1)$. Now multiply~\eqref{e.ma} by
$\lambda^{n-|I|}\,(1-\lambda)^{|I|}$; we now think of 
$I$ as a random subset of $[n]$, where each $i\in[n]$ is in $I$ independently with probability
$1-\lambda$. We will use $\P^I$ for this independent extra randomness. Summing~\eqref{e.ma}
over all choices of $j\in[n]$ and $I\subset[n]$, one gets for the left-hand side
\begin{eqnarray}
\sum_{j,I} \lambda^{n-|I|} (1-\lambda)^{|I|} \Pb{y_j=1,\, y_i =0 \, \forall i\in I} 
&=& \Eb{Y\,\P^I \bigl[ y_i = 0\, \forall i\in I \bigr]} \nonumber\\
&=& \Eb{Y \lambda^Y}\,.\nonumber
\end{eqnarray}
The right-hand side of~\eqref{e.ma} after summing over all choices of $j\in[n]$ and $I\subset[n]$ gives
\begin{multline*}
\sum_{j,I} \lambda^{n-|I|} (1-\lambda)^{|I|} \1_{\{j\notin I\}}\Pb{x_j=1,\, y_i =0 \, \forall i\in I} \\
= \Eb{ \sum_{j : x_j=1} \P^I \bigl[j\notin  I,\, y_i = 0\, \forall i\in I\bigr]}.\nonumber
\end{multline*}
Here, the random set $I$ not only has to avoid the points $i$ with $y_i=1$ but also the point $j$; hence, depending on whether 
$y_j=1$ or not, 
$I$ has to avoid $Y$ or $Y+1$ points. Therefore, $\Es{X\,\lambda^{Y+1}}$ is a lower bound on the last displayed quantity. Summarizing these computations, one ends up with
$$
\Eb{Y\,\lambda^Y}\ge a\,
\Eb{X\,\lambda^{Y+1}}.
$$
This may be rewritten as $\Es{Z}\ge 0$, where
$ Z:= (Y-a\,\lambda\,X)\,\lambda^Y$.
At this point, we choose $\lambda:=e^{-1}$.
In order to bound $Z$ from above by a function of $X$ only, we
maximize $Z$ over $Y$, and get the bound
$Z\le \exp(-1-a\,X/e)$.
On $X=0$, we also have $Y=0$ and $Z=0$,
while on $Y=0<X$, we have $Z\le -a\,e^{-1}$.
Therefore, $\Es{Z}\ge 0$ gives
$$
a\,e^{-1}\,\Pb{Y=0<X}\le \Eb{1_{X>0}\,\exp(-1-a\,X/e)}.
$$
Dividing by $a\,e^{-1}\,\Pb{X>0}$, we obtain~\eqref{e.y0}.
\QED

We now prove~\eqref{e.st}.
Set $r:= (e/a)\,(t+s)$, $Z_+:=\max(Z,0)$ and $Z_-:=Z-Z_+$.
Note that on the event $\{X\ge r, Y\le t\}$ we have $Z \le -s\,e^{-t}$.
Hence,
$$
\Eb{Z_-}\le -s\,e^{-t}\, \Pb{X\ge r,\,Y\le t} \le -s\,e^{-t}\,\bigl(\Ps{Y\le t}-\Ps{X<r}\bigr).
$$
On the other hand, $\Eb{Z_+}\le \Eb{\exp(-1-a\,X/e)}$.
Since $0\le\Es{Z}=\Eb{Z_+}+\Eb{Z_-}$,~\eqref{e.st} follows.

\section{The lower tail of the spectrum} \label{s.concent}

In this section, we prove Theorem~\ref{t.1} and a few related
results. 

\subsection{Local version}\label{ss.lcon}

We start with a version which avoids the issues involving the boundary, as we did in Subsections~\ref{ss.local} and~\ref{ss.piquad}. The bound we give here is sharp up to a constant factor, but we will not need the lower bound, so will prove sharpness only in the square case (with boundary), in Section~\ref{ss.lsq}.

\begin{theorem}\label{t.loc}
Consider some quad $\Quad$, and let $\Spec=\Spec_{f_{R\Quad}}$ be the spectral sample of $f_{R\Quad}$,
the $\pm1$ indicator function for the crossing event in $R\Quad$.
Let $U\subset\Quad$ be open, and let
$U'\subset \closure{U'}\subset U$.
Then, for some constants $\bar r=\bar r(U',U,\Quad)>0$ and $q(U',U,\Quad)>0$, 
for any $r\in [\bar r , R\diam(U)]$,
\begin{align}
\label{e.loc}
\Pb{0<|\Spec_{f_{R\Quad}}\cap RU|\le r^2\,\alpha_4(r),\,\Spec_{f_{R\Quad}}\cap RU\subset RU'}\qquad\nonumber\\
\le q(U',U,\Quad)\, 
\frac{R^2\,\alpha_4(R)^2}{r^2\,\alpha_4(r)^2}\,.
\end{align}
\end{theorem}

\proof
Let the distance between $U'$ and the complement of $U$ be $\delta>0$. Then, with no loss of generality, we may assume that $r\leq \delta R/10$. (Otherwise, $r/R$ remains bounded away from $0$, so, by choosing $q(U',U,\Quad)$ large enough compared to $\delta$, the upper bound in~(\ref{e.loc}) becomes larger than 1, and we are done.) Consider the tiling of the plane by $r\times r$ squares given by the grid $r\,\Z^2$, recall from the end of Section~\ref{ss.gen} that each square gives a box that together form a tiling of the plane, and let $\{B_1,B_2,\dots,B_n\}$ be the set of those boxes that intersect $RU'$. Let $U''\subset U$ such that $RU''\supset \bigcup_{j=1}^n B_j$ but the distance of $U''$ to the complement of $U$ is at least $\delta/2$;  one can choose a $U''$ that works for all $r \leq \delta R/10$ at the same time. 
Let $\Rs$ be a subset of $\I\cap RU''$, where each bit $i\in RU''\cap \I$ is in $\Rs$
with probability $1/(r^2\alpha_4(r))$, independently from each other and from $\Spec$.
Let $y_j$ be the indicator function of the event 
$$
\Spec\cap B_j\cap \Rs\ne\emptyset\,,
$$
and let $x_j$ be the indicator function of the event $\Spec\cap B_j\ne\emptyset$.
As in Propositions~\ref{pr.SQBW} and~\ref{pr.SQBWloc}, let $\P$ denote the law of $\Spec$ coupled with the independent 
point process $\Rs$. Let $\tilde\P$ denote the law $\P$ (on $(\Spec,\Rs)$) conditioned on
the event $\Spec\cap (RU\setminus RU'')=\emptyset$,
and let $\tilde\E$ denote the corresponding expectation operator.
For every $I\subseteq\{1,\dots,n\}$ and every $j\in\{1,\dots,n\}\setminus I$, 
by applying Proposition~\ref{pr.SQBWloc} to $U''$ (in place of $U$ there) and $\Quad$, we get that
$$
\Bb{\tilde\P}{y_j=1\md y_i=0\, \forall i\in I}
\ge
a\,
\Bb{\tilde\P}{x_j=1\md y_i=0\,\forall i\in I}\,,
$$
for some constant $a=a(U'',\Quad)>0$. (Technically, in order to apply Proposition~\ref{pr.SQBWloc} here, one has to first condition on the values 
of $\Rs$ on the boxes in $I$. Furthermore, in order to obtain $\tilde\P$ instead of $\P$, before applying Proposition~\ref{pr.SQBWloc}, one needs to add the set
$RU \setminus RU''$ to the set $W$.) Therefore, the large deviation result
Proposition~\ref{pr.YX} gives (using $RU' \subseteq \bigcup_i B_i$) that
$$
\Bb{\tilde\P}{\Spec \cap \Rs =\emptyset\ne \Spec\cap RU'}
\le a^{-1}\, \Bb{\tilde\E}{e^{-aX/e}\, 1_{X>0}},
$$
where $X:=\bigl|\{j:\Spec\cap B_j\ne \emptyset\}\bigr|$.  This yields
\begin{multline*}
\Pb{\Spec\cap\Rs=\emptyset\ne\Spec\cap RU\subset RU'}
\\
\le a^{-1}\sum_{k=1}^\infty e^{-ak/e}\,
\Pb{X=k,\, \Spec\cap RU\subset RU''}.
\end{multline*}
We estimate the terms $ \Pb{X=k,\, \Spec\cap RU\subset RU''}$ using Proposition~\ref{pr.local}, and get the bound
\begin{equation}
\label{e.spsp}
\begin{aligned}
\Pb{\Spec\cap\Rs=\emptyset\ne\Spec\cap RU\subset RU'}
&
\le O(1)\, \sum_{k=1}^\infty e^{-ak/e}\,g(k)\,\gamma_r(R)
\\&
= O(1)\,\gamma_r(R)\,,
\end{aligned}
\end{equation}
where $g$ is as defined in the proposition and the constants implied by the $O(1)$ terms may
depend on $(U',U,\Quad)$. 

Now, by the choice of $\Rs$, for $|\Spec\cap RU'| \leq r^2\alpha_4(r)$ we have
$$
\Pb{\Spec\cap\Rs\cap RU'=\emptyset\md \Spec} = \left(1-\frac{1}{r^2\alpha_4(r)}\right)^{|\Spec\cap RU'|} \geq \, c
$$
for some absolute constant $c>0$, and hence 
\begin{align*}
c \,\Pb{|\Spec\cap RU'| \leq r^2\alpha_4(r),\,\emptyset\ne\Spec\cap RU\subset RU'}\hskip 1in\\
\leq \Pb{\Spec\cap \Rs =\emptyset\not=\Spec\cap RU\subset RU'}\\
\leq O(1)\,\gamma_r(R)\qquad\qquad\text{by \eqref{e.spsp}.}
\end{align*}
This proves~\eqref{e.loc}.
\QED

\subsection{Square version}\label{ss.lsq}

We prepare for the proof
of Theorem~\ref{t.1} by first showing that
\begin{equation}
\label{e.Sr}
\Ess_R \asymp R^2\,\alpha_4(R)\, ,
\end{equation}
which will be needed for the (easier) lower bound.
Note that this also implies~\eqref{e.3d4} for the triangular lattice, by quasi-multiplicativity
and~\cite{\SmirnovWerner}.

First, the lower bound on $\Ess_R$ follows immediately from
Lemma~\ref{l.moments}.
For the upper bound, we will need to consider the half-plane $3$-arm events
and the quarter-plane $2$-arm events that were discussed
in Section~\ref{ss.boundary}. Let $\Quad=[0,R]^2$, $f=f_\Quad$ and $\Spec=\Spec_f$.
Let $x\in\I$ be an input bit of $f$.
If $x$ is at distance $r_0$ from the
closest edge of $[0,R]^2$, and at distance $r_1$ from
the closest corner, then by~\eqref{e.singleBit} and quasi-multiplicativity,
we have
\begin{equation*}
\Pb{x\in\Spec}=\alpha_\square(x,\Quad)\asymp \alpha_4(r_0)\,\alpha_3^+(r_0,r_1)\,
\alpha_2^{++}(r_1,R)\,.
\end{equation*}
Now observe that
$\alpha_3^+(r_0,r_1)\le O(1)\,\alpha_4(r_0,r_1)$ follows from~\eqref{e.a+3},
 and~\eqref{e.a4less2}.
Thus, $\alpha_4(r_0)\,\alpha_3^+(r_0,r_1)\le O(1)\,\alpha_4(r_1)$.
Moreover, $\alpha_2^{++}(r_1,R)\le O(r_1/R)$, by~\eqref{e.a++2}
and~\eqref{e.a+3}. 
(For the triangular lattice, we have $\alpha_2^{++}(r_1,R)\leq \alpha_4(r_1,R)$ from~(\ref{e.a++2tri}) and~(\ref{e.a4}), 
hence we get $\Pb{x\in\Spec}\leq O(1)\alpha_4(R)$, which proves~(\ref{e.Sr}) immediately.)

Since the number of $x\in\I$ with $r_1\in[2^j,2^{j+1})$ is $O(2^{2j})$,
we get
\begin{align*}
\E|\Spec| 
&=
\sum_{x\in\I} \Pb{x\in\Spec}
\le
\sum_{j=0}^{\lceil\log_2 R\rceil} O(2^{2j})\,
\alpha_4(2^j)\,\frac{2^j}R
\\&
=
\frac{\alpha_4(R)}{R}
\sum_{j=0}^{\lceil\log_2 R\rceil} \frac{O(2^{3j})}{\alpha_4(2^j,R)}
\overset{\eqref{e.a4less2}}
\le
\frac{\alpha_4(R)}{R}
\sum_{j=0}^{\lceil\log_2 R\rceil} \frac{O(2^{3j})}{(2^j/R)^2}
=O(R^2)\,\alpha_4(R)
\,.
\end{align*}
Thus, we get~\eqref{e.Sr}.

\proofof{Theorem~\ref{t.1}}
The proof of
\begin{equation}
\label{e.up}
\Pb{0<|\Spec|<r^2\,\alpha_4(r)} \le O(1)\,\gamma_r(R)\,,
\end{equation}
for $1\le r\le R$, is very similar to the proof of Theorem~\ref{t.loc},
with some small modifications, which we now discuss.
Note that the set of $x\in\I$ that are relevant for $f$ are all within distance at most
$2$ from $[0,R]^2$. Note that we may assume, without loss of generality, that $r\ge \bar r$ for some absolute constant $\bar r$, and that $(R+4)/r\in\N_+$. Let $\{B_1,B_2,\dots,B_n\}$ be the set of boxes corresponding to the tiling of $[-2,R+2]^2$ by $r\times r$ squares. In the proof of Theorem~\ref{t.loc} we are now allowed 
to take $RU=RU'=RU''=[-2,R+2]^2$, by replacing the appeal to Propositions~\ref{pr.SQBWloc} and~\ref{pr.local} with an appeal to Propositions~\ref{pr.SQBW} and~\ref{pr.verysmall}, respectively. (In particular, we do not need to introduce the measure $\tilde\P$ now.) This gives~\eqref{e.up}.

We now show that the inequality in~\eqref{e.up} is actually an equality up to
constants.
Let $\NN$ be the set of indices $i$ such that the $r$-box $B_i$ is at least at distance $\nn/10$ from the boundary $\p [0,\nn]^2$. Consider the events
\begin{eqnarray*}
V_i &:=& \big\{|\Spec \cap B_i| \geq C\,r^2\alpha_4(r)\big\}\,,\\
W_i &:=& \big\{\Spec\cap B_i\not=\emptyset,\ \Spec \subseteq B_i\big\}\,,
\end{eqnarray*}
for $i\in\NN$. We claim that we may take the constant $C$ large enough so that $\Pb{V_i | W_i} \leq 1/2$. This will follow from Markov's inequality, once we know that
\begin{equation}\label{e.plus}
\EB{|\Spec \cap B_i| \md W_i} \leq O(1)\,r^2\alpha_4(r).
\end{equation}
To prove (\ref{e.plus}), first observe that for each $i\in\NN$,
$$
\Pb{W_i} \overset{(\ref{e.SBasymp})}\asymp \alpha_\square(B_i,[0,R]^2)^2  
\overset{(\ref{e.qmsq})}\asymp \alpha_4(r,R)^2\,.
$$
Then, we need a good upper bound on $\Pb{x\in\Spec,\ \Spec\subseteq B_i}$. We know this equals $\Eb{\lala(x,B_i^c)^2}$, see, e.g., Subsection~\ref{subs.int}. Similarly to the proof of Lemma~\ref{l.basic}, or following Subsection~\ref{subs.int}, one can easily show that this is at most $\alpha_4(x,B_i)\,\alpha_\square(B_i,[0,\nn]^2)^2$.
Summing up for all $x\in B_i$, and using the above estimate on $\Ps{W_i}$ and a computation similar to~(\ref{e.Sr}), we get (\ref{e.plus}).

So, we have $\Pb{V_i^c \md W_i} \geq 1/2$. Note that the events $V_i^c \cap W_i$ for different $i$'s are disjoint, hence 
\begin{eqnarray*}
\PB{0< |\Spec| \leq C\,r^2\alpha_4(r)} \geq \sum_{i\in\NN} \Pb{V_i^c \cap W_i}
\geq c\, (\nn/r)^2 \, \alpha_4(r,\nn)^2\,,
\end{eqnarray*}
for some $c>0$, and the lower bound is proved.
\qed

\begin{remark}\label{r.variantofelam}
For the triangular lattice, the following variant of~\eqref{e.lam} may also be established:
\begin{equation}
\label{e.lam2} 
\begin{aligned}
\limsup_{R\to\infty}\PB{0<|\Spec_{f_R}|\le\lambda\,\Ess_R} & \asymp
 \lambda^{2/3},\qquad\text{and}
\\
\liminf_{R\to\infty}\PB{0<|\Spec_{f_R}|\le\lambda\,\Ess_R} & \asymp
 \lambda^{2/3} ,
\end{aligned}
\end{equation}
holds for every $\lambda\in(0,1]$, where the implied constants do not depend on $\lambda$.
In view of Theorem~\ref{t.1}, this follows from the fact that
$$
\lim_{R\to\infty}\alpha_4(t\,R,R)\asymp t^{5/4},\qquad t\in(0,1]\,,
$$
which holds since
the limit of critical percolation is described by SLE$_6$
(this is explained in~\cite{\SmirnovWerner}),
and the probabilities for the corresponding events for SLE are determined up to constant
factors~\cite{\LSWii} and have no lower order corrections to the power law.
\end{remark}

\subsection{Radial version}\label{ss.lrad}

We also have the following radial version, where $\Spec$ is the spectral sample of the 0-1 indicator function $f$ of the crossing event from $\p[-1,1]^2$ to $\p[-R,R]^2$, so that $\Eb{f^2} \asymp \alpha_1(R)$. Recall that we have the measures $\Pb{\Spec=S}=\Qb{\Spec=S}/\Eb{f^2}=\widehat f(S)^2/\Eb{f^2}$.

\begin{theorem}\label{t.radial}
Let $\Spec$ be as above, and let $r\in [1,R]$. Then
$$
\Qb{|\Spec|<\alpha_4(r)\,r^2} \le O(1)\,\frac{\alpha_1(R)^2}{\alpha_1(r)}\,, \qquad
\Pb{|\Spec|<\alpha_4(r)\,r^2} \le O(1)\,\frac{\alpha_1(R)}{\alpha_1(r)}\,.
$$
\end{theorem}

\proof Again, the bits relevant for $f$ are contained in $[-R',R']^2$, where $R'=R+2$. We may assume that $r$ is such that $R'/r\in\N_+$ and $r\in [ \bar r,R/8]$
for some fixed constant $\bar r>0$, which guarantees that
$k:=\alpha_4(r)\,r^2>1$.
Take a subdivision of $[-R',R']^2$ into boxes $\{B_j\}$ of
side-length $r$, and let $K$ denote the union of the boxes that intersect $[-4r,4r]^2$.
We now let $\Rs$ be the random set in $\bits\cap [-R',R']^2\setminus K$ where each bit is in $\Rs$
with probability $1/k$, and $\Rs$ is independent from $\Spec$. 
We also let $X:=|\{j: \Spec\cap B_j\ne \emptyset,\,B_j\not\subset K\}|=|(\Spec\setminus K)_r|$. 
Note that $|\Spec_r|-X$ is bounded from above by the number of boxes in $K$, which is bounded by a constant.

Exactly as before, Propositions \ref{pr.SQBWrad} and \ref{pr.YX} give that 
$$
\Pb{\Spec \cap \Rs =\emptyset \not=\Spec\setminus K}
\le a^{-1} \Eb{e^{-aX/e}\, 1_{X>0}},
$$
with some absolute constant $a>0$. We can use Proposition \ref{pr.verysmallRad} to bound each $\Pb{X=n}$, $n\in\N_+$, and the argument that finished the proof of Theorem~\ref{t.loc} above now gives
$$
\Pb{0<|\Spec \setminus K|<k} \leq O(1)\,\Pb{\Spec\cap \Rs=\emptyset \not=\Spec\setminus K} \le O(1)\,\alpha_1(r,R)\,.
$$
(The situation is more similar to Subsection~\ref{ss.lsq} than to Subsection~\ref{ss.lcon} in that the boundary issues are already dealt with in Propositions~ \ref{pr.SQBWrad}  and~\ref{pr.verysmallRad}, hence we do not need $U''$ and the measure $\tilde \P$). Finally, observe that 
\begin{align*}
\Pb{|\Spec|<k} &\leq \Pb{\Spec\subset K}+\Pb{0<|\Spec \setminus K|<k}\\ 
&\leq \PB{\Spec^*\text{ is compatible with } [-R',R']^2 \setminus K} + O(1)\,\alpha_1(r,R)\\
&\leq O(1)\,\alpha_1(r,R) + O(1)\,\alpha_1(r,R) \qquad \text{by Lemma \ref{l.ComphRad}}\,,
\end{align*}
and the theorem is proved.
\QED

\subsection{Tightness of the quad spectral sample}\label{ss.quadconc}

This section will be devoted to the following analog of~\eqref{e.tight} in the setting of
more general quads, showing that the appropriately normalized spectral sample is \concentrated.

\begin{theorem}\label{t.quadtight}
Let $\Quad\subset\R^2$ be a quad and for $R>0$ let $\Spec_{f_{R\Quad}}$
denote the spectral sample of $f_{R\Quad}$, the $\pm1$-indicator function
of the crossing event of $R\Quad$. Then
$$
\lim_{t\to\infty}\,
\inf_{R>1}
\PB{|\Spec_{f_{R\Quad}}|\in \bigl[t^{-1}\,R^2\,\alpha_4(R),t\,R^2\,\alpha_4(R)\bigr]\cup\{0\}}
= 1\,.
$$
\end{theorem}

We do not presently prove that $\E|\Spec_{f_{R\Quad}}|\asymp R^2\,\alpha_4(R)$ as $R\to\infty$,
though we tend to believe that this holds.

The main technical difficulty in the case of a general quad $\Quad$ compared to a square is the boundary: our explicit computations in Subsection~\ref{ss.boundary} do not apply to a general quad (even if it has piecewise smooth boundary).
Thus, the proof will begin by showing that even in a general quad the spectral sample is unlikely to be very close to $\p\Quad$. 

\proof
For every fixed $\delta>0$ we can find a quad $\Quad'$ that is contained in the interior of $\Quad$ and such that 
\begin{equation}\label{e.quad'perc}
\limsup_{R\to\infty} \Pb{f_{R\Quad}\ne f_{R\Quad'}}<\delta\,.
\end{equation}
This is easy to see, and also worked out in detail in~\cite{\SchrammSmirnovNoise}.
Let $U',U\subset\Quad$ be open sets satisfying $\Quad'\subset U'\subset\closure{U'}\subset U$.

Now, (\ref{e.quad'perc}) and~\eqref{e.specapprox} imply that for all large enough $R$, the laws of the spectral samples $\Spec_{f_{R\Quad}}$ and $\Spec_{f_{R\Quad'}}$ have a total variation distance at most $4\,\sqrt\delta$, and 
\begin{equation}\label{e.quad'spec}
\Pb{\Spec_{f_{R\Quad}} \subseteq RU'}\ge 1- 4\,\sqrt\delta\,.
\end{equation}
Theorem~\ref{t.loc} can now be invoked to get
$$
\lim_{t\to\infty}\,\limsup_{R\to\infty} \Pb{\Spec_{f_{R\Quad}}\subseteq RU',\,0<|\Spec_{f_{R\Quad}}|<t^{-1}\,R^2\,\alpha_4(R)}=0\,.
$$
In conjunction with~\eqref{e.quad'spec}, this gives
\begin{equation*}
\limsup_{t\to\infty}\,\limsup_{R\to\infty} \Pb{0<|\Spec_{f_{R\Quad}}|<t^{-1}\,R^2\,\alpha_4(R)}\le 4\,\sqrt{\delta}\,,
\end{equation*}
and since $\delta$ was an arbitrary positive number, 
\begin{equation}
\label{e.nosmall}
\lim_{t\to\infty}\,\limsup_{R\to\infty} \Pb{0<|\Spec_{f_{R\Quad}}|<t^{-1}\,R^2\,\alpha_4(R)}=0\,.
\end{equation}
In the other direction, it is easy to see that $\E|\Spec_{f_{R\Quad}}\cap RU'|= O(R^2)\,\alpha_4(R)$, as $R\to\infty$.
Therefore, Markov's inequality and~\eqref{e.quad'spec} imply that for all sufficiently large $R$,
$$
\Pb{|\Spec_{f_{R\Quad}}|>t\,R^2\,\alpha_4(R)} \le 4\,\sqrt\delta\,+O(1/t)\,,
$$
where the implied constant may depend on $\delta$ but not on $R$. Thus,
\begin{equation}
\label{e.nolarge}
\lim_{t\to\infty}\,\limsup_{R\to\infty}
\Pb{|\Spec_{f_{R\Quad}}|>t\,R^2\,\alpha_4(R)} =0\,.
\end{equation}
Since for every $R_0\in(1,\infty)$ we obviously have
$$
\lim_{t\to\infty}\,\sup_{R\in[1,R_0]}
\PB{|\Spec_{f_{R\Quad}}|\notin \bigl[t^{-1}\,R^2\,\alpha_4(R),t\,R^2\,\alpha_4(R)\bigr]\cup\{0\}}=0\,,
$$
the theorem follows immediately from~\eqref{e.nosmall} and~\eqref{e.nolarge}.
\QED

\section{Applications to noise sensitivity}\label{s.noise}

We are ready to prove Corollary~\ref{c.noise} and Theorem~\ref{t.horizontal}, together with some generalizations.

\subsection{Noise sensitivity in a square and a quad}\label{ss.noisesquare}

We will now prove Corollary~\ref{c.noise} and its generalization Corollary~\ref{c.noisequad} to quad crossings. The proof of the latter will be quite simple using Theorem~\ref{t.quadtight}, and, actually, the same argument could be used to prove 
Corollary~\ref{c.noise}  from~(\ref{e.tight}). Nevertheless, we are giving a longer proof of Corollary~\ref{c.noise} that has the advantage of being more quantitative. In particular, it implies~(\ref{e.triangnoise}) below, and a variation on it will also be used in Section~\ref{s.dyns} to prove our quantitative dynamical sensitivity results.

\proofof{Corollary \ref{c.noise}}
Let $\I_R$ denote the set of bits on which $f_R$ depends,
and write $\eps=\eps_R$.
Recall from~(\ref{e.fcor}) that if $y$ is an $\eps$-noisy version of $x\in \{-1,+1\}^{\I_R}$, then 
\begin{equation} \label{e.corr}
\Psi_R:=
\Eb{f_R(y)f_R(x)}-\Es{f_R(x)}^2 =
\sum_{k=1}^{|\I_\nn|} (1-\eps)^{k}\,\Pb{|\Spec_{f_R}| = k}\,.
\end{equation}

Breaking the sum over $k$ in (\ref{e.corr}) into parts $(j-1)/\eps < k \le j/\eps $, with $j=1,2,\dots$, we get
\begin{equation}\label{e.PsiR}
\begin{aligned}
\Psi_R &\leq \sum_{j=1}^\infty (1-\eps)^{(j-1)/\eps} \, \Pb{(j-1)/\eps < |\Spec_{f_R}| \leq j/\eps}\\
&\leq  \sum_{j=1}^\infty e^{1-j} \, \Pb{ 0< |\Spec_{f_R}| \le j/\eps}\,. 
\end{aligned}
\end{equation}
Recall that $r^2\,\alpha_4(r)\to\infty$ as $r\to\infty$, by~\eqref{e.a4less2}.
For $s\ge 1$, let $\rho(s)$ be the least $r\in\N_+$ such that
$r^2\,\alpha_4(r)\ge s$, and let $\gamma(r):=r^2\,\alpha_4(r)^2$.
Since $\alpha_4(r+1)\leq \alpha_4(r)$, we have $(r+1)^2\alpha_4(r+1)  \le O(1) r^2\alpha_4(r)$, and thus
\begin{equation}
\label{e.rhobd}
 s\leq \rho(s)^2\,\alpha_4\bigl(\rho(s)\bigr)\le O(s)\qquad\text{for } s\ge 1.
\end{equation}
By~\eqref{e.a4less2}, there exist $c_1,c_2>0$ such that $s^{c_1}/O(1) \leq \rho(s) \leq O(1)s^{c_2}$. 
Substituting this and~\eqref{e.qRSW} into~(\ref{e.rhobd}) gives that 
\begin{equation}
\label{e.rhobdrat}
(s'/s)^{c_3}/O(1) \leq \rho(s')/\rho(s) \leq O(1) (s'/s)^{c_4} \qquad \text{for } s'\ge s\ge 1,
\end{equation}
with some $c_3,c_4>0$. This and~\eqref{e.qRSW} imply 
\begin{equation}
\label{e.gamrat}
  O(1)\,\gamma\bigl(\rho(s')\bigr) / \gamma\bigl(\rho(s)\bigr) \ge (s/s')^{O(1)} \qquad \text{for } s'\ge s\ge 1.
\end{equation}
Now set $\rho_j:=\rho(j/\eps)$. Then, for $j\in\N_+$,
$$
\begin{aligned}
\Pb{ 0< |\Spec_{f_R}| \le j/\eps}
&
\le \Pb{0<|\Spec_{f_R}| \le \rho_j^2\,\alpha_4(\rho_j)}
\\ &
\le O(1)\,\gamma(R)/\gamma(\rho_j)\qquad\text{by~\eqref{e.up}}
\\ &
\le O(1)\,\gamma(R)\,j^{O(1)}/\gamma(\rho_1)\qquad\text{by~\eqref{e.gamrat}}.
\end{aligned}
$$
Therefore~\eqref{e.PsiR} gives
\begin{equation}
\label{e.Psilast}
\Psi_R\le O(1)\,\gamma(R)/\gamma(\rho_1)\,.
\end{equation}
If $\lim_{R\to\infty} \Ess_R\,\eps_R=\infty$, then by~\eqref{e.Sr}
and the usual properties of $\alpha_4$ we have $R/\rho(1/\eps)\to\infty$
as well as
$\gamma(R)/\gamma(\rho_1)=\gamma(R)/\gamma\bigl(\rho(1/\eps)\bigr)\to 0$,
which together with~\eqref{e.Psilast} proves~\eqref{e.uncor}.

Now assume that $\eps_R\,\Ess_R\to0$. Applying~\eqref{e.fcor} 
and Jensen's inequality, we get
$$
\Eb{f_R(x)\,f_R(y)}= \EB {(1-\eps)^{|\Spec_{f_R}|}}
\ge (1-\eps)^{\Ess_R}\to 1\,,
$$
as $R\to\infty$.
Since $f_R(x)\,f_R(y)\le 1=f_R(x)^2$,~\eqref{e.corel}
follows.
\QED

Suppose that we are in the setting of the triangular grid,
and $\eps=t/\Ess_R$, where $t>1$.
Then with the above notations, we have by~\eqref{e.lam2}
and~\eqref{e.PsiR} that $ \limsup_{R\to\infty}\Psi_R \leq O(1)\, t^{-2/3}$.
Using the fact that we also have lower bounds in Theorem~\ref{t.1}, it is easy to see that
\begin{equation}
\label{e.triangnoise}
\limsup_{R\to\infty}
\Psi_R 
\asymp t^{-2/3}
\asymp
\liminf_{R\to\infty}
\Psi_R 
\,.
\end{equation}
In a forthcoming paper we plan to use this to show that in the appropriate
scaling limit of critical dynamical percolation (where both space and time are rescaled), the crossing events in the unit square at time $0$
and at time $t$ have correlations that decay like $t^{-2/3}$ as $t\to\infty$.

\bigskip
We also have the following generalization for the $\pm 1$-indicator function
of the left-right crossing in scaled versions of an
arbitrary fixed quad $\Quad$.

\begin{corollary}\label{c.noisequad}
Assume that $\eps_R\in(0,1)$ is such that $\eps_R\, R^2\alpha_4(R)\to\infty$ as $R\to\infty$, and $y$ is an $\eps_R$-noisy version of $x$. Then 
$$
\Eb{f_{R\Quad}(y)\,f_{R\Quad}(x)}-\Eb{f_{R\Quad}(x)}\,\Eb{f_{R\Quad}(y)} \to 0\,.
$$
On the other hand, if $\eps_R\, R^2\alpha_4(R)\to 0$, then
$$
\Eb{f_{R\Quad}(y)\,f_{R\Quad}(x)}-\Eb{f_{R\Quad}(x)^2} \to 0\,.
$$
\end{corollary}

\proof
First assume $\eps_R\, R^2\alpha_4(R)\to\infty$. Given $\delta>0$, by Theorem~\ref{t.quadtight} we have
$$\sup_{R>1}
\Pb{0< |\Spec_{f_{R\Quad}}| < t^{-1}\, R^2\,\alpha_4(R)} < \delta
$$
if $t \geq t_1(\delta)$ is large enough. Now let $R$ be large enough so that $\eps_R\, R^2\alpha_4(R) > t^2$.
Then, by~(\ref{e.corr}),
\begin{multline*}
\Psi_R
\leq \delta + \sum_{|S| \geq R^2\alpha_4(R)/t} \widehat f(S)^2 \,(1-\eps_R)^{|S|}
 \\
\leq \delta + \left(1-\frac{t^2}{R^2\,\alpha_4(R)}\right)^{R^2\alpha_4(R)/t}
\leq \delta + O(1)\,\exp(-t)\,, 
\end{multline*}
which is at most $2\delta$ if $t \geq t_2(\delta)$. We can choose $R$ large enough with respect to this new $t$, and hence $\Psi_R\to 0$ is proved.

Now assume $\eps_R\, R^2\alpha_4(R)\to 0$. Given $\delta>0$, by Theorem~\ref{t.quadtight} we have
$$\sup_{R>1}
\Pb{|\Spec_{f_{R\Quad}}| > t\, R^2\,\alpha_4(R)} < \delta
$$
if $t \geq t_1(\delta)$ is large enough. Now let $R$ be so large that 
$\eps_R\,R^2\alpha_4(R) < t^{-2}$. Then, by~(\ref{e.fcor}), for $\Spec=\Spec_{f_{R\Quad}}$,
\begin{align*}
\Eb{f_{R\Quad}(y)\,f_{R\Quad}(x)} 
&\geq \EB{(1-\eps_R)^{|\Spec|} \md
 |\Spec| < t R^2 \alpha_4(R)} \,\Pb{|\Spec| < t R^2 \alpha_4(R)}\\
&\geq \left(1-\frac{t^{-2}}{R^2\,\alpha_4(R)}\right)^{t R^2 \alpha_4(R)}(1-\delta)\\
&\geq \exp(-O(1)/t)\,(1-\delta)\,. 
\end{align*}
This is arbitrarily close to $\Eb{f_{R\Quad}(x)^2}=1$ for $t$ large; hence we are done.
\qed

\subsection{Resampling a fixed set of bits}\label{ss.deterministic}

We now prove a general version of Theorem~\ref{t.horizontal}.
If $y\in\Omega=\{-1,1\}^\bits$ is a noisy version of $x$,
then $y_j=x_j$, except on a small random set of $j\in\bits$.
We may consider a variation of this situation, where
we have some fixed deterministic set $\determ\subseteq\bits$,
and we take $y_j=x_j$ for $j\in\determ$ and take
the restriction of $y$ to $\determc:=\bits\setminus\determ$
 be independent from $x$
(and $y$ is uniform in $\Omega$). Although this setup
was mentioned in~\cite{\BKSnoise}, the techniques developed there
and in~\cite{\SchrammSteif} fell short of being able to handle
this variation.
Now, we can analyse this variation without difficulty, and prove the following proposition:

\begin{proposition}
\label{p.determ}
Let $\Quad\subset \R^2$ be some quad and for $R>1$ let $f_R=f_{R\Quad}$ be the $\pm 1$-indicator function of the crossing event in $R\Quad$ (either in $\Z^2$ or in the triangular lattice).
For every $R>1$, let $\determ_R\subseteq\bits$ be some  set of bits, and let 
$r_R$ be the maximal radius of any disk contained in $R\,\Quad$ that is disjoint from 
$\determc_R:=\bits\setminus \determ_R$. 
If 
\begin{equation}
\label{e.rR}
\lim_{R\to\infty} \frac{r_R}{\sqrt{R^2\,\alpha_4(R)}}=0\,,
\end{equation}
then the family $(\determ_R)_{R>1}$ is {\bf asymptotically clueless} in the sense that
$$
\lim_{R\rightarrow \infty} \bigl\| \Es{f_R\md \mathcal{F}_{\determ_R}} - \Es{f_R} \bigr\| = 0\,.
$$
On the other hand, if 
\begin{equation}
\label{e.Uc}
\lim_{R\to\infty}\bigl|\determc_R\bigr|\,\alpha_4(R)=0\,,
\end{equation}
then $(\determ_R)_{R>1}$ is {\bf asymptotically decisive} in the sense that
$$
\lim_{R\rightarrow \infty} \bigl\| \Es{f_R\md \mathcal{F}_{\determ_R}} - f_R \bigr\| = 0\,,
$$
which means that there is asymptotically no loss of information about the crossing $f_R$.
\end{proposition}

Notice that even though the convergence of $\Eb{f_{R\Quad}}$ is not known in $\Z^2$ for general quads or even rectangles other than squares, our definitions of being asymptotically clueless or decisive still make perfect sense.

Note that $O(1)\,|\determc_R|\ge (R/r_R)^2$ and there are examples where $|\determc_R|\asymp (R/r_R)^2$.
Thus, in some sense the conditions~\eqref{e.rR} and~\eqref{e.Uc} are nearly complementary.
However, the following two examples
are not covered.  Suppose that $\Quad$ is the unit square, and for each $R$ we take
$\determ_R$ to be the set of bits contained in the left half of the square $R\,\Quad$.
It is left to the reader to verify that in this case $\determ_R$ is neither asymptotically
decisive, nor asymptotically clueless. 

In the second example, we take $\Quad$ to be the unit square again, and let
$\determ_R$ be the set of bits outside of the disk of radius $\rho_R$ centered
at the center of the square $R\,\Quad$. Then $\determ_R$ is asymptotically decisive
as long as $\rho_R/R\to 0$, but this does not follow from the proposition (unless $\rho_R$ is small enough).
However, Remark~\ref{r.hausdeter} below does give a general statement which
covers this example.

\begin{remark}\label{r.deter}
When $(\determ_R)_{R>1}$ is asymptotically clueless, it is immediate to see that if $x_R$ and $y_R$ are two coupled percolation configurations which coincide on $\determ_R$, but are independent elsewhere, then
$$
\lim_{R \rightarrow \infty} \Eb{f_{R} (x_R)f_{R}(y_R)} - \Eb{f_{R}}^2 = 0.
$$
On the other hand, if $(\determ_R)_{R>1}$ is asymptotically decisive, then 
$$
\lim_{R \rightarrow \infty} \Eb{f_{R} (x_R)f_{R}(y_R)} - \Eb{f_{R}^2} = 0.
$$
\end{remark}

\proofof{Proposition \ref{p.determ}}
By~\eqref{e.SinB} and orthogonality of martingale differences, we have
\begin{multline*}
\Pb{\emptyset\ne \Spec_{f_R}\subseteq\determ_R}
=
\EB{\Es{f_R\md \mathcal{F}_{\determ_R}}^2-\Es{f_R}^2}
\\
=
\EB{\bigl(\Es{f_R\md \mathcal{F}_{\determ_R}}-\Es{f_R}\bigr)^2}.
\end{multline*}
Thus, $(\determ_R)_{R>1}$ is asymptotically clueless if and only if the spectral sample of $f_R$
satisfies
$\Pb{\emptyset \neq \Spec_{f_R}\subseteq \determ_R}\to 0$.
Similarly, 
$$
\Pb{\Spec_{f_R}\not\subseteq\determ_R}
=
\EB{f_R^2-\Es{f_R\md \mathcal{F}_{\determ_R}}^2}
=
\EB{\bigl(\Es{f_R\md \mathcal{F}_{\determ_R}}-f_R\bigr)^2}.
$$
Hence, a necessary and sufficient condition for $(\determ_R)_{R>1}$ to be asymptotically decisive is that
$\Pb{\Spec_{f_R}\subseteq \determ_R}\to 1$.

We now consider the simpler case in which $\Quad=[0,1]^2$,
and assume~\eqref{e.rR}.
Since the proof is rather similar to the proof of Theorem~\ref{t.1}, we will be brief here,
and only indicate some of the essential points and the arguments where a 
more substantial modification is necessary.
As in Section \ref{ss.lsq}, subdivide $[-2,R+2]^2$ into boxes $B_1,B_2,\ldots,B_{m^2}$ of radius $(R+4)/(2\,m)$,
where $m=m_R\in\N$ tends to infinity as $R\to\infty$, but very slowly.
As above let $B_j'$ denote the box concentric with $B_j$ whose radius is a third of the radius of $B_j$.
Let $H_R\subseteq \determc_R$ be a maximal subset of $\determc_R$ with the property that the distance between
any two distinct elements in $H_R$ is at least $r_R$.
Then for some constant $C$ every disk of radius $C\,r_R$ in $R\,\Quad$ contains a point of $H_R$, but
a disk of radius smaller than $r_R/2$ contains at most one point of $H_R$.
   Let $x_j$ be the indicator function
of the event $\Spec_{f_R}\cap B_j\neq\emptyset$ and let $y_j$ be the indicator function 
of the event $$\Spec_{f_R}\cap B'_j\cap H_R\neq \emptyset\,.$$

Our goal is to prove that for each
$I\subseteq\{1,\dots,m^2\}$ and every $j\in \{1,\ldots,m^2\}\setminus I$,
\begin{equation}
\label{e.xyI}
    \Pb{y_j=1 \md y_i=0\,\forall i\in I} \geq a\,\Pb{x_j=1\md y_i=0\, \forall i\in I},
\end{equation}
holds with some absolute constant $a>0$.
We mimic the proof of Proposition~\ref{pr.SQBW}.
Fix such $j$ and $I$, and set
$n:=|B_j'\cap H_R|$, $W:= H_R\cap\bigcup_{i\in I} B_i'$
and $Y:=\bigl| \Spec_{f_R}\cap B'_j\cap H_R\bigr|$.
Using Proposition~\ref{pr.first}, we get
$$
O(1)\,\Eb{Y\,1_{\Spec_{f_R}\cap W=\emptyset}} \ge
\Eb{\llwb^2}\,\alpha_4(R/m)\,n\,,
$$
and Proposition~\ref{pr.second} can be used to obtain
$$
\Eb{Y^2\,1_{\Spec_{f_R}\cap W=\emptyset}} \le
O(1)\,\Eb{\llwb^2}\,\alpha_4(R/m)^2\,n^2\,,
$$
provided that $\liminf_{R\to\infty} \alpha_4(R/m)\,n > 0$ and hence the diagonal term is dominated by a constant times the off-diagonal term.  
(Intuitively, we need that the ``density'' of the set $H_R$ inside the box $B_j$ of radius $R/m$ is good enough to see pivotals once the box has any of them.)
Since $n\asymp R^2\,(m\,r_R)^{-2}$, this follows from~\eqref{e.rR},
provided that $m_R$ tends to $\infty$ sufficiently slowly.
Then,~\eqref{e.xyI} follows from the Cauchy-Schwarz second moment bound.

Using Propositions~\ref{pr.verysmall} and~\ref{pr.YX}, and following the proof of the \concentration Theorem~\ref{t.1}, we have for $R$ large enough that
$$\Pb{\emptyset \neq \Spec_{f_{R}} \subseteq \determ_R} \leq O(1) \,\gamma_{R/m}(R)
\overset{\eqref{e.a4less2}}{\underset{R\to\infty}\longrightarrow} 0
\,.
$$
By the above, it follows that
$(\determ_R)_{R>1}$ is asymptotically clueless.
\bigskip

For the case of a general quad, we can do the same trick as in Section~\ref{ss.quadconc} and in 
the proof of Corollary~\ref{c.noisequad}.
For any $\delta>0$, there is a quad $\Quad'$ contained in the interior of $\Quad$ such that for $R$ large enough
\begin{equation}
\label{e.Qp}
  \Pb{\Spec_{f_{R}} \subseteq R\Quad'}>1-\delta\,.
\end{equation}
We may assume that $\Quad'$ is smooth, though this is not really needed, since we are going to use Proposition~\ref{pr.SQBWloc}, with $\Quad'$ in place of $\Quad$.
So, the above arguments (for the square) can easily be adapted to show that
$$
\limsup_{R\to\infty} \Pb{\emptyset\ne \Spec_{f_R}\subseteq \determ_R\cap R\,\Quad'}=0\,,
$$
because the distance from $R\,\Quad'$ to the complement of $R\,\Quad$ is bounded from
below by a positive constant times $R$.
In combination with~\eqref{e.Qp}, this gives
$$\limsup_{R\rightarrow \infty} \Pb{\emptyset \neq \Spec_{f_{R}}\subseteq \determ_R} \leq \delta\,.$$ 
Since the left hand side does not depend on $\delta$, we may let $\delta$
tend to $0$ and deduce the first claim of the proposition.

\bigskip

For the other direction, we do a first moment argument. Let $\Quad'$ and $\delta$ satisfy~\eqref{e.Qp}, as before.
Let $\rho_0>0$ denote the distance from $\partial \Quad'$ to $\partial \Quad$.
Then for $x\in\bits\cap (R\,\Quad')$, we have $\Pb{x\in\Spec_{f_R}}\le O(1)\,\alpha_4(\rho_0\,R)$.
Thus,
$$
\Eb{|\Spec_{f_R}\cap \determc\cap(R\,\Quad')|} \le |\determc|\,\alpha_4(\rho_0\,R)\,.
$$
If we assume~\eqref{e.Uc}, then this tends to zero as $R\to\infty$.
Thus, 
$$
\lim_{R\to\infty}\Pb{\Spec_{f_R}\cap\determc\cap (R\,\Quad')\ne\emptyset}=0\,,
$$
and~\eqref{e.Qp} gives
$$
\limsup_{R\to\infty} \Pb{\Spec_{f_R}\cap\determc\ne\emptyset}\le\delta\,.
$$
Once again, since $\delta>0$ is arbitrary, this completes the proof.
\QED

\proofof{Theorem~\ref{t.horizontal}} The Theorem follows immediately from
Proposition~\ref{p.determ} and Remark~\ref{r.deter}.
\qed

\begin{remark}\label{r.ZU}
Recall that we used the random set $\Rs$ in Sections~\ref{s.mainlemma} and~\ref{s.concent}, with $\Ps{x\in\Rs}=1/(\alpha_4(r)\,r^2)$, just as a tool to measure the size of $\Spec$. However, in the spirit of our above proof, in place of $\determc_r$, we can think of $\Rs$ as the actual set of bits being resampled.  
\end{remark}

\begin{remark}\label{r.hausdeter}
It may be concluded from a slight variation on the proof of Proposition~\ref{p.determ} that
in the setting of the triangular grid if the Hausdorff
limit $F:=\lim_{R\to\infty} \determc_R/R$ exists and has Hausdorff dimension strictly less
than $5/4$, then $(\determ_R)_{R>1}$ is asymptotically decisive.
Indeed, assuming that $s:=\dim(F)<5/4$, for every $\eps>0$ one may find a countable collection of points
$z_j$ and radii $\rho_j$, such that $\sum_j \rho_j^{s+\eps}<\eps$ and the union of the
disks with these centers and radii contains a neighborhood of $F$. The probability
that $\Spec_{f_R}$ comes within distance $O(1)\,\rho_j\,R$ of $R\,z_j$ is bounded from above by
$O(1)\,\rho_j^{5/4-\eps}$, if $z_j$ is not too close to the boundary of
$\Quad$ and $R$ is sufficiently large. A sum bound and the above argument for dealing
with a neighborhood of the boundary of $\Quad$ complete the proof.
\end{remark}

\begin{remark}\label{r.polarset}
We obtain here a somewhat sharp result for ``sensitivity to selective noise'', though it would be even more satisfying
to have a necessary and sufficient condition for a family $(\determ_R)_{R>1}$ to be asymptotically clueless.
We believe that $(\determ_R)_{R>1}$ is asymptotically clueless if and only if
$\Pb{\emptyset\ne\Piv_R \subseteq \determ_R }\rightarrow 0$, where $\Piv_R$ is the set of pivotals.
Similarly, $(\determ_R)_{R>1}$ should be asymptotically decisive if and only if $\Pb{\Piv_R \subseteq \determ_R }\rightarrow 1$.
In other words, even though $\Piv_R$ and $\Spec_R$ are asymptotically quite different (compare, e.g., Remark~\ref{r.annpivo}
with Proposition~\ref{pr.verysmall}), they should have the same polar sets.
\end{remark}

\begin{remark}\label{r.macro}
Tsirelson~\cite{\TsirelsonStFlour} distinguishes two types of noise sensitivities: micro and block 
sensitivities, where the latter is stronger than the micro sensitivity we have been considering so far. He gives the following illustrative examples. 
Consider the two functions on $\{-1,1\}^n$: $f_1 = \frac 1 {\sqrt{n}} \sum_{i=1}^n x_i$
and $f_2=\frac 1 {\sqrt{n}} \sum_{i=1}^{n-n^{1/2}} \prod_{k=i}^{i+n^{1/2}}  x_k$.
Both correspond to renormalized random walks which converge to Brownian motion, but the first is stable
while the second is noise-sensitive. Block sensitivity is defined as follows: instead of resampling 
the bits one by one, each
with probability $\eps$, we resample simultaneously blocks of bits. For $\delta>0$, divide the $n$ bits
into about $\delta^{-1}$ blocks of about $\delta n$ bits:
$B_i:=\N\cap\bigl[i\,\delta\,n,(i+1)\,\delta\,n\bigr)$. Each block is now resampled (i.e., all bits within the block)
with probability $\eps$. A sequence of functions is block sensitive if for any fixed $\eps>0$,
the limsup as $n$ goes to infinity
of the correlation in this block procedure is bounded by a function of $\delta$ which goes to 0 when $\delta$ goes to 0.
It is easy to see that the sequence of functions $f_2$ ($n=1,2,\dots$) is not block-sensitive.
This is related to the fact that the sensitivity of $f_2$ is ``localized'',
in the sense that its spectral sample $\Spec_{f_2}$, when rescaled by $1/n$, converges in law to a finite
(random) set of points.
As we will see in Section~\ref{s.spectralscaling}, this is not at all the case with the spectral sample of percolation.
It is easy to check that percolation crossing events 
are indeed block sensitive.
\end{remark}

\section{Applications to dynamical percolation} \label{s.dyns}

In this section, we prove Theorems~\ref{t.dyn} and \ref{t.Z2dyn}.

 As in Subsection~\ref{ss.lrad}, we consider the 0-1 indicator function $f=f_R$
of the percolation crossing
event from $\p([-1,1]^2)$ to $\p([-R,R]^2)$. Then 
$\Eb{f}=\Eb{f^2} \asymp \alpha_1(R)$. We let $\omega_t$ be the dynamical percolation configuration at time $t$, started at the stationary distribution at $t=0$. Recall that we have
\begin{equation}\label{e.corrad}
\Eb{f(\omega_0)\,f(\omega_t)}
= \sum_{k=0}^\infty e^{-kt}\sum_{|S|=k}\widehat f(S)^2,\qquad t>0\,.
\end{equation}
As in Subsection \ref{ss.noisesquare}, for $s\ge 1$
define $\rho(s)$ as the least $r\in\N_+$ such
that $r^2\,\alpha_4(r)\ge s$, and break the sum over $k$
in (\ref{e.corrad}) into parts according to the
$j\in\N$ satisfying $j/t \leq k < (j+1)/t $. 
Then, the same way we got~(\ref{e.Psilast}), 
just now using Theorem~\ref{t.radial} and the estimate
$$\alpha_1(\rho(1/t)) \leq  j^{O(1)} \alpha_1(\rho(j/t))\,,$$
which follows from~\eqref{e.rhobdrat} and~\eqref{e.qRSW},
we get
\begin{equation}\label{e.correll}
\Eb{f(\omega_0)\,f(\omega_t)}
\le O(1)\,\frac{\alpha_1(R)^2}{\alpha_1(\rho(1/t))}
= O(1)\,\frac{\Es{f(\omega_0)}^2}{\alpha_1(\rho(1/t))}\,.
\end{equation}

Let
$\EX$ denote the
set of exceptional times $t\in[0,\infty)$ for the event that
the origin is in an infinite open cluster.
To give a lower bound on the Hausdorff dimension of $\EX$, a well-known technique is Frostman's criterion, see e.g.~\cite[Theorem 4.27]{\BMbook} or \cite[Theorem 8.9]{\Mattila}.
Combined with a compactness argument, it gives the following; see \cite[Theorem 6.1]{\SchrammSteif}. For any $\gamma>0$, let
\begin{equation}\label{e.energy}
M_\gamma(R):=\int_0^1\int_0^1 \frac{\Eb{f_R(\omega_s)\,f_R(\omega_t)}}{\Eb{f_R(\omega_0)}^2\, |t-s|^\gamma}\,dt\,ds\,.
\end{equation}
If $\sup_{R} M_\gamma(R) < \infty$, then $\EX\cap[0,1]$ is nonempty with positive probability, and on this event a.s.\ its dimension is at least $\gamma$.
It is easy to see that there is a constant $d$ such that $\dim_H(\EX)=d$ 
a.s. Therefore, $\sup_R M_\gamma(R)<\infty$ also implies
$\dim_H(\EX) \geq \gamma$ almost surely.

\proofof{Theorem \ref{t.dyn}} 
To start with, we have $\rho(s)=s^{4/3+o(1)}$, by~(\ref{e.a4}). Secondly, 
$$\alpha_1(r)= r^{-5/48+o(1)}$$
by~\cite{\LSWoneArm}.
Thus, translation invariance and~(\ref{e.correll}) give 
$$
\Eb{f(\omega_s)\,f(\omega_t)} / \Eb{f(\omega)}^2 \leq 
O(1)\, |t-s|^{-(4/3)(5/48)+o(1)},
$$
as $|t-s|\to 0$.
Therefore, as long as
$\gamma < 1-(4/3)\,(5/48)$, we have $\sup_R M_\gamma(R)<\infty$.
The above discussion therefore gives $\dim_H(\EX)\ge 31/36$ a.s.
The matching upper bound is given by Theorem~1.9 of~\cite{\SchrammSteif}. 
This implies statement 1 of the Theorem.

The proof of Statement 2 is similar.
For the $0$-$1$ indicator function $f^+$ of the crossing event between
radius $1$ and $R$ in a half plane one gets the following analog of
Theorem~\ref{t.radial}: if $k\in\N_+$ satisfies $k\le \alpha_4(r)\,r^2$,
then
$$
\Qb{|\Spec_{f^+}|<k} \le O(1)\,\frac{\alpha^+_1(R)^2}{\alpha^+_1(r)}\,.
$$
The proof is similar to the proof of Theorem~\ref{t.radial}, and is left to the reader.
The bound corresponding to~\eqref{e.correll} is then
\begin{equation}
\label{e.fp}
\frac{\Eb{f^+(\omega_0)f^+(\omega_t)}}{\Eb{f^+(\omega_0)^2}}
= O(1)\,\alpha_1^+\bigl(\rho(1/t)\bigr)^{-1}.
\end{equation}
By~\cite[Theorem 3]{\SmirnovWerner}, we have $\alpha_1^+(r)=r^{-\xi_1^++o(1)}$, with
$\xi_1^+=1/3$.
The proof of the lower bound of $1-(4/3)\,\xi_1^+$ on the Hausdorff dimension then proceeds as above.
For the upper bound, we refer to~\cite[Theorem 1.13]{\SchrammSteif}. This proves
part 2.

For the proof of the third part, we now let $f$ be the indicator function of the
event that there is a white crossing in the upper half plane (here, the set of hexagons whose center has non-negative imaginary part) and a black crossing in the lower half plane (the set of hexagons whose center has negative imaginary part), both crossings
from some fixed radius $r_0=2$ to radius $R$. The two half planes are chosen here in such a 
way that that the percolation configurations in the two halves are independent, and
$r_0=2$ is chosen so that the event has positive probability for all $R>r_0$.
Then, by independence, we get
from~\eqref{e.fp} that
$$
\frac{\Eb{f(\omega_0)f(\omega_t)}}{\Eb{f(\omega_0)^2}}
= O(1)\,\alpha_1^+\bigl(\rho(1/t)\bigr)^{-2}.
$$
Thus, in this case we get the lower bound of $1/9$ for the
corresponding Hausdorff dimension, which completes the proof.
\qed

\proofof{Theorem \ref{t.Z2dyn}}
Let $f=f_R$ be the indicator function for the existence of an open crossing from 
$0$ to $\p([-R,R]^2)$.
We will apply the relation between $\alpha_1$, $\alpha_4$ and $\alpha_5$ that comes from the $k=2$ case of Proposition~\ref{p.strong} in our Appendix. Since $\alpha_5(r)\asymp r^{-2}$ (see \cite[Lemma 5]{\KestenSZh} or \cite[Corollary A.8]{\SchrammSteif}), the estimate~(\ref{e.strong}) says that there are some constants $c_1,\eps>0$ such that
\begin{equation}
\label{e.14}
\alpha_1(r)\, \alpha_4(r) > c_1 r^{\,\eps-2},
\end{equation}
holds for all $r>1$.
Thus,
$$
\begin{aligned}
\frac{\Eb{f(\omega_0)\,f(\omega_t)}}{\Eb{f(\omega_0)}^2}
&
\leq \frac{ O(1)}{\alpha_1\bigl(\rho(1/t)\bigr)}
\qquad&\text{by~\eqref{e.correll}} \\
&
\le O(1)\,\rho(1/t)^{2-\eps}\,\alpha_4\bigl(\rho(1/t)\bigr)
\qquad&\text{by~\eqref{e.14}} \\
&
\le \frac{O(1)}{t}\,\rho(1/t)^{-\eps}
\qquad&\text{by~\eqref{e.rhobd}} \\
&
\le O(1)\,t^{\eps/2-1}
,
\end{aligned}
$$
where the last inequality follows from the definition of $\rho$.
Therefore, if we take $\gamma \in(0, \eps/2)$, then $\sup_R M_\gamma(R) < \infty$, and the set of exceptional times for having an infinite cluster almost surely has a positive Hausdorff dimension.
\qed

Finally, note that if there were exceptional times with two distinct infinite white clusters with positive probability, then there would also be times with the 4-arm event from the origin to infinity. It was shown in \cite{\SchrammSteif} that this does not happen on the triangular lattice, and that there are no exceptional times on $\Z^2$ with three infinite white clusters. However, one can also easily prove the stronger result for $\Z^2$. Recall that~(\ref{e.a4less2}) implies that $\alpha_4(r)^2\, r^2 < O(1)\, r^{-\eps}$ for some $\eps>0$.
 This implies that the expected number of pivotals for the 4-arm event between radius $4$ and $R$ tends to
zero as $R\to\infty$.  (One should also take into account the sites near the outer boundary and
near the inner boundary. Indeed, the total expected number of pivotals is $O(1)\,R^2\,\alpha_4(R)^2$.)
Hence \cite[Theorem 8.1]{\SchrammSteif} says that there are a.s.\ no exceptional times for the 4-arm event even on $\Z^2$.

\section{Scaling limit of the spectral sample}\label{s.spectralscaling}

Given $\eta>0$ let $\mu_\eta$ denote the law of Bernoulli$(1/2)$ site percolation on the
triangular grid $T_\eta$ of mesh $\eta$.
Let $\omega$ denote a sample from $\mu_\eta$.
Given a quad $\Quad\subset\C$, we can consider the event that $\Quad$ is crossed by $\omega$.
To make this precise in the case where $\Quad$ is not adapted to the grid, we
may consider the white and black coloring of the hexagonal grid dual to $T_\eta$, as in Subsection~\ref{ss.gen}. We let $f_\Quad$ denote the $\pm1$ indicator function of the crossing event.
Let $\widehat\mu_\eta^\Quad$ denote the law of the spectral sample of $f_\Quad$, that is,
if $\ev X$ is a collection of subsets of the vertices of $T_\eta$, then
$$
\widehat\mu_\eta^\Quad(\ev X)=\sum_{S\in \ev X} \widehat f_\Quad(S)^2.
$$

Let $d_0$ denote the spherical metric on $\hat\C=\C\cup\{\infty\}$
(with diameter $\pi$).
If $S_1,S_2\subseteq \hat \C$ are closed and nonempty, let
$d_H(S_1,S_2)$ be the Hausdorff distance between $S_1$ and $S_2$
with respect to the underlying metric $d_0$.
If $S\ne\emptyset$, define $d_H(\emptyset,S)=d_H(S,\emptyset):=\pi$,
and set $d_H(\emptyset,\emptyset)=0$.
Then $d_H$ is a metric on the set $\mathfrak S$ of closed subsets of $\hat\C$.
Since $(\mathfrak S\setminus\{\emptyset\},d_H)$ is compact,
the same holds for $(\mathfrak S,d_H)$.
We may consider the probability measure
$\widehat\mu_\eta^\Quad$ as a Borel measure on $(\mathfrak S,d_H)$.

\begin{theorem}\label{t.scalinglimit}
Let $\Quad$ be a piecewise smooth quad in $\C$.
Then the weak limit $\widehat\mu^\Quad:=\lim_{\eta\searrow0}\widehat \mu_\eta^\Quad$ (with respect
to the metric $d_H$) exists.
Moreover, it is conformally invariant, in the sense that
if $\phi$ is conformal in a neighborhood of $\Quad$ and $\Quad':=\phi(\Quad)$,
then $\widehat\mu^{\Quad}=\widehat\mu^{\Quad'}\circ \phi$.
\end{theorem}

As mentioned in the introduction, the existence of the limit follows
from Tsirelson's theory (see \cite[Theorem 3c5]{\TsirelsonStFlour}) and~\cite{\SchrammSmirnovNoise}. Nevertheless,
we believe that our exposition below might be helpful.

\begin{remark}\label{r.Z2scaling}
The proof of
the existence of the limit also works for subsequential
scaling limits of critical bond percolation on $\Z^2$:
if $\eta_j\searrow0$ is a sequence along which bond percolation
on $\eta_j\,\Z^2$ has a limit (in the sense of~\cite{\SchrammSmirnovNoise}, say),
then the corresponding spectral sample measures also have a limit. 
(The existence of such sequences $\{\eta_j\}$ follows from compactness.)
\end{remark}

In the proof of Theorem~\ref{t.scalinglimit}, we will use the
following result.

\begin{proposition}[\cite{\SchrammSmirnovNoise}]\label{p.ss}
Let $\Quad$ be a piecewise smooth quad in $\C$.
Suppose that $\alpha\subset\C$ is a finite union of finite length paths,
and that $\alpha\cap\p \Quad $ is finite.
Then for every $\eps>0$ there is a finite collection of piecewise smooth quads
$\Quad_1,\Quad_2,\dots,\Quad_n\subset\C\setminus\alpha$
and a function $g:\{-1,1\}^n\to\{-1,1\}$ such that 
\begin{equation*}
\label{e.finapprox}
\lim_{\eta\searrow 0}\,
\mu_\eta\Bigl[ f_{\Quad}(\omega)\ne g\bigl( f_{\Quad_1}(\omega),f_{\Quad_2}(\omega),\dots,
f_{\Quad_n}(\omega)\bigr)\Bigr]<\eps\,.
\QED
\end{equation*}
\end{proposition}

Another result from~\cite{\SchrammSmirnovNoise} that we will need is
that for any finite sequence $\Quad_1,\Quad_2,\dots,\Quad_n$ of piecewise smooth quads
in $\C$, the law of the vector $\bigl(f_{\Quad_1}(\omega),f_{\Quad_2}(\omega),\dots,f_{\Quad_n}(\omega)\bigr)$
under $\mu_\eta$ has a limit as $\eta\searrow 0$, and that this limiting
joint law is conformally invariant.

\proofof{Theorem~\ref{t.scalinglimit}}
Let $U\subset\C$ be an open set such that $\p U$ is a disjoint finite union
of smooth simple closed paths and $\p U\cap \p \Quad$ is finite.
In order to establish the existence of the limit $\widehat \mu^\Quad$,
it is clearly enough to show that for every such $U$ the limit
$$
W(\Quad,U):=\lim_{\eta\searrow 0}\widehat\mu^{\Quad}_\eta\bigl(\Spec\subseteq U\bigr)
$$
exists, and for the conformal invariance statement,
it suffices to show that $W\bigl(\Quad',\phi(U)\bigr)=W(\Quad,U)$.

Fix some $\eps>0$ arbitrarily small.
In Proposition~\ref{p.ss}, take $\alpha:=\p U$
and let $\Quad_1,\dots,\Quad_n$ and $g$ be as guaranteed there.
Let $J:=\bigl\{j\in\{1,2,\dots,n\}:\Quad_j\subset U\bigr\}$
and $J':=\{1,\dots,n\}\setminus J$.
Then $\Quad_j\cap U=\emptyset$ when $j\in J'$.
Set $x:=\bigl(f_{\Quad_1}(\omega),\dots,f_{\Quad_n}(\omega)\bigr)\in\{-1,1\}^n$,
and let $x_J$ and $x_{J'}$ denote the restrictions of
$x$ to $J$ and $J'$, respectively.
Then for all $\eta$ sufficiently small $x_{J'}$ is independent
from $\omega_U$ and $x_J$ is determined by $\omega_U$.

Let $ G=G(\omega):=g(x)$,
let $\nu_\eta$ be the law of $x$ under $\mu_\eta$,
and let $\nu:=\lim_{\eta\searrow 0}\nu_\eta$.
By~\eqref{e.SinB}, we have for all $\eta$ sufficiently small
\begin{equation*}
\label{e.Gexp}
W_\eta(G,U):= \sum_{S\subset U} \widehat G(S)^2=\EB{\Eb{G(\omega)\md \omega_U}^2}.
\end{equation*}
We may write $g(x)$ as a sum
$$
g(x)=\sum_{y\in\{-1,1\}^{J}} 1_{x_J=y}\,g_y(x_{J'})
$$
with some functions $g_y:\{-1,1\}^{J'}\to\{-1,1\}$.
Then
$$
W_\eta(G,U) =
\sum_y \nu_\eta(x_J=y)\,\nu_\eta[g_y]^2
\;\;
\underset{\eta\searrow0}{\longrightarrow}
\;\;
\sum_y \nu(x_J=y)\,\nu[g_y]^2
.
$$
(Here, $\nu[g_y]$ denotes the expectation of $g_y$ with respect to
$\nu$, and similarly for $\nu_\eta$.)
Hence, $W(G,U):=\lim_{\eta\searrow0}W_\eta(G,U)$ exists.

By~\eqref{e.specapprox}, we have
$$
\bigl| \widehat\mu_\eta^\Quad(\Spec\subseteq U) - W_\eta(G,U) \bigr|
\le 4\,\mu_\eta(G\ne f_\Quad)^{1/2}\,.
$$
For $\eta$ sufficiently small, the right hand side is smaller than
$4\sqrt\eps$, by our choice of $g$.
Since $W(G,U)=\lim_{\eta\searrow0} W_\eta(G,U)$, we conclude that 
$$
\bigl|\widehat\mu_\eta^\Quad(\Spec\subseteq U)-W(G,U)\bigr|<5\,\sqrt\eps
$$
for all $\eta$ sufficiently small. 
(But we cannot say that
$\lim_{\eta\searrow0}\widehat\mu_\eta^\Quad(\Spec\subseteq U)= W(G,U)$, since
$G$ depends on $\eps$.) This implies
$$
\limsup_{\eta\searrow0}\widehat\mu_\eta^\Quad(\Spec\subseteq U)-
\liminf_{\eta\searrow0}\widehat\mu_\eta^\Quad(\Spec\subseteq U) \le 10\,\sqrt\eps\,.
$$
Since $\eps$ is an arbitrary positive number, this establishes
the existence of the limit $W(\Quad,U)$.
The proof of conformal invariance is similar, and left to the reader.
\QED

We now describe some a.s.\ properties of the limiting law.

\begin{theorem}\label{t.limprop}
If $\Spec$ is a sample from
$\widehat\mu^\Quad$, then a.s.\ $\Spec$ 
is contained in the interior of $\Quad$ and for every
open $U\subset\C$, if $U\cap\Spec\ne\emptyset$,
then $U\cap \Spec$ has Hausdorff dimension $3/4$.
In particular, $\Spec$ is a.s.\ homeomorphic to a Cantor set,
unless it is empty.
\end{theorem}

\proof
It follows from~\eqref{e.quad'spec} from Subsection~\ref{ss.quadconc}
that $\Spec$ is $\widehat\mu^\Quad$-a.s.\ contained in the interior of $\Quad$.
Now fix some open $U$ whose closure is contained in the interior
of $\Quad$.
Fix $\eta>0$ and let $\Spec_\eta$ denote a sample from
$\widehat\mu^\Quad_\eta$.
Let $\lambda_\eta$ denote the counting measure on $\Spec_\eta\cap U$
divided by $\eta^{-2}\,\alpha_4(1,1/\eta)$.
We may consider $\lambda_\eta$ as a random point in the metric space
of Borel measures on $\Quad$ with the Prokhorov metric.
By the estimate~\eqref{e.singleBit}, 
we have $\limsup_{\eta\searrow0}\Eb{\lambda_\eta(U)}<\infty$.
Therefore, the law of $\lambda_\eta$ is tight as $\eta\searrow 0$.
Likewise, the law of the pair $(\Spec_\eta,\lambda_\eta)$ is tight.
Hence, there is a sequence $\eta_j\to 0$ such that the
law of the pair $(\Spec_{\eta_j},\lambda_{\eta_j})$ converges weakly
as $j\to\infty$. Let $(\Spec,\lambda)$ denote a sample from
the weak limit.
Then $\lambda$ is a.s.\ a measure whose support is contained in $\Spec$.

Now let $B\subset U$ be a closed disk.
Let $B'\subset B$ be a concentric open disk with smaller
radius, and let $\delta$
denote the distance from $\partial B'$ to $\partial B$.
Theorem~\ref{t.loc} with
$B'$ and $B$ in place of $U'$ and $U$ implies that $\lambda(B)>0$
a.s.\ on the event $\emptyset\ne \Spec\cap B\subset B'$.
(Note that $\emptyset \ne \Spec\cap B\subset B'$ is
an open condition on $\Spec$, since
it is equivalent to having $\emptyset\ne\Spec\cap B'$
and $\Spec\cap (B\setminus B')=\emptyset$.)
But $B\setminus B'$ can be covered by $O(\delta^{-1})$ disks
of radius $\delta$. By~\eqref{e.SintersectsB} and~\eqref{e.a4},
for each of these radius $\delta$ disks, the probability
that $\Spec$ intersects it is $O(\delta^{5/4+o(1)})$.
Therefore, $\Pb{\Spec\cap B\not\subset B'} = O(\delta^{1/4+o(1)})$.
In particular, we have
$\Pb{\Spec\cap B\ne\emptyset,\,\lambda(B)=0}=o(1)$ as
$\delta\searrow0$; that is,
$\Pb{\Spec\cap B\ne\emptyset,\,\lambda(B)=0}=0$.
By considering a countable collection of disks covering $U$
 it follows that on $U\cap \Spec\ne\emptyset$ we have
$\lambda(U)>0$ a.s.
The correlation estimate~\eqref{e.percmom2} 
and the asymptotics $\alpha_4(r)=r^{-5/4+o(1)}$ from~\eqref{e.a4}
imply that for $\eta>0$ and every $s>-3/4$ we have
$$
\EB{\int_U\int_U \bigl(|x-y|\vee\eta\bigr)^s\,d\lambda_\eta(x)\,d\lambda_\eta(y)}
= O(1)\,.
$$
This implies that a.s.\ 
$$
\int_{U}\int_{U} |x-y|^s\,d\lambda(x)\,d\lambda(y)<\infty\,.
$$
Therefore, Frostman's criterion implies that the Hausdorff dimension
of $\Spec$ is a.s.\ at least $3/4$ on the event $\Spec\cap U\ne\emptyset$.
By Lemma~\ref{l.basic}, the expected number of disks of radius $r$ needed
to cover $\Spec\cap U$ is bounded from above by $r^{3/4+o(1)}$.
Hence, the Hausdorff dimension of $U\cap\Spec$ is a.s.\ at most $3/4$
on the event $U\cap\Spec\ne\emptyset$.
This proves the claim for any fixed $U$. The assertion for every $U$ then
follows by considering a countable basis for the topology (i.e.,
disks having rational radius and centers with rational coordinates).
\QED

\begin{remark}\label{r.weakconv}
It would be interesting to prove the weak convergence of the law
of $(\Spec_\eta,\lambda_\eta)$ as $\eta\searrow 0$.

Note that the proof above shows that for any subsequential scaling limit $(\Spec,\lambda)$ of $(\Spec_{\eta},\lambda_{\eta})$,
the support of the measure $\lambda$ a.s is the whole $\Spec$.
\end{remark}

\proofof{Theorem~\ref{t.speclim}}
Since $[0,R]^2$ is a square, we have $\Eb{f_R}\to 0$ a.s.
Therefore $\Pb{\Spec_{f_R}=\emptyset}\to 0$.
Consequently, the claims follow from Theorems~\ref{t.scalinglimit} and~\ref{t.limprop}.
\QED

\section{Some open problems}\label{s.open}

Here is a list of some questions and open problems:
\begin{enumerate}
\item For any Boolean function $f:\{-1,1\}^n \rightarrow \{-1,1\}$, define its {\bf spectral entropy} $\Ent(f)$ to be \begin{equation*} \Ent(f)= \sum_{S\subseteq \{1,\ldots,n\}} \widehat f(S)^2 \log \frac 1 {\widehat f(S)^2}\,.
\end{equation*}
Friedgut and Kalai conjectured in \cite{\FriedgutKalai} that there is some absolute constant $C>0$ such that for any Boolean function $f$, \begin{equation*} \Ent(f)\leq C \sum_{S\subseteq \{1,\ldots,n\}} \widehat f(S)^2 |S|=C\,\Eb{|\Spec_f|}; \end{equation*} in other words, that the spectral entropy is controlled by the total influence.
 As was pointed out to us by Gil Kalai, it is natural to test the conjecture in the setting of percolation: if $f_R$ is the $\pm1$-indicator function of the left-right crossing in the square $[0,R]^2$, is it true that
$\Ent(f_R) = O( R^2 \alpha_4(R))$?

\item Our paper deals with noise sensitivity of percolation and its applications to dynamical percolation.
One could ask similar questions about the Ising model, for which a natural dynamics is the Glauber dynamics.
For instance, Broman and Steif ask in \cite[Question 1.8]{\BromanSteif} if there exist exceptional times for the Ising model on $\Z^2$ at $\beta=\beta_c$ for which there is an infinite
up-spin cluster. Since  $SLE_3$ (which is supposedly the scaling limit of critical Ising interfaces, see Smirnov's recent breakthrough \cite{\SmirnovICM}) 
does not have double points, there should be very few pivotals, and thus such exceptional times should not exist, but the missing argument is a quasi-multiplicativity property for the probabilities of the alternating 4-arm events in the Ising model. 
Similar questions can be asked for the FK model, Potts models, etc.

\item Prove the weak convergence of the law of $(\Spec_{\eta}, \lambda_{\eta})$; see Remark~\ref{r.weakconv}.
In the paper \cite{\PivotalMeasure}, we prove the weak convergence of the law of 
$(\Piv_{\eta}, \tilde\lambda_{\eta})$, where $\tilde\lambda_{\eta}$
 is the counting measure on the set of pivotals renormalized by $\eta^{-2}\alpha_4(1,1/\eta)$.

\item Prove that $\Piv_R$ and $\Spec_R$ asymptotically have the same ``polar sets''. See Remark~\ref{r.polarset} for a more precise description.

\item Prove that the laws of $\Piv_R$ and $\Spec_R$ are asymptotically mutually singular, or that
their scaling limits are singular. Remark~\ref{r.annpivo} suggests that this should be the case, since
both these sets should be statistically self similar, in some sense. 

\item Do we have $\Eb{|\Spec_{R\Quad}|}=\Eb{|\Piv_{R\Quad}|}\asymp R^2 \alpha_4(R)$ for any quad $\Quad \subset \C$, as $R$ goes to infinity? (See Theorem~\ref{t.quadtight}).

\item In the same fashion, prove that the sharp \concentration as in Theorem~\ref{t.1} still holds for general quads. With our techniques, this would require a uniform control over the domain on the constants involved in Proposition~\ref{pr.SQBWloc}, as well as a statement analogous to Proposition~\ref{pr.verysmall} for the case of general quads.

\item Prove that the main statement in Section~\ref{s.mainlemma} (Proposition~\ref{pr.SQBW}) still holds for non-monotone functions
such as the $\ell$-arm annulus crossing events.
(See Subsection~\ref{subs.int} for an explanation why we needed the monotonicity assumption for the first moment.) If such a generalization was proved, then it would imply in particular that for the triangular lattice
 the set of exceptional times with both infinite black and white clusters has dimension $2/3$ a.s.,
strengthening the last statement in Theorem~\ref{t.dyn}.

For $\ell >1$, there is a further small complication when $\ell$ is odd: a bit can be pivotal for the $\ell$-arm event even without having the exact 4-arm event around its tile.
(That is why we restricted Proposition~\ref{pr.verysmallRad} to the $\ell\in \{1\}\cup 2\N_+$ case.) Resolving this technicality and the
non-monotonicity problem would imply the existence of exceptional times where there are polychromatic three arms from 0 to infinity
(the dimension of this set of exceptional times would then be 1/9).
We cannot prove the existence of such times with the results of the present paper.

\item 
Let us conclude with a computational problem: find any ``efficient'' algorithmic way to sample $\Spec$
in the case of percolation, say (in order, for instance, to make pictures of it), or prove that such
an algorithm does not exist. 

As we have recently learnt from Gil Kalai, the fact that the crossing function itself is computable in polynomial time (in the number of input bits) implies that there is a polynomial time {\it quantum algorithm} to sample $\Spec$ \cite[Theorem 8.4.2]{\BernsteinVazirani}. In fact, the Fourier sampling problem provided the first formal evidence that quantum Turing machines are more powerful than bounded-error probabilistic Turing machines.
\end{enumerate}


\appendix

\section{Appendix: an inequality for multi-arm probabilities}\label{s.appendix}

We prove here an estimate regarding the multi-arm crossing probabilities for annuli in critical bond percolation on $\Z^2$, which is due to Vincent Beffara (private communication) and included here with his permission.

\begin{proposition}\label{p.strong}
Fix $k\in\N_+$
and consider bond percolation on $\Z^2$ with parameter $p=1/2$.
There are constants, $C,\eps>0$,
which may depend on $k$, such that for all $1<r<R$, \begin{equation} \label{e.strong} \alpha_{2k+1}(r,R)\le C\, \alpha_1(r,R)\,\alpha_{2k}(r,R)\,(r/R)^\eps.
\end{equation}
\end{proposition}

The method of proof can be generalized to give a few similar results.
However, new ideas seem to be necessary for the corresponding statement with $k=1/2$. The case $k=1/2$ is of particular significance:
it was proved in~\cite{\SchrammSteif} that it implies the existence of exceptional times for Bernoulli$(1/2)$ bond percolation on $\Z^2$. In the present paper we prove their existence using~\eqref{e.strong} instead.

\proof
For simplicity of notation, we will restrict the proof to the case $k=2$, which is the case we need, but the proof very easily carries over to the general case.
Let $A(r,R)$ denote the annulus which is the closure of $B(0,R)\setminus B(0,r)$.
If we have the 4-arm event in $A(r,R)$, i.e., four crossings of alternating colors between the two boundary components of $A(r,R)$, with white (primal) and black (dual) colors on the tiles given in Subsection~\ref{ss.gen}, then there are at least $4$ interfaces $\gamma_1,\gamma_2,\gamma_3,\gamma_4$ separating these clusters.
These interfaces are simple paths on the grid $\Z^2+(1/4,1/4)$ in $A(r,R)$, and each of them has one point on each boundary component $A(r,R)$.

Let $\gamma=(\gamma_1,\gamma_2,\gamma_3)$ be a triple of $3$ simple paths that can arise as $3$ consecutive interfaces in cyclic order.
Let $\ev A_\gamma$ denote the event that these are actual interfaces between crossing clusters.
Let $S:=S_\gamma$ denote the connected component of $A(r,R)\setminus(\gamma_1\cup\gamma_3)$
that does not contain $\gamma_2$,
and let $\ev B_\gamma$ denote the event that $\ev A_\gamma$ occurs and there are at least two disjoint primal crossings in $S$.
Our first goal is to prove that
\begin{equation}
\label{e.condA}
\Pb{\ev B_\gamma\md \ev A_\gamma}\le O(1)\,\alpha_1(r,R)\,(r/R)^\eps, \end{equation} with some constant $\eps>0$.

Note that on $\ev A_\gamma$, we have in $S$ at least one primal crossing and at least one dual crossing, which are adjacent to $\gamma_1$ and $\gamma_3$. For the sake of definiteness, we will assume that the primal crossing is adjacent to $\gamma_1$ and the dual crossing is adjacent to $\gamma_3$.
(This can be determined from $\gamma$.)
Let $S'=S_\gamma'$ denote the set of edges in $S$ that are not adjacent to $\gamma_1$. Then given $\ev A_\gamma$, we have $\ev B_\gamma$ if and only if there is a primal crossing also in $S'$.
Therefore,~\eqref{e.condA} follows once we show that for every such $\gamma$ the probability that there is a crossing in $S'$ is bounded from above by the right hand side of~\eqref{e.condA}.

We may consider a percolation configuration $\omega$ in the whole plane and also restrict it to $S'$.
Let $\omega_\rho$ denote the restriction of $\omega$ to $B(0,\rho)$, for any $\rho\in [r,R]$.
Write $\rho\leftrightarrow\rho'$
for the event that there is a crossing of $\omega$ between the two boundary component of the annulus $A(\rho,\rho')$, and write $\rho\overset D\leftrightarrow \rho'$ for the existence of such a crossing within some specified set $D$.
We will prove that for some constant $a\in(0,1)$ and every $\rho$ and $\rho'$ satisfying $r\vee 100\le\rho\le \rho'/8\le R/8$ we have \begin{equation} \label{e.step} \Pb{r\overset{S'}{\leftrightarrow}\rho'
\md \omega_\rho,\,r \leftrightarrow R }\le a\,.
\end{equation}
Using induction, this implies~\eqref{e.condA} with $\eps=\log_8 (1/a)$.

The interface $\gamma_1$ crosses the annulus $A(3\,\rho,4\,\rho)$ one or more times.
Let $w$ denote the winding number around $0$ of one of these crossings, that is, the signed change of the argument along the crossing divided by $2\,\pi$.
Suppose that $\beta$ is a simple path in $A(3\,\rho,4\,\rho)$ with one endpoint on each boundary component of the annulus and let $w_\beta$ denote the winding number of $\beta$. If $\beta\cap\gamma_1=\emptyset$, then we may adjoin to $\beta\cup\gamma_1$ two arcs on the boundary components of the annulus to form a simple closed curve which has winding number in $\{0,\pm1\}$. Therefore, we see that $|w-w_\beta|>3$ implies $\beta\cap\gamma_1\ne\emptyset$.

\begin{figure}[htbp]
\SetLabels
(.65*.33)$2\rho$\\
(.75*.23)$3\rho$\\
(.85*.13)$4\rho$\\
(.95*.03)$5\rho$\\
\endSetLabels
\centerline{
\AffixLabels{\epsfysize=2.5in \epsffile{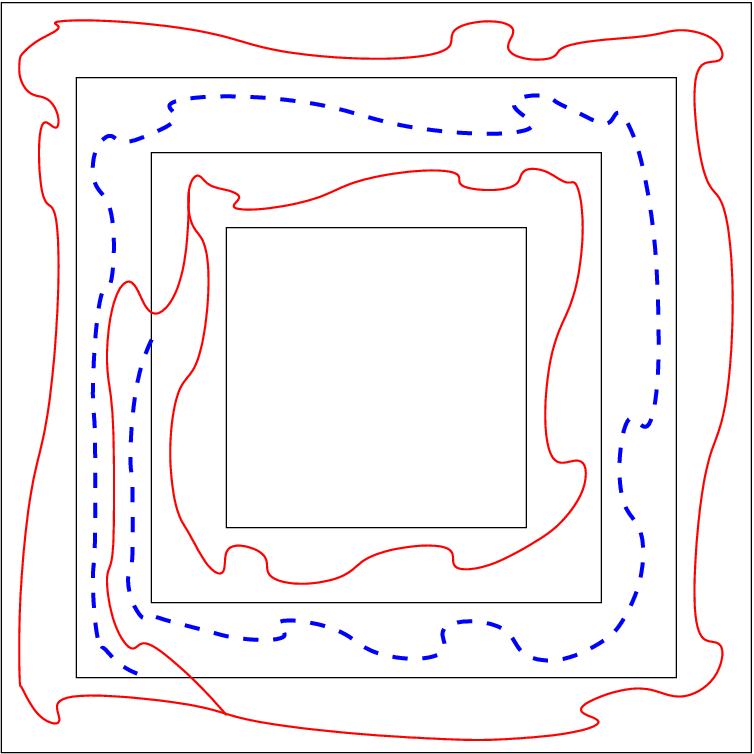}}
}
\begin{caption}{The event $\ev D_j$ with $j=1$.
\label{f.winding}
}
\end{caption}
\end{figure}

Let $\ev D_j$ denote the event that there is a dual crossing in $\omega$ of $A(3\,\rho,4\,\rho)$ with winding number in the range $[j-1/2,j+1/2]$, and there are primal circuits in $A(2\,\rho,3\,\rho)$ and in $A(4\,\rho,5\,\rho)$, each of them separating the two boundary component of its annulus, and these primal circuits are connected to each other in $\omega$; see Figure~\ref{f.winding}. By the RSW theorem, there is constant $\delta>0$ such that $\Pb{\ev D_{\pm 10}}\ge\delta$.
We claim that
\begin{equation}
\label{e.condD}
\Pb{\ev D_{\pm 10}\md r\leftrightarrow R,\,\omega_\rho}\ge\delta\,.
\end{equation}
If we condition on $\ev D_{10}$,
and we further condition on the outermost primal circuit $\alpha_0$ in $A(2\,\rho,3\,\rho)$ and on the innermost primal circuit $\alpha_1$ in $A(4\,\rho,5\,\rho)$, then the configuration inside $\alpha_0$ and the configuration outside $\alpha_1$ remains unbiased, and if additionally $\alpha_0$ is connected to the inner boundary component of $A(r,R)$ and $\alpha_1$ is connected to the outer boundary component of $A(r,R)$, then we also have $r\leftrightarrow R$. This implies $\Pb{r\leftrightarrow R\md \omega_\rho,\,\ev D_{10}}\ge \Pb{r\leftrightarrow R \md \omega_\rho}$.
The same holds for $\ev D_{-10}$, and since $\ev D_{\pm 10}$ is independent of $\omega_\rho$, the inequality~\eqref{e.condD} easily follows.

Now note that if $\ev D_{j}$ holds, then every primal crossing of $A(3\,\rho,4\,\rho)$ in $\omega$ is with winding number in the range $[j-4,j+4]$.
Hence, if $|j-w|>7$, then $\ev D_j\cap\{3\,\rho\overset {S'}\leftrightarrow 4\,\rho\}=\emptyset$.
Consequently, at least one of the two events $\ev D_{\pm10}$ is disjoint from $\{\rho\overset {S'}\leftrightarrow 4\,\rho\}$. This gives~\eqref{e.step} with $a:=1-\delta$.
As we have argued before,~\eqref{e.condA} follows.

The $5$-arm crossing event is certainly contained in $\bigcup_\gamma \ev B_\gamma$, where the union ranges over all $\gamma$ as above.
Hence,~\eqref{e.condA} gives
\begin{equation}
\label{e.ad}
\alpha_5(r,R)\le O(1)\,(r/R)^\eps\,\alpha_1(r,R)\,\EB{ \sum_\gamma 1_{\ev A_\gamma}}.
\end{equation}
If $X$ is the number of interfaces crossing the annulus $A(r,R)$ (which is necessarily even), then $\sum_\gamma 1_{\ev A_\gamma}$ is not more than $X^3\,1_{X\ge 4}$.
Since for all $j\in\N$ we have
$\Pb{X\ge j}\le O(1)\,(r/R)^{j\,\eps_0}$ with some constant $\eps_0>0$ (by RSW and BK) and $\Pb{X\ge 4}\ge (r/R)^{\eps_1}/O(1)$ for some $\eps_1\in(0,\infty)$, we have $\Eb{X^3\,1_{X\ge 4}}\le O(1)\,\Pb{X\ge 4}$ when $R>2\,r$, say.
Therefore,
$$
\EB{ \sum_\gamma 1_{\ev A_\gamma}} \le
\Eb{X^3\,1_{X\ge 4}}\le O(1)\,\Pb{X\ge 4}=O(1)\,\alpha_4(r,R)\,.
$$
When combined with~\eqref{e.ad}, this proves the proposition in the case $k=2$.
The general case is similarly obtained.
\QED

\bibliographystyle{halpha}
\addcontentsline{toc}{section}{Bibliography}
\bibliography{../mr,../prep,../notmr}

\end{document}